\documentclass[reqno]{amsart}
\usepackage[pxpazo, numberbasedon={subsection}, addressfirstpage=false]{zzzams}

\DeclareMathOperator{\St}{St}

\DeclareMathOperator{\Rep}{Rep}

\DeclareMathOperator{\Perv}{Perv}

\DeclareMathOperator{\LInd}{LInd}
\DeclareMathOperator{\nInd}{nInd}

\DeclareMathOperator{\Av}{Av}

\DeclareMathOperator{\triv}{triv}
\newcommand{\ceil}[1]{{\lceil #1 \rceil}}
\newcommand{\floor}[1]{{\lfloor #1 \rfloor}}
\newcommand{\ubar}[1]{\underline{#1}}
\renewcommand{\mod}{\operatorname{mod}}
\renewcommand{\supp}{\operatorname{supp}}

\title[Comparing complex and $p$-adic $\GL_n$]{On a geometric comparison of representations of complex and $p$-adic $\GL_n$}

\author{Taiwang Deng}
\address{Beijing Institute of Mathematical Sciences and Applications, Huairou District, Beijing 101408, China}
\email{dengtaiw@bimsa.cn}
\thanks{TD was supported by NSFC grant No. 12401013 and BJNSF grant No.1244042}

\author{Chang Huang}
\address{Department of Mathematics \\ Tsinghua University, Haidian District, Beijing 100084, China}
\email{hc21@mails.tsinghua.edu.cn}

\author{Bin Xu}
\address{Yau Mathematical Sciences Center and Department of Mathematics \\  Tsinghua University, Haidian District, Beijing 100084, China}
\email{binxu@tsinghua.edu.cn}
\thanks{BX was supported by Ministry of Science and Technology of China grant No. 2021YFA1000700}

\author{Qixian Zhao}
\address{Yau Mathematical Sciences Center \\ Tsinghua University, Haidian District, Beijing 100084, China}
\email{zhao{\_}qixian@tsinghua.edu.cn}

\date{March 17, 2026}

\begin{document}

\begin{abstract}
	In this paper, we use geometric methods to study the relations between admissible representations of $\GL_n(\BC)$ and unramified representations of $\GL_m(\BQ_p)$. We show that the geometric relationship between Langlands parameter spaces of $\GL_n(\BC)$ and $\GL_m(\BQ_p)$ constructed by the first named author is compatible with the functor recently defined algebraically by Chan-Wong. We then show that the said relationship intertwines translation functors on representations of $\GL_n(\BC)$ and partial Bernstein-Zelevinskii derivatives on representations of $\GL_m(\BQ_p)$, providing purely geometric counterparts to some results of Chan-Wong. In the sequels, the techniques of this work will be extended to real and $p$-adic classical groups and used to study their Arthur packets. 
\end{abstract}

\maketitle

\tableofcontents

\section{Introduction}

In a recent work of Chan-Wong \cite{Chan-Wong:GLn}, the authors construct algebraically an exact functor $\Gamma_{n,m}$ from representations of $\GL_n(\BC)$ to those of $\GL_m(\BQ_p)$ (this is recalled briefly in \textsection \ref{subsec:alg-comparison}). Further, they showed that their functor plays well with certein translation functors on $\GL_n(\BC)$ and Bernstein-Zelevinskii derivatives on $\GL_m(\BQ_p)$. In other words, the following diagram commutes 
\begin{equation}\label{diag:rep-square}
	\begin{tikzcd}
		\Rep^I(\GL_m(\BQ_p)) \ar[d, "\text{BZ derivative}"'] 
		& \Rep(\GL_n(\BC)) \ar[l, "\Gamma_{n,m}"'] \ar[d, "\text{translation}"]\\
		\Rep^I(\GL_{m'}(\BQ_p))
		& \Rep(\GL_n(\BC)) \ar[l, "\Gamma_{n,m'}"]
	\end{tikzcd}.
\end{equation}
Here $\Rep(\GL_n(\BC))$ stands for the category of Harish-Chandra modules attached to the pair $(\fgl_n(\BC), \bU(n))$, and $\Rep^I(\GL_m(\BQ_p))$ stands for the Iwahori component of the category of smooth representations of $\GL_m(\BQ_p)$.

The upshot of this paper is to present a purely geometric construction on the Langlands dual side that overlaps with the above result on the level of Grothendieck groups.

Let us explain what this means. Recall from the work of Adams-Barbasch-Vogan \cite{ABV} that the Langlands correspondence for a real group $\check G_\BR$ at an infinitesimal character $\Lambda_\BR$ can be roughly phrased as a perfect pairing between Grothendieck groups
\begin{equation*}
	K \Rep(\check G_\BR)_{\Lambda_\BR} \times K \Perv( \cY(\Lambda_\BR), G_{\Lambda_\BR}) \aro \BZ,
\end{equation*}
see \textsection \ref{subsec:LLC-GLnC}. Here $\cY(\Lambda_\BR)$ is the so called \textit{ABV space}, a smooth projective variety over $\BC$ equipped with an action of the endoscopic group $G_{\Lambda_\BR}$ of $\Lambda_\BR$, so that the set of $G_{\Lambda_\BR}$-orbits parameterize conjugacy classes of Langlands paremeters of representations of $\check G_\BR$ with infinitesimal character $\Lambda_\BR$. The category $\Perv( \cY(\Lambda_\BR), G_{\Lambda_\BR})$ is the category of $G_{\Lambda_\BR}$-equivariant perverse sheaves on $\cY(\Lambda_\BR)$, and $K$ denotes the Grothendieck group. This rough statement is in fact true on the nose for $\check G_\BR = \GL_n(\BC)$. This pairing is compatible with the translation functors of a certain type in the sense that 
they are adjoint to pullback-pushforward functors on sheaves on the ABV spaces, see \textsection\ref{subsec:translation}.

In \cite{Vogan:p-adic-Langlands} Vogan made a similar conjecture on the $p$-adic side. Roughtly speaking, the Langlands correspondence for a $p$-adic group $\check H$ at an infinitesimal character $\Lambda_p$ should be phrased as a perfect pairing
\begin{equation*}
	K \Rep(\check H)_{\Lambda_p} \times K \Perv(E_{\Lambda_p}, H_{\Lambda_p}) \aro \BZ,
\end{equation*}
see \textsection \ref{subsec:LLC/Qp}, where $E_{\Lambda_p}$ is the so called \textit{Vogan variety}, a conical affine variety playing a similar role as the ABV space. It is equipped with an action of $H_{\Lambda_p}$ with finitely many orbits, where $H_{\Lambda_p}$ is the centralizer of $\Lambda_p$ in the Langlands dual group of $\check H$. In the case where $\check H = \GL_m(\BQ_p)$, this conjectural statement was first formulated by Zelevinskii \cite{Zelevinskii} and proven by Ariki \cite{Ariki} and Chriss-Ginzburg \cite{Chriss-Ginzburg}. By the recent work of the first named author \cite{Deng:BZ}, certain partial Bernstein-Zelevinskii derivatives on $\Rep(\GL_m(\BQ_p))$ are adjoint to Lusztig's inductions on sheaves on $E_{\Lambda_p}$ (this is recalled in \textsection \ref{subsec:Lusztig-ind}).

Therefore, under the Langlands correspondences, the adjoints of the two vertical arrows in (\ref{diag:rep-square}) have explicit geometric descriptions. Finally, the first named author  constructed in \cite{Deng:symm} a geometric comparison between the Langlands parameter spaces $\cY(\Lambda_\BR)$ for $\GL_n(\BC)$ and $E_{\Lambda_p}$ for $\GL_m(\BQ_p)$. More precisely, there exists an open immersion of quotient stacks $[E_{\Lambda_p}/(\GL_m)_{\Lambda_p}] \hookleftarrow [\cY(\Lambda_\BR)/ (\GL_n(\BC)_\BC)_{\Lambda_\BR}]$ whenever the pairs $(n, \Lambda_\BR)$ and $(m, \Lambda_p)$ satisfy certain compatibility conditions (\textsection \ref{subsec:geom-comparison}). In particular, we have a pullback functor from the equivariant derived category on $E_{\Lambda_p}$ to the equivariant derived category of $\cY(\Lambda_\BR)$. We thus obtain a diagram
\begin{equation}\label{diag:param-square}
	\begin{tikzcd}
		D^b( E_{\Lambda_p}, \GL_{m,\Lambda_p}) \ar[r,"\text{pullback}"] 
		& D^b( \cY(\Lambda_\BR) , \GL_{n,\Lambda_\BR}) \\
		D^b( E_{\Lambda_p}', \GL_{m',\Lambda_p})  \ar[u, "\text{Lusztig induction}"] \ar[r,"\text{pullback}"]
		& D^b( \cY(\Lambda_\BR') , \GL_{n,\Lambda_\BR'}) \ar[u, "\text{push-pull}"']
	\end{tikzcd}.
\end{equation}
The goal of this paper is two fold. First, we show the two horizontal arrows in (\ref{diag:param-square}) are dual to the horizontal arrows in (\ref{diag:rep-square}) on the level of Grothendieck groups under Langlands correspondence. Second, we show the diagram (\ref{diag:param-square}) commutes. See \S \ref{subsec:main-results} below for more precise statements. 

The geometric nature of our construction allows us to compare finer structures on the parameter spaces. In particular, the pullback functor from $E_{\Lambda_p}$ to $\cY(\Lambda_\BR)$ preserves microlocal data of perverse sheaves in the sense of \cite[Proposition 20.2]{ABV}. On the real side, such data encodes information of Arthur packets (see \cite{ABV} and \cite{Adams-Arancibia-Mezo}), and a similar statement has been conjectured for $p$-adic groups \cite{Cunningham-etal:microlocal}. For general linear groups this observation does not have new representation theoretic implications since Arthur packets for general linear groups are rather simple. For real unitary groups and $p$-adic symplectic and orthogonal groups, however, our approach already allows us to say something interesting, see \textsection \ref{subsec:U-groups} below. In fact this is one of our main motivations for the current project.

The idea of relating real and $p$-adic groups by the geometry of their Langlands parameter spaces is not new. Earlier works in this direction include Lusztig-Zelevinskii \cite{Zelevinskii:RvsQp}, Ciubotaru-Trapa \cite{Ciubotaru-Trapa:duality}, and Barchini-Trapa \cite{Barchini-Trapa}. In particular, Barchini-Trapa constructed, for each reductive group $G$ of classical or certain exceptional types, a map from a Vogan variety attached to $G(\BQ_p)$ to an ABV space attached to $G(\BR)$. Our construction goes in the other direction, where the ABV space of a real group is mapped into the Vogan variety of a $p$-adic group of larger rank. It is worth noting that our map is an open immersion of stacks, which in particular leads to a bijection on the sets of local systems on a given orbit in the ABV space and those on the corresponding orbit in the Vogan variety. In representation theoretic terms, this means we have a bijection between the corresponding L-packets, a property that sometime fails in the approach of Barchini-Trapa.

\subsection{Main results}\label{subsec:main-results}

Let us state the main results in simplified form. 

Fix $m$ and let $\Lambda_p$ be an unramified infinitesimal character of $\check H = \GL_m(\BQ_p)$. Up to conjugacy, it is completely determined by multiplicities of each eigenvalue $p^a$ of $\Lambda_p(\ff)$, where $\ff$ is a fixed lift of the arithmetic Frobenius. Write $\varphi$ for the function that sends $a \in \BC$ to the multiplicity of the eigenvalue $p^a$ of $\Lambda_p(\ff)$. For the sake of the introduction, assume $\varphi$ is supported on $\BZ$. In other words, the eigenvalues of $\Lambda_p(\ff)$ are of the form $p^a$, $a \in \BZ$. Assume moreover that there exists a number $r' \in \BZ$ so that $\varphi$ is weakly increasing on $(-\infty, r'+1]$ and weakly decreasing on the interval $[r', \infty)$ (see Assumption \ref{assump:r}\footnote{The number $r'$ here is equal to the number $\floor r$ in \ref{assump:r}.}). Let $n = \varphi(r')$. Then there exists an integral infinitesimal character $\Lambda_\BR$ of $\GL_n(\BC)$ and an open immersion
\begin{equation*}
	\zeta: [\cY(\Lambda_\BR)/ \GL_n(\BC)_\BC] \injects [E_{\Lambda_p} / (\GL_m)_{\Lambda_p}]
\end{equation*}
see \cite{Deng:symm} or Corollary \ref{cor:zeta}. Note that if $\Lambda_\BR = (\lambda_L, \lambda_R) \in \BC^n \oplus \BC^n$ in standard notations (see \ref{cls:RT-notations-GLn}), the parameter space $[\cY(\Lambda_\BR)/ \GL_n(\BC)_\BC]$ can be written as $[P_{\lambda_L} \backslash \GL_n / P_{\lambda_R}]$ where $P_{\lambda_L}$ (resp. $P_{\lambda_R}$) is the upper-triangular parabolic subgroup of $\GL_n$ determined by simple roots orthogonal to $\lambda_L$ (resp. $\lambda_R$). In particular, we have a pullback functor
\begin{equation*}
	\zeta^* : D^b(E_{\Lambda_p}, (\GL_m)_{\Lambda_p}) \aro D^b(\GL_n, P_{\lambda_L} \times P_{\lambda_R}) .
\end{equation*}

Let us first state our second result. Suppose $\varphi$ can be decomposed as $\varphi = \psi + c[k]$, where $c \ge 0$ and $k > r'+1$ are integers, $[k]: \BC \to \BN$ is the characteristic function at $k$, and $\psi: \BC \to \BN$ is a function constructed similar to $\varphi$ for some infinitesimal character $\Lambda_p'$ of $\GL_{m-c}(\BQ_p)$ so that $\psi$ is also weakly increasing on $(-\infty, r'+1]$ and weakly decreasing on $[r',\infty)$. Then we have another map
\begin{equation*}
	\zeta': [P_{\lambda_L'} \backslash \GL_n / P_{\lambda_R'}] = [\cY(\Lambda_\BR')/ \GL_n(\BC)_\BC] \injects [E_{\Lambda_p'} / (\GL_{m-c})_{\Lambda_p'}]
\end{equation*}
for some other real infinitesimal character $\Lambda_\BR' = (\lambda_L', \lambda_R')$ of $\GL_n(\BC)$ and some upper-triangular parabolics $P_{\lambda_L'}$, $P_{\lambda_R'}$, and hence another pullback functor
\begin{equation*}
	(\zeta')^* : D^b(E_{\Lambda_p'}, (\GL_{m-c})_{\Lambda_p'}) \aro D^b(\GL_n, P_{\lambda_L'} \times P_{\lambda_R'}).
\end{equation*}
On the $p$-adic side, Lusztig has defined an induction functor on sheaves (see \ref{cls:Lind} or \cite{Lusztig:quiver})
\begin{equation*}
	\LInd_{\psi}^\varphi\footnotemark: D^b(E_{\Lambda_p'}, (\GL_{m-c})_{\Lambda_p'}) \aro D^b(E_{\Lambda_p}, (\GL_m)_{\Lambda_p}).
\end{equation*}
\footnotetext{The functor $\LInd_\psi^\varphi$ here is the same as $\LInd_{\psi, c[k]}^\varphi \circ \operatorname{Infl}_1^{G_{c[k]}}$ in the body of the paper, where $\LInd_{\psi, c[k]}^\varphi$ is defined in \ref{cls:Lind} and $\operatorname{Infl}$ is the inflation functor.}%
On the real side, let $Q_l = P_{\lambda_L} \cap P_{\lambda_L'}$ and $Q_r = P_{\lambda_R} \cap P_{\lambda_R'}$\footnote{Because of an unfortunate conventional reason, the group $Q_l$ (resp. $Q_r$) here is in fact denoted by $Q_r$ (resp. $Q_l$) in the body of the paper.}. Then we have natural maps
\begin{equation*}
	[P_{\lambda_L'} \backslash \GL_n / P_{\lambda_R'}] \laro [Q_l \backslash \GL_n / Q_r] \aro [P_{\lambda_L} \backslash \GL_n / P_{\lambda_R}]
\end{equation*}
and hence a pullback-pushforward functor
\begin{equation*}
	\bI_{\Lambda_\BR'}^{\Lambda_\BR}: D^b(\GL_n, P_{\lambda_L'} \times P_{\lambda_R'}) \aro D^b(\GL_n, P_{\lambda_L} \times P_{\lambda_R}).
\end{equation*}
Some important special cases and examples of such $\varphi$, $\psi$, $\Lambda_\BR$ and $\Lambda_\BR'$ are discussed in \ref{cls:reg-to-sing}-\ref{cls:general_ex}.

\begin{theorem}\label{thm:trans-der}
	The following diagram commutes
	\begin{equation*}
		\begin{tikzcd}
			D^b(E_{\Lambda_p}, (\GL_m)_{\Lambda_p}) \ar[r, "\zeta^*"]
			& D^b(\GL_n, P_{\lambda_L} \times P_{\lambda_R})
			\\
			D^b(E_{\Lambda_p'}, (\GL_{m-c})_{\Lambda_p'}) \ar[r, "(\zeta')^*"] \ar[u, "\LInd_{\psi}^\varphi"]
			&  D^b(\GL_n, P_{\lambda_L'} \times P_{\lambda_R'}) \ar[u, "\bI_{\Lambda_\BR'}^{\Lambda_\BR}"']
		\end{tikzcd}.
	\end{equation*}
	By taking adjoints, we obtain the following commutative diagram 
	\begin{equation*}
		\begin{tikzcd}
			K\Rep_{\Lambda_p}(\GL_m(\BQ_p)) \ar[d, "{}^kD"']
			& K\Rep_{\Lambda_\BR}(\GL_n(\BC)) \ar[l]  \ar[d, "T_{\Lambda_\BR}^{\Lambda_\BR'}"]
			\\
			K\Rep_{\Lambda_p'}(\GL_{m-c}(\BQ_p)) 
			& K\Rep_{\Lambda_\BR'}(\GL_n(\BC)) \ar[l]
		\end{tikzcd}
	\end{equation*}
	where ${}^kD$ is the partial Bernstein-Zelevinskii derivative restricted to $K\Rep_{\Lambda_p}(\GL_m(\BQ_p))$.
\end{theorem}

This appears as Corollary \ref{cor:compare-functors}. There is an analogous statement in the case where $\varphi = c[k] + \psi'$ with $k < r$, see \S \ref{subsec:left-to-right}. Although our requirements on $\Lambda_p$ and $\Lambda_p'$ seem rather strict on first sight, they are met in the case we are most interested in, see \S \ref{subsec:U-groups}.

We now turn to the first result. Instead of fixing $m$, we fix $n$ instead, and let 
\begin{equation*}
	\Lambda_\BR = (\lambda_L, \lambda_R) = (\lambda_{L,1},\ldots, \lambda_{L,n}, \lambda_{R,1},\ldots , \lambda_{R,n}) \in \BC^n \oplus \BC^n
\end{equation*}
be an integral dominant infinitesimal character for $\GL_n(\BC)$ (viewed as a real group) so that $\min_i\{\operatorname{Re} \lambda_{L,i}\} > \max_i \{\operatorname{Re} \lambda_{R,i}\} +1$ (any integral dominant $\Lambda_\BR$ can be transformed into this form by adding a twist by the determinant character). Write
\begin{align*}
	m &:= \sum_i (\lambda_{L,i} - \lambda_{R,i})\\ 
	\bm &:= \sum_i \big[ \lambda_{R,i} + \tfrac12, \lambda_{L,i} - \tfrac12 \big].
\end{align*}
Here each $\big[ \lambda_{R,i} + \tfrac12, \lambda_{L,i} - \tfrac12 \big]$ is a \textit{segment} in the sense of Zelevinskii \cite{Zelevinskii}, that is a closed interval in a $\BZ$-coset in $\BC$ from $\lambda_{R,i} + \tfrac12$ to $\lambda_{L,i} - \tfrac12$, and $\bm$ is a formal sum of segments, called a \textit{multisegment}. Each multisegment corresponds to an orbit in the Vogan variety of $\GL_m(\BQ_p)$. Let $\varphi: \BC \to \BN$ be the function that records the multiplicity of each number appearing in $\bm$. Then $\varphi$ comes from an infinitesimal character $\Lambda_p$, it satisfies the increasing-decreasing assumption for some number $r'$ in the sense explained above, and the fixed number $n$ equals $\varphi(n)$ (this function $\varphi$ is supported inside a $\BZ$-coset in $\BC$ which may not be equal to $\BZ$, but this affects nothing). 

In this setup, the exact functor $\Gamma_{n,m}$ constructed by Chan-Wong \cite{Chan-Wong:GLn} restricts to 
\begin{equation*}
	\Gamma_{n,m}: K \Rep(\GL_n(\BC))_{\Lambda_\BR} \injects K \Rep^I(\GL_m(\BQ_p))_{\Lambda_p}.
\end{equation*}
On the other hand, we have the pullback functor $\zeta^*$. Our second result is:

\begin{theorem}
	Consider the perfect pairings
	\begin{equation}
		K \Rep (\GL_n(\BC))_{\Lambda_\BR} \times K D^b(\GL_n, P_{\lambda_L} \times P_{\lambda_R}) \aro \BZ \tag{\ref{eqn:LLC-GLn-pairing}}
	\end{equation}
	\begin{equation}
		K \Rep(\GL_m(\BQ_p))_{\Lambda_p} \times K D^b(E_{\Lambda_p}, (\GL_m)_{\Lambda_p}) \aro \BZ \tag{Proposition \ref{prop:LLC-GLmQp}}
	\end{equation}
	given by the local Langlands correspondences. Up to a normalization, the functors $\Gamma_{n,m}$ and $\zeta^*$ are adjoint to each other, that is
	\begin{equation*}
		\langle \Gamma_{n,m} M, \cF\rangle =  \langle M,  \zeta^* \cF \rangle
	\end{equation*}
	for any $M \in K \Rep(\GL_n(\BC))_{\Lambda_\BR}$ and any $\cF \in K D^b(E_{\Lambda_p}, (\GL_m)_{\Lambda_p})$.
\end{theorem}

This appears as Theorem \ref{thm:alg-vs-geom}. By ``normalization'', we mean the map $\zeta$ should in fact be replaced by $\zeta \circ (\tau \times \tau)$\footnote{In Theorem \ref{thm:alg-vs-geom}, the map $\zeta \circ (\tau \times \tau)$ is denoted by $\imath$.}, where $\tau \times \tau : [P_{\lambda_L} \backslash \GL_n / P_{\lambda_R}] \bij [P_{\tau(\lambda_L)} \backslash \GL_n / P_{\tau(\lambda_R)}]$ is an isomorphism induced by the involution on the Cartan of $\GL_n$ given by $-w_0$.

\subsection{Motivation: unitary groups}\label{subsec:U-groups}

We outline a conjectural relation between certain Arthur-Vogan packets of real unitary groups and $p$-adic orthogonal/symplectic groups. The proof will be contained in the sequel of this paper \cite{DHXZ:U}.

Let us describe the packet on the real side. For a positive integer $n$, let $\check G_\BR$ be the unitary group $\bU(\frac{n}{2}, \frac{n}{2})$ if $n$ is even or $\bU(\frac{n-1}{2}, \frac{n+1}{2})$ if $n$ is odd. Via base change, we identify the set $\Psi(\check G_\BR)$ of Arthur parameters of $\check G_\BR$ with the subset of conjugate self-dual parameters of parity $(-1)^{n}$ in the set $\Psi(\GL_n(\mathbb{C}))$ of Arthur parameters of $\GL_n(\BC)$. For $\psi^\BR \in \Psi(\check G_\BR)$, its base change $BC(\psi^\BR)$ can be viewed as an $n$-dimensional continuous representation of $\BC^\times \times \SL_2(\BC)$. Suppose it has good parity, i.e. it takes the form
\begin{equation*}
	BC(\psi^{\mathbb{R}}) = \bigoplus_{i = 1}^{r} \, (z/\bar{z})^{\frac{k_i}{2}} \boxtimes S_{m_i}
\end{equation*}
where $k_i, m_i \in \mathbb{Z}$ and $n \equiv k_i + m_i$ mod $2$, $z$ denotes an element of $\BC^\times$, and $S_{m_i}$ denotes the $m_i$-dimensional irreducible representation of $\SL_2(\BC)$. By twisting $BC(\psi^{\mathbb{R}})$ with the character $(z/\bar{z})^{l}$ of $\mathbb{C}^{\times}$, we may further assume $k_i > m_i - 1$. Define Arthur-Vogan packet to be
\begin{equation*}
	\Pi_{\psi^\BR}^{\text{A-V}} := \bigsqcup_{p + q = n} \Pi_{\psi^\BR}^{\bU(p,q)}
\end{equation*}
where $\Pi_{\psi^\BR}^{\bU(p,q)}$ is the Arthur packet of $\bU(p,q)$ corresponding to the parameter $\psi^\BR$. 

On the $p$-adic side, let $N = \sum_{i = 1}^{r} (k_i + 1)m_i$, and let 
\begin{equation*}
	\check H = 
	\begin{cases}
		\Sp(N-1) & \text{ if $n$ is odd; } \\
		\SO(N+1) & \text{ if $n$ is even }
	\end{cases}
\end{equation*}
(note that $N \equiv n$ mod $2$). Through the standard representation of $H$ ($H$ is the Langlands dual group of $\check H$), we identify the set $\Psi(\check H)$ of Arthur parameters of $\check H$ with the subset of self-dual parameters of orthogonal/symplectic type in the set $\Psi(\GL_N(\BQ_p))$ of Arthur parameters of $\GL_N(\BQ_p)$. Consider the parameter
\begin{equation*}
	\psi^{p} = \bigoplus_{i = 1}^{r} \mathbf{1}_{W_{\BQ_p}} \boxtimes S_{k_i + 1} \boxtimes S_{m_i}
\end{equation*}
where $\mathbf{1}_{W_{\mathbb{Q}_p}}$ is the trivial representation of the Weil group $W_{\BQ_p}$. Then $\psi^{p} \in \Psi(\check H)$ and has good parity. Define the corresponding Arthur-Vogan packet to be
\begin{equation*}
	\Pi_{\psi^{p}}^{\text{A-V}} = 
	\begin{cases} 
		\Pi_{\psi^p}^{\check H} & \text{ if $n$ is odd; } \\
		\Pi_{\psi^p}^{\check H} \sqcup \Pi_{\psi^p}^{\check H_{an}} & \text{ if $n$ is even. }
	\end{cases} 
\end{equation*}
Here $\check H_{an}$ is the non-split inner form of $\check H$.

\begin{conjecture}\label{conj: comparison of A-packets}
	When $n$ is even, there is a bijection 
	\begin{equation*}
		\Pi_{\psi^\BR}^{\text{A-V}} \bijects \Pi_{\psi^p}^{\text{A-V}}.
	\end{equation*}
	When $n$ is odd, there is a bijection
	\begin{equation*}
		\bigsqcup_{\substack{%
				p + q = n\\%
				p \equiv \frac{n-1}{2} \, {\rm mod} \, 2%
			}} 
		\Pi_{\psi^\BR}^{\bU(p,q)} \bijects \Pi_{\psi^{p}}^{\text{A-V}}.
	\end{equation*}
\end{conjecture}

The bijection can be described very explicitly. When the infinitesimal character of $\psi^\BR$ is regular, $\Pi_{\psi^\BR}^{\bU(p,q)}$ is an Adams-Johnson packet \cite{Adams-Johnson,Arancibia-Moeglin-Renard}, which can be computed explicitly by cohomological induction in the good range \cite{Knapp-Vogan:induction}. In this case, Moeglin \cite{Moeglin:A-classical} has also computed $\Pi_{\psi^p}^{\text{A-V}}$ explicitly. Therefore the conjecture can be verified directly in the regular case. When the infinitesimal character of $\psi^\BR$ is singular, Moeglin-Renard \cite{Moeglin-Renard:unitary} showed that $\Pi_{\psi^\BR}^{\bU(p,q)}$ can be obtained from a regular packet by applying translation functors. In this case, Moeglin \cite{Moeglin:A-classical} also showed that $\Pi_{\psi^p}^{\text{A-V}}$ can be obtained from the regular case by computing derivates, which are operations analogous to the Bernstein-Zelevinskii derivatives for classical groups. In the sequel to this paper \cite{DHXZ:U} we will prove this conjecture by relating translation functors with derivative through the geometry of Langlands parameters as in the current work.

\subsection{Outline of the paper}

The preliminary section \textsection\ref{sec:prelim} contains an overview of Langlands correspondence for real and $p$-adic groups and their specializations to $\GL_n(\BC)$ and $\GL_m(\BQ_p)$, as well as their relations with translation functors and partial BZ derivatives. In \textsection \ref{sec:compare-spaces} we review the geometric and algebraic comparison theorems of \cite{Deng:symm} and \cite{Chan-Wong:GLn} that relate representations of $\GL_n(\BC)$ and $\GL_m(\BQ_p)$. The compatibility of the two approaches is proven in \textsection \ref{subsec:alg-vs-geom}. Then in \textsection \ref{sec:compare-functors} we prove the commutativity of the diagram (\ref{diag:param-square}). 

\subsection{General notations}\label{subsec:gen-notn}

Our notations related to equivariant perverse sheaves generally follow \cite{Achar:book}. Suppose $X$ is a variety with an action of an algebraic group $G$. We write $[X/G]$ for the quotient stack, $\Perv(X,G)$ for the category of $G$-equivariant perverse sheaves on $X$, and $D^b(X,G)$ for the bounded equivariant derived category. The Grothendieck group of $D^b(X,G)$ is naturally identified with the Grothendieck group of $\Perv(X,G)$. If $G$ is a closed subgroup of another algebraic group $H$, we write $H \times_G X$ for the induced variety, which is the quotient of $H \times X$ by the middle action of $G$ given by $g \cdot (h,x) = (h g\inv, g\cdot x)$. In this situation, we have an isomorphism $[G \backslash X] \cong [H \backslash (H \times_G X)]$. This will be used frequently.

Suppose $f: X \to Y$ is a map of varieties over $\BC$, the (derived) pushforward and pullback functors in the six operation formalism are denoted by $f_*$, $f_!$, $f^*$, and $f^!$. 
If $X$ and $Y$ admit group actions $G \acts X$, $H \acts Y$ and if the actions are compatible with respect to a group homomorphism $\varphi: G \to H$, then $f$ induces a morphism of quotient stacks $[G \backslash X] \to [H \backslash Y]$ and the above functors lift to functors between equivariant derived categories $D^b(X,G)$ and $D^b(Y,H)$ which will be denoted by the same symbols. If moreover the induced map $[G \backslash X] \to [H \backslash Y]$ is an isomorphism, we always identify the derived categories $D^b(X,G)$ and $D^b(Y,H)$ via $f^*$, unless otherwise stated.

For any subset $S$ inside a $\BQ$-coset $a+\BQ \subset \BC$ (typically $S = \supp \varphi$ for some integral weight function $\varphi$, see \textsection \ref{subsec:LLC/Qp}), we will translate functions defined on $\BQ$ to $S$. Examples include the floor and ceiling functions: for $r \in a + \BQ$, 
\begin{align*}
	\floor r &:= a + (\text{largest integer that is } \le r-a),\\
	\ceil r &:= a + (\text{smallest integer that is } \ge r-a).
\end{align*}
Another example is the order $\le$: for $x,y \in a + \BQ$,
\begin{equation*}
	x \le y \iff x-a \le y-a.
\end{equation*}
The $\min$ and $\max$ functions are modified similarly. 

\subsection{Acknowledgement}

We wish to thank Jeffrey Adams, Peng Shan, and Peter Trapa for valuable discussions. We would like to thank the referee(s) for many valuable comments.

\section{Preliminaries}\label{sec:prelim}

\subsection{Local Langlands correspondence for $\GL_n(\BC)$}\label{subsec:LLC-GLnC}

Let us briefly review the shape of local Langlands correspondence for real reductive algebraic groups due to Adams-Barbasch-Vogan \cite{ABV}, or rather the pure version formulated in \cite{Vogan:p-adic-Langlands}. We then apply it to $\GL_n(\BC)$.

\begin{clause}[Local Langlands correspondence for real groups {\cite{Vogan:p-adic-Langlands}}]\label{cls:LLC/R}	
	Let $G_\BC$ be a connected reductive algebraic group over $\BC$ with a fixed pinning, and let $\check G_\BC$ be its Langlands dual group equipped with the dual pinning. In particular, we have fixed Cartan subalgebras $\check \fh_\BC \subset \check \fg_\BC$ and $\fh_\BC \subset \fg_\BC$ so that $(\check \fh_\BC)^* = \fh_\BC$. Fix an action $\Gamma = \operatorname{Gal}(\BC/\BR) = \{1,\delta\}$ on $\check G_\BC$ by \textit{anti-holomorphic} automorphisms giving rise to a \textit{quasi-split} real form $\check G_\BR$ of $\check G_\BC$. This action induces an involution on the root datum of $\check G_\BC$ and hence one on the root datum of $G_\BC$, which lifts to a unique \textit{holomorphic} involution of $G_\BC$ preserving the pinning, also denoted by $\delta$. 
	We write $\check G_\BC^\Gamma = \check G_\BC \rtimes \{1,\delta\}$ and $G_\BC^\Gamma = G_\BC \rtimes \{1,\delta\}$ for the semidirect products with respect to the above actions.
	
	An \textit{infinitesimal character} for $\check G_\BC$ is a $\BC$-algebra homomorphism $Z(\cU(\check \fg_\BC)) \to \BC$, where $\cU(\check \fg_\BC)$ is the enveloping algebra and $Z(\cU(\fg_\BC))$ is its center. By the Harish-Chandra isomorphism, infinitesimal characters can be described by an element $\Lambda_\BR$ in $(\check \fh_\BC)^*$ up to Weyl group conjugacy. By the identification $(\check \fh_\BC)^* = \fh_\BC$, an infinitesimal character is a $W$-orbit of an element $\Lambda_\BR$ in $\fh_\BC$, where $W$ denotes the Weyl group. 
	
	
	The integral Weyl group of $\Lambda_\BR$ is 
	\begin{equation*}
		W_{\Lambda_\BR} = \big\{ w \in W \mid w \Lambda_\BR - \Lambda_\BR \text{ is a $\BZ$-linear combination of roots} \big\}.
	\end{equation*}
	
	A \textit{pure real form} of $\check G_\BC$ is an element $x \in \check G_\BC^\Gamma - \check G_\BC$ so that $x^2 = 1$. Two pure real forms are equivalent if they are conjugate by an element of $\check G_\BC$. There is a bijection
	\begin{equation*}
		\big\{ x \in \check G_\BC^\Gamma - \check G_\BC \mid x^2 =1 \big\}/\sim \; \cong H^1(\Gamma, \check G_\BC),
	\end{equation*}
	see \cite[Proposition 2.7(c)]{Vogan:p-adic-Langlands}. Given a pure real form $x$, $\Ad x|_{\check G_\BC}$ (conjugation on $\check G_\BC$ by $x$) is an anti-holomorphic involution and gives rise to a real form $\check G_{\BR,x} = (\check G_\BC)^{\Ad x}$ of $\check G_\BC$.
	We will write $\Rep_{\Lambda_\BR}(\check G_{\BR,x})$ for the category of Harish-Chandra modules over the pair $(\check \fg_\BC, \check K_x)$ with infinitesimal character $\Lambda_\BR$, where $\check K_x$ is the complexification of a maximal compact subgroup of $\check G_{\BR,x}$.
	
	Fix a dominant (possibly singular) infinitesimal character $\Lambda_\BR \in \fh_\BC$ for $\check G_\BC$. Consider the group $G_{\BC,\Lambda_\BR} = Z_{G_\BC}(\exp(2\pi i \Lambda_\BR))$, sometimes called the endoscopic group. We set 
	\begin{equation*}
		\cY({\Lambda_\BR}) = \big\{ y \in G_\BC^\Gamma- G_\BC \mid y^2 = \exp(2\pi i {\Lambda_\BR}) \big\} \times \cP_{\Lambda_\BR}
	\end{equation*}
	where $\cP_{\Lambda_\BR}$ is the partial flag variety of $G_{\BC,\Lambda_\BR}$ whose type is given by simple roots orthogonal to ${\Lambda_\BR}$. It is equipped with the natural diagonal action of $G_{\BC, \Lambda_\BR}$ with finitely many orbits, and hence one may consider standard and simple perverse sheaves as $!$- and $!*$-extensions of locally constant perverse sheaves on orbits. The space $\cY(\Lambda_\BR)$ is the space of Langlands parameters with infinitesimal character ${\Lambda_\BR}$ introduced in \cite{ABV} and reformulated in \cite{Adams-du-Cloux}, which we call the \textbf{ABV space}. The set of $G_{\BC, \Lambda_\BR}$-orbits (resp. simple $G_{\BC, \Lambda_\BR}$-equivariant perverse sheaves) on $\cY(\Lambda_\BR)$ parameterize $\Phi(\check G_\BR, \Lambda_\BR)$ (resp. $\Xi(\check G_\BR, \Lambda_\BR)$), the set of conjugacy classes of (resp. complete) Langlands parameters with infinitesimal character $\Lambda_\BR$.
	
	\begin{theorem}[LLC/$\BR$, {\cite[Conjecture 8.11']{Vogan:p-adic-Langlands}, \cite[Theorem 1.24, Proposition 17.16]{ABV}}]\label{thm:Vogan-duality}~
		There is a perfect pairing
		\begin{equation*}
			\bigoplus_{x \in H^1(\Gamma, \check G_\BC)} K \Rep_{\Lambda_\BR}(\check  G_{\BR,x}) \times 
			K \Perv(\cY(\Lambda_\BR), G_{\BC,\Lambda_\BR})
			\aro \BZ,
		\end{equation*}
		under which the classes of simple (resp. standard) objects on both sides form dual bases up to some signs. In particular, the change of basis matrices on both sides between simple objects and standard objects are transpose to each other.
	\end{theorem}
	
	The statement regarding change of basis matrices was first obtained by Vogan \cite{Vogan:IC4} in the case of regular infinitesimal characters, and the above perfect pairing is also known as the Vogan duality. The signs here are a combination of the ones due to Kottwitz \cite[pp. 291-292]{Kottwitz:sign} and $(-1)^{d_Q}$ where $d_Q$ is the dimension of a $G_{\BC,\Lambda_\BR}$-orbit in $\cY(\Lambda_\BR)$. The Kottwitz signs are equal to $1$ for quasi-split groups, in particular for $\GL_n(\BC)$.
\end{clause}

In the case of $\check G = \check G_\BR = \GL_n(\BC)$, both sides of these pairings are significantly simplified. Namely, $H^1(\Gamma, \check G_\BC) = 1$, and there is only one $G_\BC$-conjugacy class in the set $\{y \in G_\BC^\Gamma-G_\BC \mid y^2 = \exp(2\pi i \Lambda_\BR)\}$. Let $y$ be any element in the latter set. Then 
\begin{equation}\label{eqn:real-parameter-stack}
	[G_{\BC,\Lambda_\BR} \backslash \cY(\Lambda_\BR) ]
	= [G_{\BC,\Lambda_\BR} \backslash \big( (G_{\BC,\Lambda_\BR} \cdot y) \times \cP_{\Lambda_\BR} \big)]
	= [K_{\Lambda_\BR,y} \backslash \cP_{\Lambda_\BR}],
\end{equation}
where $K_{\Lambda_\BR,y} = G_{\BC,\Lambda_\BR}{}^{\Ad y}$ (fixed point subgroup). Hence the pairing in Theorem \ref{thm:Vogan-duality} reduces to the following form
\begin{equation*}
	K \Rep_{\Lambda_\BR}( \check G) \times K \Perv(\cP_{\Lambda_\BR}, K_{\Lambda_\BR,y}) \aro \BZ
\end{equation*}
We need more notations to describe these objects.

\begin{clause}[Representation theoretic notations for $\GL_n(\BC)$]\label{cls:RT-notations-GLn}
	Let us write $\check G = G = \GL_n(\BC)$, $\check \fg = \fg = \fgl_n(\BC)$. Following \cite[\textsection 7.1]{Vogan:book}, the embedding
	\begin{equation*}
		\check \iota: \check g \injects \check \fg \oplus \check \fg,\quad
		x \mapsto (x,\bar x)
	\end{equation*}
	realizes $\check \fg$ as a real form of $\check \fg \oplus \check \fg$ with respect to the anti-holomorphic involution
	\begin{equation*}
		\delta: \check \fg \oplus \check \fg \bijects \check \fg \oplus \check \fg,\quad
		(x,y) \mapsto (\bar y, \bar x).
	\end{equation*}
	
	Write $\check H \subset \check G$ for the diagonal torus; so $\check H \cong (\BC^\times)^n$ as a real group. Denote its Lie algebra by $\check \fh$, then the complexification of $\check \fh$ along $\check \iota$ is $\check \fh \oplus \check \fh$. Recall that a continuous character $\BC^\times \to \BC^\times$ has the form $z \mapsto z^a \bar z^b$ for some $a,b \in \BC$ with $a-b \in \BZ$. So a continuous character $\check H \to \BC^\times$ is determined by a $2n$-tuple
	\begin{equation*}
		\Lambda_\BR = (\lambda_L, \lambda_R) = (\lambda_{L,1},\ldots,\lambda_{L,n},\lambda_{R,1}, \ldots, \lambda_{R,n}) \in \BC^{2n},\quad 
		\lambda_{L,i} - \lambda_{R,i} \in \BZ,
	\end{equation*}
	namely
	\begin{equation*}
		\chi_{\Lambda_\BR} = \chi_{(\lambda_L, \lambda_R)}: \check H \cong (\BC^\times)^n \aro \BC^\times,\quad
		(z_1,\ldots,z_n) \mapsto z_1{}^{\lambda_{L,1}} \bar z_1{}^{\lambda_{R,1}} \cdots z_n{}^{\lambda_{L,n}} \bar z_n{}^{\lambda_{R,n}}.
	\end{equation*}
	Its differential is the restriction of the Lie algebra character
	\begin{equation}\label{eqn:Lambda_R-as-Lie-alg-char}
		\Lambda_\BR = (\lambda_L, \lambda_R): \check \fh \oplus \check \fh \aro \BC,\quad
		(x,y) \mapsto \lambda_L(x) + \lambda_R(y)
	\end{equation}
	along $\check \iota: \check \fh \to \check \fh \oplus \check \fh$.
	
	Let $\check B = \check H \check N \subset \check G$ be the upper-triangular Borel. We consider the normalized parabolic inductions
	\begin{align*}
		\nInd_{\check B}^{\check G} \chi_{(\lambda_L, \lambda_R)} 
		&= \big\{ f \in C^\infty(\check G) \mid \forany b = hn \in \check B, f(bg) = \delta(h)^{\frac12} \chi_{(\lambda_L, \lambda_R)}(h) f(g) \big\}.
	\end{align*}
	Here $\delta$ is the modular character of $\check B$. We write
	\begin{equation*}
		X(\lambda_L, \lambda_R) := \check K \text{-finite part of } \nInd_{\check B}^{\check G} \chi_{(\lambda_L, \lambda_R)},
	\end{equation*}
	where $\check K = \bU(n) \subset \check G$ is a maximal compact subgroup, and
	\begin{equation*}
		\bar X(\lambda_L, \lambda_R) := \text{the Langlands subquotient of } X(\lambda_L, \lambda_R).
	\end{equation*}
	Then $X(\lambda_L, \lambda_R)$ has infinitesimal character given by $\Lambda_\BR = (\lambda_L, \lambda_R)$ (we are viewing $\Lambda_\BR$ as a character on $\check \fh \oplus \check \fh$ as in (\ref{eqn:Lambda_R-as-Lie-alg-char})). Morever, the set of ``standard'' (resp. irreducible) objects in $\Rep_{\Lambda_\BR}( \check G)$ is precisely
	\begin{equation*}
		\big\{ X(\lambda_L, w \lambda_R) \mid w \in W_{\Lambda_\BR} \big\} \quad 
		\big(\text{resp. } \big\{ \bar X( \lambda_L, w \lambda_R) \mid w \in W_{\Lambda_\BR} \big\} \big),
	\end{equation*}
	where by an abuse of notation we write
	\begin{equation}\label{eqn:W-LambdaR}
		W_{\Lambda_\BR} := \{w \in W \mid w \lambda_L - \lambda_L \text{ is a $\BZ$-lienar combination of roots of $\check \fh$ in $\check \fg$}\}
	\end{equation}	
	for the integral Weyl group of $\lambda_L \in \check \fh^*$ (which is the same as the integral Weyl group of $\lambda_R$). This $W_{\Lambda_\BR}$ is different from the integral Weyl group $W_{\Lambda_\BR}$ defined in \ref{cls:LLC/R}; in the current notation, the integral Weyl group in \ref{cls:LLC/R} is equal to $W_{\Lambda_\BR} \times W_{\Lambda_\BR}$. The ``standard'' modules $X(\lambda_L, w\lambda_R)$ are different from the actual standard modules as representations, but they are the same in the Grothendieck group. Since we almost entirely deal with Grothendieck groups, it doesn't hurt to work with $X(\lambda_L, w\lambda_R)$ instead.
\end{clause}

\begin{clause}[Geometric setup and LLC for $\GL_n(\BC)$]\label{cls:real-geom-setup}
	On the dual side, it is more convenient to use a different embedding of $\fg$ into $\fg \oplus \fg$, namely
	\begin{equation*}
		\iota: \fg \injects \fg \oplus \fg,\quad
		x \mapsto (x, - \bar x^t),
	\end{equation*}
	and the involution $\delta$ is changed to
	\begin{align*}
		\delta: \fg \oplus \fg \bijects \fg \oplus \fg,&\quad 
		(x,y) \mapsto (y,x).
	\end{align*}
	
	Write $G_{\Lambda_\BR} := G_{\lambda_L} = G_{\lambda_R}$ (the second equality is because $\lambda_L - \lambda_R$ is integral), a Levi subgroup in $G$. Then the endoscopic group for $\Lambda_\BR$ is
	\begin{equation*}
		(G \times G)_{\Lambda_\BR} = G_{\Lambda_\BR} \times G_{\Lambda_\BR},
	\end{equation*}
	and 
	\begin{equation*}
		K_{\Lambda_\BR,y} = \Delta G_{\Lambda_\BR}.
	\end{equation*}	
	Hence the quotient $[G_\BC \backslash \cY(\Lambda_\BR)]$ defined in (\ref{eqn:real-parameter-stack}) can be identified with $[\Delta G_{\Lambda_\BR} \backslash (\cP_{\lambda_L} \times \cP_{\lambda_R})]$, where $\cP_{\lambda_L}$ (resp. $\cP_{\lambda_R}$) is the partial flag variety of $G_{\Lambda_{\BR}}$ whose type is determined by roots orthogonal to $\lambda_L$ (resp. $\lambda_R$), i.e. is given by the simple roots
	\begin{equation*}
		\Pi_L = \{ \alpha_i \mid \lambda_{L,i} = \lambda_{L,i+1} \} \quad
		(\text{resp. } \Pi_R = \{ \alpha_i \mid \lambda_{R,i} = \lambda_{R,i+1} \}).
	\end{equation*}
	
	We now describe standard\footnotemark and irreducible sheaves on $\cP_{\lambda_L} \times \cP_{\lambda_R}$. Assume first that $\Lambda_\BR$ is regular so that $\cP_{\lambda_L} \times \cP_{\lambda_R} = \cB_{\Lambda_\BR} \times \cB_{\Lambda_\BR}$. For $w \in W_{\Lambda_\BR}$, write
	\begin{equation*}
		Z_w := \big\{ (x,y) \in \cB_{\Lambda_\BR} \times \cB_{\Lambda_\BR} \mid \text{$\fb_x$ and $\fb_y$ are in relative position $w$} \big\}.
	\end{equation*}
	Here we say $\fb_x$ and $\fb_y$ are in relative position $w$ if for any maximal torus $T$ contained in $B_x \cap B_y$ and any lift $\dot{w}$ of $w$ in the normalizer $N_G(T)$, we have $\dot{w} F(B_y) = F(B_x)$ where $F(B)$ denotes the flag corresponding to the Borel $B$. The $Z_w$'s are precisely the $\Delta G_{\Lambda_\BR}$-orbits in $\cB_{\Lambda_\BR} \times \cB_{\Lambda_\BR}$. The orbit $Z_w$ contains $Z_v$ in its closure if and only if $w \ge v$ in the Bruhat order. The $!$-extension (resp. $!*$-extension) of the sheaf $\ubar \BC_{Z_w}[\dim Z_w]$ from $Z_w$ to $\cB_{\Lambda_\BR} \times \cB_{\Lambda_\BR}$ in the six operation formalism is denoted by $\cM_w$ (resp. $\cL_w$) and they are precisely the standard (resp. irreducible) objects in the category $\Perv(\cB_{\Lambda_\BR} \times \cB_{\Lambda_\BR}, \Delta G_{\Lambda_\BR})$. 
	\footnotetext{%
		In the geometric context, ``standard objects'' refer to the $!$-pushforwards of local systems on orbits. It originates from the study of category $\cO$, where the Beilinson-Bernstein localizations of Verma modules can be realized as $!$-pushforwards of trivial local systems on Schubert cells, and they are usually called ``standard $\cD$-modules'', see for example \cite[\S 2.2]{Soergel:n-coh}. The terminology was later adapted to the study of real groups and other areas of geometric representation theory, see for example \cite{HMSW:localization1}.
	}
	
	If $\Lambda_\BR$ is singular, then the $\Delta G_{\Lambda_\BR}$-orbit in $\cP_{\lambda_L} \times \cP_{\lambda_R}$ are parameterized by the double cosets $W(P_{\lambda_L}) \backslash W_{\Lambda_\BR} /W(P_{\lambda_R})$, where a coset $C$ corresponds to the orbit which is the image of $Z_w \subset \cB_{\Lambda_\BR} \times \cB_{\Lambda_\BR}$ under the projection $\cB_{\Lambda_\BR} \times \cB_{\Lambda_\BR} \surj \cP_{\lambda_L} \times \cP_{\lambda_R}$ for any $w \in C$. By abuse of notation, we again write $Z_w$ for orbits in $\cP_{\lambda_L} \times \cP_{\lambda_R}$, write $\cM_w$ for the standard module supported on $Z_w$, and write $\cL_w$ for the unique simple quotient of $\cM_w$. 
	
\end{clause}

\begin{namedtheorem}[LLC for $\GL_n(\BC)$]\label{prop:LLC-GLnC}
	Suppose $\Lambda_\BR$ is a dominant infinitesimal character for $\check G = \GL_n(\BC)$. The perfect pairing of Theorem \ref{thm:Vogan-duality} reduces to
	\begin{equation}\label{eqn:LLC-GLn-pairing}
		K \Rep_{\Lambda_\BR}( \check G) \times K \Perv(\cP_{\lambda_L} \times \cP_{\lambda_R}, \Delta G_{\Lambda_\BR}) \aro \BZ
	\end{equation}
	under which $X(\lambda_L, w \lambda_R)$ (resp. $\bar X(\lambda_L, w \lambda_R)$) is dual to $(-1)^{\dim Z_{w\inv}} \cM_{w\inv}$ (resp. $(-1)^{\dim Z_{w\inv}} \cL_{w\inv}$) for any $w \in W(P_{\lambda_L}) \backslash W_{\Lambda_\BR} /W(P_{\lambda_R})$. In other words, under the pairing,
	\begin{itemize}
		\item $\{ X(\lambda_L, w \lambda_R) \}_{w \in W(P_{\lambda_L}) \backslash W_{\Lambda_\BR} /W(P_{\lambda_R})}$ and $\{(-1)^{\dim Z_{w\inv}} \cM_{w\inv}\}_{w \in W(P_{\lambda_L}) \backslash W_{\Lambda_\BR} /W(P_{\lambda_R})}$ form dual bases, and 
		\item $\{ \bar X(\lambda_L, w \lambda_R)\}_{w \in W(P_{\lambda_L}) \backslash W_{\Lambda_\BR} /W(P_{\lambda_R})}$ and $\{(-1)^{\dim Z_{w\inv}} \cL_{w\inv}\}_{w \in W(P_{\lambda_L}) \backslash W_{\Lambda_\BR} /W(P_{\lambda_R})}$ form dual bases.
	\end{itemize}	 
\end{namedtheorem}

\subsection{Translation functors and push-pull functors}\label{subsec:translation}

In this subsection we recall the notion of translation functors and their geometric counterparts on the Langlands dual side.


We temporarily return to the general setup \ref{cls:LLC/R}. Let $\Lambda_\BR$ and $\Lambda_\BR'$ be two infinitesimal characters of $\check G_\BC$ that differ by an element $\mu$ in the character lattice of  $\check G_\BC$ (i.e. $\Lambda_\BR' = \Lambda_\BR + \mu$). Then the endoscopic groups $G_{\Lambda_\BR}$ and $G_{\Lambda_{\BR'}}$ are equal. For any real form $\check G_\BR$ of $\check G_\BC$, the \textit{translation functor} from $\Lambda_\BR$ to $\Lambda_\BR'$ is
\begin{equation*}
	T_{\Lambda_\BR}^{\Lambda_\BR'} : \Rep_{\Lambda_\BR}( \check G_\BR ) \aro \Rep_{[\Lambda_{\BR'}]}( \check G_\BR ),\quad
	V \mapsto \big( V \dotimes_\BC F_\mu \big)_{[\Lambda_\BR']},
\end{equation*}
where $F_\mu$ is the irreducible finite dimensional representation of $\check G_\BC$ with extremal weight $\mu$, and the subscript $(-)_{[\Lambda_\BR']}$ takes the generalized infinitesimal character $\Lambda_\BR'$. This is an exact functor. If we take the functor $T_{\Lambda_\BR}^{\Lambda_\BR'}$ on the direct sum $\bigoplus_x K \Rep_{\Lambda_\BR}(\check  G_{\BR,x})$ over pure real forms, then under the perfect pairing of Theorem \ref{thm:Vogan-duality}, we obtain an adjoint operator $(T_{\Lambda_\BR}^{\Lambda_\BR'})^*: K\Perv(\cY(\Lambda_\BR'), G_{\BC,\Lambda_\BR'}) \to K\Perv(\cY(\Lambda_\BR), G_{\BC, \Lambda_\BR})$. In a special situation, the map $(T_{\Lambda_\BR}^{\Lambda_\BR'})^*$ admits a nice geometric description. Suppose $\Lambda_\BR$ and $\Lambda_\BR'$ are in the closure of the same Weyl chamber, and $\Lambda_\BR$ is at least as regular as $\Lambda_\BR'$, i.e. if a simple root $\alpha$ is orthogonal to $\Lambda_\BR$, then it is also orthogonal to $\Lambda_\BR'$. Then we have a natural projection $\varpi: \cP_{\Lambda_\BR} \to \cP_{\Lambda_{\BR'}}$, and hence a projection $\cY(\Lambda_\BR) \surj \cY(\Lambda_\BR')$ which we still denote by $\varpi$.

\begin{proposition}\label{prop:adjoint-of-translation}
	Suppose $\Lambda_\BR$ and $\Lambda_\BR'$ are in the closure of the same Weyl chamber, and $\Lambda_\BR$ is at least as regular as $\Lambda_\BR'$. Consider the perfect pairings in Theorem \ref{thm:Vogan-duality} base changed to $\BC$
		\begin{gather*}
			\bigoplus_{x \in H^1(\Gamma, \check G_\BC)} K \Rep_{\Lambda_\BR}( \check G_{\BR,x} )_\BC
			\times
			K \Perv(\cY(\Lambda_\BR),  G_{\BC,\Lambda_\BR})_\BC
			\aro \BC\\
			\bigoplus_{x \in H^1(\Gamma, \check G_\BC)} K \Rep_{\Lambda_{\BR'}}( \check G_{\BR,x} )_\BC
			\times
			K \Perv(\cY(\Lambda_\BR'), G_{\BC,\Lambda_\BR'})_\BC
			\aro \BC.
		\end{gather*}
		Under these pairings, $T_{\Lambda_\BR}^{\Lambda_\BR'}$ (resp. $T_{\Lambda_\BR'}^{\Lambda_\BR}$) is adjoint to $\varpi^*$ (resp. $\varpi_*$), i.e.
		\begin{equation*}
			\langle T_{\Lambda_\BR}^{\Lambda_\BR'} -,- \rangle = \langle -, \varpi^* (-) \rangle,\quad
			\langle T_{\Lambda_\BR'}^{\Lambda_\BR} -,- \rangle = \langle -, \varpi_* (-) \rangle.
		\end{equation*}
\end{proposition}

The theorem essentially follows from \cite[14.9.(b)]{Vogan:IC4} (see also \cite[Proposition 17.16]{ABV}). We postpone the proof to Appendix \ref{subsec:trans_vs_pushpull}, see Proposition \ref{prop:sing_to_sing}.

Now let us restrict to $\check G = \check G_\BR = \GL_n(\BC)$, where a geometric description is available for more general types of translations. Suppose $\Lambda_\BR = (\lambda_{L,1},\ldots, \lambda_{L,n}, \lambda_{R,1},\ldots,\lambda_{R,n})$ is a dominant integral infinitesimal character. Suppose $1 \le j \le n$ and $c >0$ is so that either $j+c-1=n$ or 
\begin{equation}\label{eqn:lambdaL-and-lambdaL'-a}
	\lambda_{L,j} = \lambda_{L,j+1} = \cdots = \lambda_{L,j+c-1} > \lambda_{j+c}.
\end{equation}
Let $\Lambda_\BR' = (\lambda_L', \lambda_R')$ be so that $\lambda_R' = \lambda_R$, and
\begin{equation}\label{eqn:lambdaL-and-lambdaL'-b}
	\lambda_L' = \lambda_L - (e_j + \cdots + e_{j+r-1})
\end{equation}
where the $e_i$'s are the coordinate vectors of $\BC^{2n}$. In other words,
\begin{equation}\label{eqn:lambdaL-and-lambdaL'-c}
	\lambda_L' = (\lambda_{L,1}, \ldots, \lambda_{L,j-1}, \lambda_{L,j}-1, \ldots, \lambda_{L,j+c-1}-1, \lambda_{L,j+c}, \ldots, \lambda_{L,n}).
\end{equation}
Recall that one of our main theorems (Theorem \ref{thm:trans-der}, see also Corollary \ref{cor:compare-functors}) aims to compare Lusztig induction $\LInd_{\psi, c[k]}^\varphi$ on the $p$-adic side and the push-pull functor $\bI_{\Lambda_{\BR'}}^{\Lambda_\BR}$ on the real side, where the weight functions $\psi$ and $\varphi$ appearing in the $p$-adic side satisfy $\varphi = \psi + c[k]$ and $k > r$. Correspondingly, the infinitesimal characters $\Lambda_\BR$ and $\Lambda_\BR'$ on the real side can be chosen to satisfy (\ref{eqn:lambdaL-and-lambdaL'-a})-(\ref{eqn:lambdaL-and-lambdaL'-c}). This explains the reason for studying these cases.

Neither $\Lambda_\BR$ or $\Lambda_\BR'$ is more regular than the other, and hence the parabolic subgroups $P_{\lambda_L}$ and $P_{\lambda_R}$ may not contain one another. Let $Q_L = P_{\lambda_L} \cap P_{\lambda_L'}$, and write $\cQ_L = G/Q_L$. Then we have natural projections
\begin{equation*}
	\cP_{\lambda_L'} \times \cP_{\lambda_R'} = \cP_{\lambda_L'} \times \cP_{\lambda_R} \xtwoheadleftarrow{\pi'} \cQ_L \times \cP_{\lambda_R} \xtwoheadrightarrow{\pi} \cP_{\lambda_L} \times \cP_{\lambda_R}
\end{equation*}
and hence a pullback-pushforward functor
\begin{equation*}
	\bI_{\Lambda_{\BR'}}^{\Lambda_\BR} = \pi_* \circ \pi'{}^*: D^b( \cP_{\lambda_L'} \times \cP_{\lambda_R}, \Delta G) \aro D^b( \cP_{\lambda_L} \times \cP_{\lambda_R}, \Delta G).
\end{equation*}

\begin{proposition}\label{prop:adjoint-of-translation-mixed}
	Suppose $\Lambda_\BR$ and $\Lambda_\BR'$ are as above. Consider the perfect pairings (\ref{eqn:LLC-GLn-pairing}) base changed to $\BC$ 
	\begin{gather*}
		K \Rep_{\Lambda_\BR}( \check G)_\BC \times K \Perv(\cP_{\lambda_L} \times \cP_{\lambda_R}, \Delta G_{\Lambda_\BR})_\BC \aro \BC\\
		K \Rep_{\Lambda_\BR'}( \check G)_\BC \times K \Perv(\cP_{\lambda_L'} \times \cP_{\lambda_R}, \Delta G_{\Lambda_\BR})_\BC \aro \BC.
	\end{gather*}
	Under these pairings, $T_{\Lambda_\BR}^{\Lambda_\BR'}$ is adjoint to $\bI_{\Lambda_{\BR'}}^{\Lambda_\BR}$, that is
	\begin{equation*}
		\langle T_{\Lambda_\BR}^{\Lambda_\BR'} -,- \rangle = \langle -, \bI_{\Lambda_{\BR'}}^{\Lambda_\BR} - \rangle.
	\end{equation*}
\end{proposition}

\begin{proof}
	By Proposition \ref{prop:2nd_decomposition}, $T_{\Lambda_\BR}^{\Lambda_\BR'}$ can be factored as
	\begin{equation*}
		T_{\Lambda_\BR}^{\Lambda_\BR'} = T_{\Lambda_\BR''}^{\Lambda_\BR'} \circ T_{\Lambda_\BR}^{\Lambda_\BR''}.
	\end{equation*}
	Here $\Lambda_\BR'' = (\lambda_L'', \lambda_R'')$ is a dominant integral infinitesimal character with $\lambda_R'' = \lambda_R$ and $P_{\lambda_L''} = P_{\lambda_L} \cap P_{\lambda_L'} = Q_L$. In particular, $\Lambda_\BR''$ is more regular than $\Lambda_\BR$ and $\Lambda_\BR'$. By Proposition \ref{prop:adjoint-of-translation}, $T_{\Lambda_\BR''}^{\Lambda_\BR'}$ is adjoint to $\pi'{}^*$ and $T_{\Lambda_\BR}^{\Lambda_\BR''}$ is adjoint to $\pi_*$. Hence $T_{\Lambda_\BR}^{\Lambda_\BR'}$ is adjoint to $\pi_* \circ \pi'{}^* = \bI_{\Lambda_{\BR'}}^{\Lambda_\BR}$, as required.
\end{proof}

\subsection{Local Langlands correspondence for $\GL_m(\BQ_p)$}\label{subsec:LLC/Qp}


\begin{clause}[Vogan's local Langlands conjecture for $p$-adic groups, {\cite{Vogan:p-adic-Langlands}}]
	Let us first review the general statement for Vogan's formulation of local Langlands correspondence. We will only describe the shape of the theory and omit technical details since they are not really needed in what follows (we will eventually be working with alternative descriptions of the objects involved). 
	 
	Let $\check G$ be a connected reductive algebraic group over $F = \BQ_p$, let $G$ be its Langlands dual group (over $\BC$), let $W_F = I_F \rtimes \langle \ff \rangle$ be the Weil group of $F$, where $I_F$ denotes the Inertia and $\ff$ is a fixed lift of the arithmetic Frobenius $x \mapsto x^p$. Let $W_F' = W_F \times \SL_2(\BC)$ be the Weil-Deligne group. An \textit{infinitesimal character} is a continuous homomorphism
	\begin{equation*}
		\Lambda : W_F \aro {}^L \check G
	\end{equation*}
	satisfying certain conditions (see \cite[(4.4)]{Vogan:p-adic-Langlands} and \cite[\S 4.1]{Cunningham-etal:microlocal}). The \textit{Vogan variety} of $\Lambda$ is
	\begin{equation*}
		E_\Lambda = \big\{ \xi \in \fg \mid \forany w \in W_F, \Ad( \Lambda(w)) \xi = |w| \xi \big\}.
	\end{equation*}
	This is equipped with an action of $G_\Lambda = Z_G(\Lambda)$ with finitely many orbits. Similar to the real case, the set of $G_\Lambda$-orbits (resp. simple $G_\Lambda$-equivariant perverse sheaves) on $E_\Lambda$ parameterize $\Phi(\check G, \Lambda)$ (resp. $\Xi(\check G, \Lambda)$), the set of conjugacy classes of (resp. complete) Langlands parameters with infinitesimal character $\Lambda$.
	
	An element $\delta \in H^1(F,\check G)$ is called a \textit{pure inner form} of $\check G$, which determines a $p$-adic group $\check G_\delta$. For each irreducible smooth representation $\pi$ of $\check G_\delta(F)$, the local Langlands correspondence/conjecture attaches to $\pi$ a Langlands parameter
	\begin{equation*}
		\phi_\pi: W_F' \aro {}^L \check G
	\end{equation*}
	a continuous map satisfying certain conditions (see \cite[Definition 4.2]{Vogan:p-adic-Langlands}, \cite[\S 3.4]{Cunningham-etal:microlocal}). For any Langlands parameter $\phi$, one may attach to it an infinitesimal character $\Lambda_\phi$ by
	\begin{equation*}
		\Lambda_\phi : W_F \aro {}^L \check G,\quad
		w \mapsto \phi\left( w, 
		\begin{pmatrix}
			|w|^{1/2} \\ & |w|^{-1/2}
		\end{pmatrix} \right).
	\end{equation*}
	Let $\Rep_\Lambda(\check G_\delta)$ be the full subcategory of smooth representations of $\check G$ of finite length whose irreducible composition factors $\pi$ all have infinitesimal character $\Lambda$. Let $G_\Lambda = Z_G(\Lambda)$.
\end{clause}

\begin{conjecture}[Vogan's local Langlands correspondence]
	There is a perfect pairing
	\begin{equation}\label{eqn:LLC-Qp-pairing}
		\Big( \bigoplus_{\delta \in H^1(F,\check G)} K \Rep_\Lambda( \check G_\delta) \Big) \times
		K D^b( E_\Lambda, G_\Lambda) \aro \BZ
	\end{equation}
	such that the classes of simple (resp. standard) objects on both sides form dual bases up to some signs. In particular, the change of basis matrices on both sides between simple objects and standard objects are transpose to each other.
\end{conjecture}


Let us specialize to $\check G = G = \GL_m$. In this case $H^1(F,\check G)$ is trivial. We will be focusing on the case where the infinitesimal character is unramified\footnote{When it comes to the multiplicities of irreducible representations in standard representations, the method used for unramified representations more or less works also in general, see \cite[Remark 8.7]{Zelevinskii} and \cite[Remarque 4.18]{Minguez-Secherre}}. Here an infinitesimal character $\Lambda$ is \textit{unramified} if it is trivial on the inertia subgroup $I_F$. They correspond to representations whose composition factors admit nonzero vectors fixed by the Iwahori subgroup. For unramified infinitesimal characters of $\GL_m(\BQ_p)$, Vogan's conjecture is proven by a combination of Ariki \cite{Ariki} and Chriss-Ginzburg \cite[Theorem 8.6.23]{Chriss-Ginzburg} (see also \cite{Kazhdan-Lusztig:Deligne-Langlands}).

\begin{clause}[Parameter space for $\GL_m(\BQ_p)$]\label{cls:ELambda-as-Evarphi}
	We will work with an alternative description of $E_\Lambda$ and $G_\Lambda$. Let $\phi$ be an unramified Langlands parameter, a map from $\langle \ff \rangle \times \SL_2$ to $\GL_m$, or equivalently an $m$-dimensional representation of $\langle \ff \rangle \times \SL_2$. The conditions for $\phi$ \cite[Definition 4.2]{Vogan:p-adic-Langlands} in particular implies that this representation decomposes into irreducible ones. So up to conjugacy, $\phi$ is of the form 
	\begin{equation*}
		\phi = \Big( |\cdot|^{\frac12 a_1} \boxtimes S_{d_1} \Big) \oplus \cdots \oplus \Big( |\cdot |^{\frac12 a_n} \boxtimes S_{d_n} \Big)
	\end{equation*}
	where $a_i \in \BC$, $|\cdot|$ is the norm function which sends $\ff$ to $p$, $d_i \in \BZ_{>0}$ with $\sum_i d_i = m$, and $S_{d_i}$ is the irreducible $d_i$-dimensional representation of $\SL_2$. Let $\Lambda = \Lambda_\phi$ be its infinitesimal character. Then $\Lambda$ can be conjugated to the following form
	\begin{equation*}
		\Lambda(\ff) = \bigoplus_{i=1}^n p^{-\frac12 a_i} \diag \big( p^{-\frac12 (d_i-1)}, p^{-\frac12(d_i-3)}, \ldots, p^{-\frac12(1-d_i)} \big)
	\end{equation*}
	Note that the conjugacy class of $\Lambda$ is determined by that of $\Lambda(\ff)$. Since the latter is determined by the multiplicities of each eigenvalue $p^i$ of $\Lambda(\ff)$ acting on $\BC^m$, it suffices to record the multiplicities of each $p^i$. Write 
	\begin{equation*}
		\varphi: \BC \aro \BN,\quad
		i \mapsto (\text{the multiplicity of the eigenvalue $p^i$ of $\Lambda(\ff)$ on $\BC^m$}),
	\end{equation*}
	called the \textbf{weight function of $\Lambda$} or simply a \textbf{weight function}. Write 
	\begin{equation*}
		V_{\varphi,i} = \text{$\Lambda(\ff)$-eigenspace with eigenvalue $p^i$} \subset \BC^m
	\end{equation*}
	and let $V_\varphi$ be the graded vector space
	\begin{equation*}
		V_\varphi = \bigoplus_{i \in \BC} V_{\varphi,i} [-i]
	\end{equation*}
	(i.e. $V_{\varphi,i}$ is in degree $i$) with $\BC^m$ as the underlying space. Then, under the isomorphism $\fgl_m(\BC) \bijects \End(\BC^m)$, $E_\Lambda$ can be identified with degree $1$ endomorphisms of $V_\varphi$
	\begin{equation*}
		E_\Lambda \bijects E_\varphi := \bigoplus_{i \in \BC} \Hom_\BC(V_{\varphi,i},V_{\varphi,i+1})
	\end{equation*}
	and
	\begin{equation*}
		G_\Lambda \bijects G_\varphi := \bigoplus_{i \in \BC} \GL(V_{\varphi,i})
	\end{equation*}
	(see \cite{Zelevinskii} or \cite[Lemma 2.1]{Cunningham-Ray}) where $G_\varphi$ acts on $E_\varphi$ by conjugation, namely
	\begin{equation*}
		g \cdot T = g \circ T \circ g\inv, \quad
		\forany g \in G_\varphi, T \in E_\varphi.
	\end{equation*}
\end{clause}

\begin{clause}[Multisegments and orbits in $E_\varphi$ {\cite{Zelevinskii}}]\label{cls:p-adic-closure-relations}
	We now recall the combinatorial description of the $G_\varphi$-orbit structure on $E_\varphi$, due to Zelevinskii. A \textbf{segment} is a subset in $\BC$ of the form
	\begin{equation*}
		[a,b] = \{a, a+1, a+2, \ldots, b-1, b\} \quad \text{ for some } a,b \in \BC \text{ with } b-a \in \BZ_{\ge 0}.
	\end{equation*}
	We use the convention that $[a] := [a,a] = \{a\}$. A \textbf{multisegment} is a multiset of segments, or equivalently a $\BZ_{\ge 0}$-linear combination of segments. The \textbf{weight function} of a multisegment $\bm$ is a function $\varphi: \BC \to \BN$ so that $\varphi(i)$ is the number of times $i$ appears in $\bm$. More generally, a \textbf{weight function} is a function $\varphi: \BC \to \BN$ with finite support.
	
	Two segments $[a,b]$ and $[c,d]$ are said to be \textbf{linked} if neither contains the other, and $[a,b] \cup [c,d]$ (union as subsets in $\BC$) is again a segment. Suppose $\bm$ is a multisegment containing two linked segments $[a,b]$ and $[c,d]$, i.e. $\bm = [a,b] + [c,d] + \bm'$ for some multisegment $\bm'$. Let $\bn := [a,b] \cup [c,d] + [a,b] \cap [c,d] + \bm'$. We say $\bn$ is obtained from $\bm$ by an \textbf{elementary operation}.
	
	Given a multisegment $\bm = \sum_j c_j [a_j, b_j]$ with weight $\varphi$, we define a subset 
	\begin{equation*}
		O_\bm \subset E_\varphi
	\end{equation*}
	as follows. Since $T$ is nilpotent and homogeneous of degree one, there exist a homogeneous basis of $V_\varphi$ under which the matrix of $T$ is in Jordan normal form. This means that each basis vector $v_i \in V_{\varphi,i}$ is a part of a unique chain of the form $0 \xmapsto{T} v_a \xmapsto{T} v_{a+1} \xmapsto{T} \cdots \xmapsto{T} v_b \xmapsto{T} 0$. Such a chain is called a \textit{Jordan cell}. Then $O_\bm$ consists of those operators $T \in E_\varphi$ that have exactly $c_j$ Jordan cells of the form $0 \xmapsto{T} v_{a_j} \xmapsto{T} \cdots \xmapsto{T} v_{b_j} \xmapsto{T} 0$.
	
	\begin{namedtheorem}[Orbit structure on $E_\varphi$]
		There is a bijection
		\begin{equation*}
			G_\varphi \backslash E_\varphi \bijects \big\{\text{multisegments of weight } \varphi \big\},\quad
			O_\bm \,\mapsfrom\; \bm.
		\end{equation*}
		For multisegments $\bm$ and $\bn$ of weight $\varphi$, $\overline{O}_\fn$ contains $O_\bm$ if and only if there is a sequence of elementary operations transforming $\bm$ into $\bn$ \cite[Proposition-Definition 1.8, Theorem 2.2]{Zelevinskii}.
	\end{namedtheorem}
	
	\begin{example}\label{ex:elem-op}
		Let $\bn = [-1,3]+[-1]+[0,2]+[0,1]+[1]$ and $\bm = [-1,3]+[-1]+[1,2]+2[0,1]$. Pictorially,
		\begin{equation*}
			\bn = 
			\begin{tikzcd}[start anchor = real center, end anchor = real center, row sep=1ex, column sep=1ex]
				{\makebox[0pt]{$-1$}} & {\makebox[0pt]{$0$}} &{\makebox[0pt]{$1$}} & {\makebox[0pt]{$2$}} & {\makebox[0pt]{$3$}}\\
				\bullet \ar[r, dash] & \bullet \ar[r, dash] & \bullet \ar[r, dash] & \bullet \ar[r, dash] & \bullet \\
				\bullet & \bullet \ar[r,dash] & \bullet \ar[r, dash] & \bullet\\
				& & \bullet \\
				& \bullet \ar[r, dash] & \bullet
			\end{tikzcd},
			\quad
			\bm = 
			\begin{tikzcd}[start anchor = real center, end anchor = real center, row sep=1ex, column sep=1ex]
				{\makebox[0pt]{$-1$}} & {\makebox[0pt]{$0$}} &{\makebox[0pt]{$1$}} & {\makebox[0pt]{$2$}} & {\makebox[0pt]{$3$}}\\
				\bullet \ar[r, dash] & \bullet \ar[r, dash] & \bullet \ar[r, dash] & \bullet \ar[r, dash] & \bullet \\
				\bullet &  & \bullet \ar[r, dash] & \bullet\\
				& \bullet \ar[r, dash] & \bullet \\
				& \bullet \ar[r, dash] & \bullet
			\end{tikzcd}.
		\end{equation*}
		Here $\bm$ contains $[0,1]$ and $[1,2]$ which are linked, and $\bn$ is obtained from $\bm$ by replacing $[0,1]+[1,2]$ by $[0,2]+[1]$, an elementary operation. Hence $\overline{O_\bn} \supseteq O_\bm$.
	\end{example}	
\end{clause}

\begin{clause}[Representation theoretic notations for $\GL_m(\BQ_p)$, \cite{Zelevinskii:GLn}]
	For each $G_\varphi$-orbit $O_\bm$ in $E_\varphi$, there is a unique irreducible $G_\varphi$-equivariant perverse sheaf on $O_\bm$, namely $\ubar \BC_{O_\bm}[\dim O_\bm]$. The $!*$-extension (resp. extension by zero) of $\ubar \BC_{O_\bm}[\dim O_\bm]$ to $E_\varphi$ in the six operation formalism will be denoted by $\cL_\bm$ (resp. $\cM_\bm$).
	They form two bases for the Grothendieck group $KD^b(E_\varphi, G_\varphi)$. 
	
	We now assign to each multisegment $\bm$ a standard representation and an irreducible representation of $\GL_m(\BQ_p)$. We assign to a segment $[a,b]$
	\begin{equation*}
		\St_{[a,b]} := 
		\begin{array}{l}
			\text{unique irreducible quotient of }\\
			\nInd_{\GL_1(F) \times \cdots \times \GL_1(F)}^{\GL_{b-a+1}(F)} \big( |\cdot|^a \boxtimes |\cdot|^{a+1} \boxtimes \cdots \boxtimes |\cdot|^b \big),
		\end{array}		
	\end{equation*}
	where $\nInd$ denotes the normalized induction.
	For a multisegment $\bm = \sum_{j=1}^k c_j [a_j,b_j]$, re-order the subscripts $j$ so that $\operatorname{Re}(a_j+b_j) \ge \operatorname{Re}(a_{j+1}+b_{j+1})$. We then assign to $\bm$ the standard representation 
	\begin{equation*}
		X_\bm :=
		\nInd_{\GL_{b_1-a_1+1}(F)^{c_1} \times \cdots \times \GL_{b_k-a_k+1}(F)^{c_k}}^{\GL_m(F)} \Big( \St_{[a_1,b_1]}{}^{\boxtimes \; c_1} \boxtimes \cdots \boxtimes \St_{[a_k,b_k]}{}^{\boxtimes \; c_k} \Big)
	\end{equation*}
	and write
	\begin{equation*}
		\bar X_\bm := \text{unique irreducible quotient of $X_\bm$}.
	\end{equation*}
\end{clause}

\begin{namedtheorem}[LLC for $\GL_m(\BQ_p)$]\label{prop:LLC-GLmQp}
	Let $\Lambda$ be an unramified infinitesimal character for $\GL_m(\BQ_p)$, let $\varphi$ be the corresponding weight function, and let $\bm$ be a multisegment of weight $\varphi$. Then the perfect pairing (\ref{eqn:LLC-Qp-pairing}) simplifies to
	\begin{equation*}
		K \Rep_\Lambda(\GL_m(\BQ_p)) \times K D^b(E_\varphi, G_\varphi) \aro \BZ,
	\end{equation*}
	under which $(-1)^{\dim O_\bm} \cM_\bm$ is dual to $X_\bm$ and $(-1)^{\dim O_\bm} \cL_\bm$ is dual to $\bar X_\bm$. In other words, under the pairing, 
	\begin{itemize}
		\item $\{X_\bm \mid \bm \text{ has weight } \varphi\}$ and $\{(-1)^{\dim O_\bm} \cM_\bm \mid \bm \text{ has weight } \varphi\}$ form dual bases, and
		\item $\{\bar X_\bm \mid \bm \text{ has weight } \varphi\}$ and $\{(-1)^{\dim O_\bm} \cL_\bm \mid \bm \text{ has weight } \varphi\}$ form dual bases.
	\end{itemize}
\end{namedtheorem}

\begin{clause}[Integral pieces of $\varphi$ and regularity]\label{cls:integral-pieces}
	Here is some notion that is convenient for reducing non-integral results to integral ones. A weight function $\varphi$ is said to be \textbf{integral} if its support is contained in $a + \BZ$ for some $a \in \BC$. If $\varphi$ is non-integral, there exists a unique way of writing 
	\begin{equation*}
		\varphi = \sum_s \varphi_s
	\end{equation*}
	so that
	\begin{itemize}
		\item each $\varphi_s$ is integral, and
		\item the number of $\varphi_s$ in the sum is minimum. 
	\end{itemize}
	Indeed, as $\supp \varphi$ is contained in a finite disjoint union of sets of the form $a_s + \BZ \subset \BC$, one may take $\varphi_s$ to be equal to $\varphi$ on $a_s +\BZ$ and zero elsewhere. We refer to $\varphi = \sum_s \varphi_s$ as the \textbf{decomposition of $\varphi$ into integral pieces}.
	
	As an immediate consequence of the closure relation described in \ref{cls:p-adic-closure-relations}, we have
	\begin{equation*}
		G_\varphi = \prod_s G_{\varphi_s} \text{ and } E_\varphi = \prod_s E_{\varphi_s}
	\end{equation*}
	if the $\varphi_s$ are the integral pieces of $\varphi$.
	
	We say $\varphi$ is \textbf{regular} if $|\varphi(a) - \varphi(a+1)| \le 1$ for any $a \in \BC$. Under the comparison of real and $p$-adic parameter spaces described in \textsection \ref{subsec:geom-comparison}, the real infinitesimal character $\Lambda_\BR$ is regular if and only if the weight function $\varphi$ of the corresponding $p$-adic infinitesimal character $\Lambda_p$ is regular.
\end{clause}

\subsection{Jacquet restrictions, BZ derivatives, and Lusztig's inductions}\label{subsec:Lusztig-ind}

In this subsection, we recall the comparison between Jacquet restriction and partial Bernstein-Zelevinskii derivatives on $\Rep(\GL_m(\BQ_p))$ with Lusztig's induction on the Vogan varieties.

For simplicity, let us write $\GL_m$ for $\GL_m(\BQ_p)$ in this subsection.

\begin{clause}[Jacquet restriction and BZ derivatives, {\cite{Deng:symm}}]
	Let $G$ be a split $p$-adic group, and let $P \subset G$ be a parabolic subgroup with Levi decomposition $P = L U$. The \textit{normalized Jacquet functor} from representations of $G$ to representations of $L$ is the functor of taking $U$-coinvariants twisted by half of the modular character:
	\begin{equation*}
		\rJ_L^G : \Rep(G) \aro \Rep(L),\quad 
		M \mapsto \big( M/ \{ u \cdot m - m \mid u \in U, m \in M \} \big) \otimes \delta_U^{\frac12}.
	\end{equation*}
	
	Now consider the parabolic subgroup $P_{i,m-i} = (\GL_i \times \GL_{m-i}) U \subset \GL_m$ containing the block-diagonal Levi $\GL_i \times \GL_{m-i}$ and the upper-triangular Borel. Then the Jacquet functor can be viewed as a map
	\begin{equation*}
		\rJ_{\GL_i \times \GL_{m-i}}^{\GL_m} : \Rep(\GL_m) \aro \Rep(\GL_i) \times \Rep(\GL_{m-i}).
	\end{equation*}	
	The idea of Bernstein-Zelevinskii derivative is to put all the Jacquet functors $\rJ_{\GL_i \times \GL_{m-i}}^{\GL_m}$ together as a whole. In more details, let 
	\begin{align*}
		\cR_m &:= 
		\begin{array}{l}
			\text{Grothendieck group of the category of finite length representations} \\
			\text{of $\GL_m$ whose composition factors are of the form } \bar X(\bm),\\
			\text{i.e. those with Iwahori fixed vectors},
		\end{array}\\
		\cR &:= \bigoplus_{m \ge 1} \cR_m.
	\end{align*}
	Then $\cR$ is a Hopf algebra whose multiplication given by normalized induction, namely
	\begin{equation*}
		m: \cR \otimes \cR \aro \cR,\quad
		M_i \otimes M_j \mapsto \nInd_{\GL_i \times \GL_j}^{\GL_{i+j}} \big( M_i \boxtimes M_j \big)
	\end{equation*}
	for any $M_i \in \cR_i$. Since Jacquet functors are exact, we obtain homomorphisms
	\begin{equation*}
		\rJ_{\GL_i \times \GL_{m-i}}^{\GL_m} : \cR_n \aro \cR_i \dotimes_\BZ \cR_{m-i}.
	\end{equation*}
	Putting all Jacquet functors together, we obtain
	\begin{equation*}
		\bc: \cR \aro \cR \otimes \cR,\quad
		M \mapsto \sum_{i=0}^m \rJ_{\GL_i, \GL_{m-i}}^{\GL_m} (M)
	\end{equation*}
	for any $\GL_m$-representation $M$.
	This gives the comultiplication on $\cR$. Zelevinskii's classification of representations of $\GL_m$ can be rephrased by saying that the $\St_{[a,b]}$'s (for segments $[a,b]$) form a set of algebraic generators for the polynomial ring $\cR$, and the $X_\fm$'s (for multisegments $\fm$) form a $\BZ$-basis of $\cR$. Hence, one may define characters on $\cR$ (i.e. algebra homomorphisms $\cR \to \BZ$) by specifying values on the $\St_{[a,b]}$'s. For any $i \in \BC$, set
	\begin{equation*}
		\phi_i : \cR \aro \BZ,\quad
		\St_{[a,b]} \mapsto \delta_{[i], [a,b]} \; (\text{Kronecker symbol}).
	\end{equation*}
	The left and right \textit{partial Bernstein-Zelevinskii (BZ) derivatives} are given by
	\begin{equation*}
		{}^i\sD = m \circ (\phi_i \otimes 1) \circ \bc : \cR \aro \cR,
	\end{equation*}
	\begin{equation*}
		\sD^i = m \circ (1 \otimes \phi_i) \circ \bc : \cR \aro \cR,
	\end{equation*}

	\begin{remark}
		The relation between the operators ${}^i\sD$, $\sD^i$ and the derivatives originally defined by Bernstein and Zelevinskii in \cite{Bernstein-Zelevinskii} can be found in \cite{Deng:symm}, see remark after Lemma 2.25.
	\end{remark}
	
	\begin{remark}
		On generators, ${}^i \sD$ is given by
		\begin{equation*}
			{}^i \sD(\St_{[a,b]}) =
			\begin{cases}
				\St_{[a,b]} & \text{if } b \neq i,\\
				\St_{[a,b]} + \St_{[a,b-1]} & \text{if } a < b = i,\\
				\St_{[a,b]} + 1 & \text{if } a = b = i.
			\end{cases}
		\end{equation*}		
		This can be obtained from applying the Zelevinskii dual to the formulas in \cite{Deng:symm} Definition 2.19 and Lemma 2.25. There is a similar formula for $\sD^i$.
	\end{remark}
\end{clause}

\begin{clause}[Lusztig's induction, {\cite{Lusztig:quiver}}]\label{cls:Lind}
	We continue to use notations from \textsection \ref{subsec:LLC/Qp}. Suppose $\varphi$, $\psi$, and $\psi'$ are weight functions $\BC \to \BN$ such that $\varphi = \psi + \psi'$. Let $V_\varphi$, $V_\psi$, and $V_{\psi'}$ be the corresponding graded vector spaces. 
	
	Consider the following spaces
	\begin{align*}
		E_{\varphi,\psi'} 
		&= \left\{
		\Big( 
		\begin{array}{c}
			\text{graded subspace}\\
			W \subset V_\varphi
		\end{array},
		T \in E_\varphi
		\Big)
		\mid 
		\begin{array}{l}
			W \cong V_{\psi'} \text{ as graded vector spaces},\\
			V_\varphi/W \cong V_\psi \text{ as graded vector spaces},\\
			W \text{ is $T$-stable}.
		\end{array}
		\right\}\\
		\widetilde{E_{\varphi,\psi'}} 
		&= \left\{
		\Big(
		\begin{array}{c}
			\text{graded short exact sequence}\\
			0 \to V_{\psi'} \xrightarrow{\iota} V_\varphi \xrightarrow{\eps} V_\psi \to 0
		\end{array},
		T \in E_\varphi \Big)
		\mid
		\text{the image $\iota(V_{\psi'})$ is $T$-stable}
		\right\}
	\end{align*}
	equipped with natural actions of $G_\varphi$ and of $G_\psi \times G_{\psi'} \times G_\varphi$, respectively. Explicitly,
	\begin{equation*}
		G_\varphi \acts E_{\varphi,\psi'} \text{ by }
		g \cdot (W,T) = ( g \cdot W, g \circ T \circ g\inv);
	\end{equation*}
	\begin{equation*}
		G_\psi \times G_{\psi'} \times G_\varphi \acts \widetilde{E_{\varphi,\psi'}} \text{ by }
		(h,h', g) \cdot \big( V_{\psi'} \xrightarrow{\iota} V_\varphi \xrightarrow{\eps} V_\psi, T \big)
		= \big( V_{\psi'} \xrightarrow{g \circ \iota \circ h'{}\inv} V_\varphi \xrightarrow{h \circ \eps \circ g\inv} V_\psi, g \circ T \circ g\inv \big).
	\end{equation*}
	
	We have natural surjections
	\begin{equation}\label{eqn:induction-diag}
		E_\psi \times E_{\psi'} \xtwoheadleftarrow{\beta} \widetilde{E_{\varphi,\psi'}} \xtwoheadrightarrow{\tilde p} 
		E_{\varphi,\psi'} \xtwoheadrightarrow{p}
		E_\varphi
	\end{equation}
	where $\tilde p$ and $p$ are forgetful maps, and
	\begin{equation*}
		\beta\big( V_{\psi'} \xrightarrow{\iota} V_\varphi \xrightarrow{\eps} V_\psi, T \big) 
		= ( \eps \circ \bar T \circ \eps\inv, \iota\inv \circ T|_{\im \iota} \circ \iota )
	\end{equation*}
	where $\bar T$ is the endomorphism on $V_\varphi/\im\iota$ induced by $T$.
	These maps 
	are equivariant with respect to the natural projections
	\begin{equation*}
		G_\psi \times G_{\psi'} \twoheadleftarrow G_\psi \times G_{\psi'} \times G_\varphi \twoheadrightarrow G_\varphi = G_\varphi,
	\end{equation*}	
	so we obtain maps on quotient stacks
	\begin{equation*}
		[(G_\psi \times G_{\psi'}) \backslash (E_\psi \times E_{\psi'})] 
		\xleftarrow{\beta} 
		[(G_\psi \times G_{\psi'}) \backslash \widetilde{E_{\varphi,\psi'}} / G_\varphi ]
		\xrightarrow{\tilde p}
		[E_{\varphi,\psi'}/G_\varphi]
		\xrightarrow{p}
		[E_\varphi/ G_\varphi].
	\end{equation*}
	It is easy to see that $\tilde p$ is a $G_\psi \times G_{\psi'}$-principal bundle, so that the map 
	$[(G_\psi \times G_{\psi'}) \backslash \widetilde{E_{\varphi,\psi'}} / G_\varphi ]
	\to
	[E_{\varphi,\psi'}/G_\varphi]$
	is an isomorphism. We hence have functors
	\begin{equation*}
		D^b(E_\psi \times E_{\psi'}, G_\psi \times G_{\psi'}) 
		\xrightarrow{\beta^*}
		D^b(\widetilde{E_{\varphi,\psi'}}, G_\psi \times G_{\psi'} \times G_\varphi)
		\xleftarrow[\cong]{\tilde p^* } 
		D^b(E_{\varphi,\psi'}, G_\varphi)
		\xrightarrow{p_*}
		D^b(E_\varphi, G_\varphi)
	\end{equation*}
	
	\begin{definition}
		\textbf{Lusztig's induction} is the composition of the above functors
		\begin{equation*}
			\LInd_{\psi,\psi'}^\varphi = p_* \circ (\tilde p^*)\inv \circ \beta^*:
			D^b(E_\psi \times E_{\psi'}, G_\psi \times G_{\psi'}) \aro
			D^b(E_\varphi, G_\varphi).
		\end{equation*}
	\end{definition}
	
	Similar to the study of BZ derivatives, we can collect Lusztig inductions together. Let
	\begin{equation*}
		\cR^* = 
		\prod_{\substack{
				\varphi: \BC \to \BN\\
				\varphi \text{ comes from an}\\
				\text{infinitesimal charater}}} 
		K D^b( E_\varphi, G_\varphi)
	\end{equation*}
	whose elements are written additively as (possibly infinite) $\BZ$-linear combinations of elements of $K D^b( E_\varphi, G_\varphi)$. The pairing in \ref{prop:LLC-GLmQp} induces a pairing
	\begin{equation}\label{eqn:LLC-GLmQp-pairing-sum}
		\langle -,- \rangle: \cR \times \cR^* \aro \BZ.
	\end{equation}
	For each $k \in \BC$, by abuse of notation we write $[k]: \BC \to \{0,1\}$ for the characteristic function at $k$, and for any integer $c \in \BN$, we write $c[k]$ for $c$ times the function $[k]$. We then set
	\begin{equation*}
		\LInd_{\BN[k]} : \cR^* \aro \cR^*,\quad
		\cF_\psi \mapsto \sum_{c \ge 0} \LInd_{\psi, c[k]}^{\psi + c[k]} \big( \cF_\psi \boxtimes \ubar \BC \big),
	\end{equation*}
	\begin{equation*}
		{}_{\BN[k]}\!\LInd : \cR^* \aro \cR^*,\quad
		\cF_\psi \mapsto \sum_{c \ge 0} \LInd_{c[k], \psi}^{\psi + c[k]} \big( \ubar \BC \boxtimes \cF_\psi \big)
	\end{equation*}
	for any $\cF_\psi \in K D^b(E_\psi, G_\psi)$. 
\end{clause}

It is shown in \cite[Theorem 4.8]{Deng:BZ} that partial BZ derivatives and Lusztig's inductions are adjoint to each other under Langlands duality \ref{prop:LLC-GLmQp}. More precisely, we have the following theorem

\begin{theorem}
	Let $k \in \BC$. The operators ${}^k \sD$ and $\LInd_{\BN[k]}$ (resp. $\sD^k$ and ${}_{\BN[k]}\!\LInd$) are adjoint under the pairing (\ref{eqn:LLC-GLmQp-pairing-sum}), that is
	\begin{equation*}
		\langle {}^k\sD M, \cF \rangle = \langle M, \LInd_{\BN[k]} \cF \rangle
		\quad (\text{resp. } \langle \sD^k M, \cF \rangle = \langle M, {}_{\BN[k]}\!\LInd \cF \rangle )
	\end{equation*}
	for any $M \in \cR$ and $\cF \in \cR^*$.
\end{theorem}

\begin{proof}
	Since the shape of the statement is different from the one in \cite[Theorem 4.8]{Deng:BZ}, we explain how to make the transition. Since we are working over Grothendieck groups, it suffices to prove the statement for basis elements. On the representation side, we take the basis formed by standard representations $X_\ba$, and on the parameter side, we take the basis given by simple perverse sheaves $\cL_\bb$, where $\ba$ and $\bb$ are any multisegments. We need to prove the equality
	\begin{equation}\label{eqn:der-vs-LInd-a}
		\langle {}^k\sD X_\ba, \cL_\bb \rangle = \langle X_\ba, \LInd_{\BN[k]} \cL_\bb \rangle.
	\end{equation}
	
	Recall that for any object $\cF \in K D^b(E_\varphi, G_\varphi)$ and $M$ arbitrary, the pairing $\langle M, \cF\rangle$ is nonzero only if $M$ has a nonzero summand in $K\Rep_\Lambda(\GL_m(\BQ_p))$, where $\Lambda$ is the infinitesimal character corresponding to $\varphi$. Similarly, for $M \in K\Rep_\Lambda(\GL_m(\BQ_p))$ and $\cF$ arbitrary, $\langle M, \cF\rangle$ is nonzero only if $\cF$ has a nonzero summand in $K D^b(E_\varphi, G_\varphi)$. Therefore, if we let $\varphi$ and $\psi$ be the weight functions of $\ba$ and $\bb$, respectively, then $\langle X_\ba, \LInd_{\BN[k]} \cL_\bb \rangle$ is nonzero if and only if $\varphi = \psi + c[k]$ for some $c \in \BZ_{\ge 0}$, if and only if $\langle {}^k\sD X_\ba, \cL_\bb \rangle$ is nonzero. Therefore, it suffices to consider the case when $\varphi = \psi + c[k]$ for some $c \in \BZ_{\ge 0}$, where (\ref{eqn:der-vs-LInd-a}) reduces to
	\begin{equation}\label{eqn:der-vs-LInd-c}
		\langle {}^k\sD X_\ba, \cL_\bb \rangle = \langle X_\ba, \LInd_{\psi,c[k]}^\varphi \cL_\bb \rangle.
	\end{equation}
	
	To proceed, let us introduce more notations. For any segment $[a,b]$ and any multisegment $\bm = \sum_{j=1}^k c_j[a_j,b_j]$, order the segments so that $\operatorname{Re}(a_j+b_j) \ge \operatorname{Re}(a_{j+1}+b_{j+1})$, and let 
	\begin{gather*}
		L_{[a,b]} := 
		\begin{array}{l}
			\text{unique irreducible subrepresentation of }\\
			\nInd_{\GL_1(F) \times \cdots \times \GL_1(F)}^{\GL_{b-a+1}(F)} \big( |\cdot|^a \boxtimes |\cdot|^{a+1} \boxtimes \cdots \boxtimes |\cdot|^b \big),
		\end{array},
		\\
		\pi_\bm := \nInd_{\GL_{b_1-a_1+1}(F)^{c_1} \times \cdots \times \GL_{b_k-a_k+1}(F)^{c_k}}^{\GL_m(F)} \Big( L_{[a_1,b_1]}{}^{\boxtimes \; c_1} \boxtimes \cdots \boxtimes L_{[a_k,b_k]}{}^{\boxtimes \; c_k} \Big),
		\\
		L_\bm := \text{unique irreducible subrepresentation of } \pi_\bm.
	\end{gather*}
	Let $\BD$ denote the Zelevinskii duality \cite{Zelevinskii:GLn}. Then at the level of Grothendieck groups we have $\BD(L_{[a,b]}) = \St_{[a,b]}$, $\BD(\pi_\bm) = X_\bm$ and $\BD(L_\bm) = \bar X_\bm$. Moreover, it is straightforward to check that $\BD \sD^k \BD = {}^k \sD$. In the situation where $\varphi = \psi + c[k]$, \cite[Theorem 4.8]{Deng:BZ} shows that
	\begin{equation}\label{eqn:der-vs-LInd-b}
		n(\bb,\ba) =  \sum_i (-1)^i\rank \cH^i\Big( \LInd_{\psi, c[k]}^\varphi \cL_\bb [\dim O_\bb] \Big)|_{O_\ba}.
	\end{equation}
	Here $n(\bb,\ba)$ is the multiplicity of $L_\bb$ inside $\sD^k \pi_\ba$ (see \cite[Proposition 2.6]{Deng:BZ}). By applying $\BD$, $n(\bb,\ba)$ is also the multiplicity of $\bar X_\bb$ inside ${}^k\sD X_\ba$, or in our notations,
	\begin{equation*}
		n(\bb,\ba) = \langle {}^k\sD X_\ba, (\bar X_\bb)^* \rangle = \langle {}^k\sD X_\ba, (-1)^{\dim O_\bb} \cL_\bb \rangle.
	\end{equation*}
	On the other hand, the right hand side of (\ref{eqn:der-vs-LInd-b}) is the Euler characteristic of the pullback $i_a^* \LInd_{\psi, c[k]}^\varphi \cL_\bb [\dim O_\bb]$ to the orbit $O_\ba$. It is equal to multiplicity of $(-1)^{\dim O_\ba}\cM_\ba$ in $\LInd_{\psi, c[k]}^\varphi \cL_\bb [\dim O_\bb]$ in the Grothendieck group (see, for example, \cite[(7.11c), (7.11d)]{ABV}), which is equal to the number 
	\begin{equation*}
		\langle (-1)^{\dim O_\ba} (\cM_\ba)^*, \LInd_{\psi, c[k]}^\varphi \cL_\bb [\dim O_\bb] \rangle = \langle  X_\ba, (-1)^{\dim O_\bb} \LInd_{\psi, c[k]}^\varphi \cL_\bb \rangle,
	\end{equation*}
	Therefore (\ref{eqn:der-vs-LInd-b}) translates to (\ref{eqn:der-vs-LInd-c}), as required.
	
	The statement for $\sD^k$ and ${}_{\BN[k]}\!\LInd$ is analogous.
\end{proof}

\section{Comparisons of Langlands parameter spaces}\label{sec:compare-spaces}

In this section, we review two constructions that compare representations of $\GL_n(\BC)$ and of $\GL_m(\BQ_p)$, one is due to the first named author \cite{Deng:symm}, one due to Chan-Wong \cite{Chan-Wong:GLn}. The goal of this section is to show that these two constructions are compatible.

\subsection{Geometric comparison}\label{subsec:geom-comparison}

The geometric comparison of \cite{Deng:symm} is, in short, an open immersion from the ABV space $[\cY(\Lambda_\BR)/ (\GL_n(\BC)_\BC)_{\Lambda_\BR}]$ into the Vogan variety $[E_{\Lambda_p}/(\GL_m)_{\Lambda_p}]$ for certain compatible pairs $(n,\Lambda_\BR)$ and $(m, \Lambda_p)$. 

Let $\Lambda_p = \Lambda$ be an infinitesimal character for $\GL_m(\BQ_p)$ of the form described in \ref{cls:ELambda-as-Evarphi}. Recall from the discussions in  \ref{cls:ELambda-as-Evarphi}-\ref{cls:p-adic-closure-relations} that $\Lambda$ can be characterized by a weight function $\varphi$ that records the multiplicities of eigenvalues of $\Lambda(\ff)$, so that $E_\Lambda = E_\varphi$ and $(\GL_m)_\Lambda = G_\varphi$. In this subsection, we make the following assumption on $\Lambda$. 

\begin{namedtheorem}[Assumption]\label{assump:r}
	Let $\varphi = \sum_s \varphi_s$ be the decomposition of $\varphi$ into integral pieces (\ref{cls:integral-pieces}), and assume the support of $\varphi_s$ is contained in $a_s + \BZ$. For each $s$, there exists a number $r_s \in (a_s + \BQ) - (a_s + \BZ)$ such that 
	\begin{enumerate}
		\item $\varphi_s$ is weakly increasing on $(-\infty, \ceil {r_s}]$, i.e. $\varphi_s(i) \le \varphi_s(j)$ whenever $i \le j \le \ceil {r_s}$; and
		
		\item $\varphi_s$ is weakly decreasing on $[\floor {r_s}, \infty)$, i.e. $\varphi_s(i) \ge \varphi_s(j)$ whenever $\floor {r_s} \ge i \ge j$.
	\end{enumerate}
	In particular, $\varphi_s(\floor {r_s}) = \varphi_s(\ceil{r_s}) = \max_{i \in \supp \varphi_s} \varphi_s(i)$. Here $\floor {r_s}$ (resp. $\ceil {r_s}$) denotes the $a_s$-shifted floor function (resp. ceiling function), see \textsection \ref{subsec:gen-notn}. 
\end{namedtheorem}

These $\varphi$'s are exactly those that appear in the geometric comparison between Langlands parameter spaces of $\GL_n(\BC)$ and $\GL_m(\BQ_p)$. For an integral weight function $\varphi$, it satisfies Assumption \ref{assump:r} if and only if the open $G_\varphi$-orbit $O_\bm$ in $E_\varphi$ corresponds to a multisegment $\bm$ in which the segments are nested (i.e. for any two segments $[a,b]$, $[c,d]$ appearing in $\bm$, either $[a,b] \subseteq [c,d]$ or $[a,b] \supseteq [c,d]$), and the shortest segment has length $\ge 1$. 

\begin{example}
	Let $\bm = [-1,4] + [-1,3]+[0,2]+[1,2]$, or pictorially
	\begin{equation*}
		\bm =
		\begin{tikzcd}[start anchor = real center, end anchor = real center, row sep=1ex, column sep=1ex]
			{\makebox[0pt]{$-1$}} & {\makebox[0pt]{$0$}} &{\makebox[0pt]{$1$}} & {\makebox[0pt]{$2$}} & {\makebox[0pt]{$3$}} & {\makebox[0pt]{$4$}}\\
			\bullet \ar[r, dash] & \bullet \ar[r, dash] & \bullet \ar[r, dash] & \bullet \ar[r, dash] & \bullet \ar[r,dash] & \bullet \\
			\bullet \ar[r,dash] & \bullet \ar[r,dash] & \bullet \ar[r, dash] & \bullet \ar[r,dash] & \bullet\\
			& \bullet \ar[r, dash] & \bullet \ar[r,dash] & \bullet\\
			& & \bullet \ar[r,dash] & \bullet
		\end{tikzcd}.
	\end{equation*}
	Then the segments in $\bm$ are nested, and the shortest segment in $\bm$ is $[1,2]$ which has length $1$. The weight function $\varphi$ of $\bm$ satisfies Assumption \ref{assump:r} with $r = 3/2$.
\end{example}

Next, we introduce an open subset $\BO_\varphi$ of $E_\varphi$, the \textit{full rank part}, which will later be the image of the open immersion $[E_{\Lambda_p}/(\GL_m)_{\Lambda_p}] \hookleftarrow [\cY(\Lambda_\BR)/ (\GL_n(\BC)_\BC)_{\Lambda_\BR}]$.

\begin{clause}[The full rank part $\BO_\varphi \subset E_\varphi$]\label{cls:symm-part}
	\begin{definition}
		Suppose $\varphi$ is integral, $\supp \varphi \subset a + \BZ$, and $\varphi$ satisfies Assumption \ref{assump:r}. Write $T_i$ for the restriction of $T$ to $V_{\varphi,i}$. The  \textbf{full rank part} of $E_\varphi$ is 
		\begin{align*}
			\BO_\varphi 
			&:= \big\{ T \in E_\varphi \mid \text{each } T_i \in \Hom(V_{\varphi,i},V_{\varphi,i+1}) \text{ has maximal possible rank} \big\}\\
			&= \left\{ T \in E_\varphi \mid 
			\substack{%
				\displaystyle \forany i \le \ceil r, T_i \text{ is injective};\\
				\displaystyle \forany i \ge \floor r, T_i \text{ is surjective}}
			\right\} 
			\subset E_\varphi.
		\end{align*}
		Here the number $r$ is defined in Assumption \ref{assump:r}.
		
		If $\varphi$ satisfies Assumption \ref{assump:r} but not necessarily integral, let $\varphi = \sum_s \varphi_s$ be its decomposition into integral pieces. The \textbf{full rank part} of $E_\varphi$ is defined to be
		\begin{equation*}
			\BO_\varphi := \prod_s \BO_{\varphi_s}.
		\end{equation*}
	\end{definition}
\end{clause}

The full rank part $\BO_\varphi$ is a $G_\varphi$-stable subvariety in $E_\varphi$. In particular it is a union of finitely many orbits. If $\varphi$ is integral, an orbit $O_\bm \subset E_\varphi$ is contained in $\BO_\varphi$ if and only if there are $n$ segments in $\bm$ where $n = \varphi(\floor r) = \varphi(\ceil r) = \max_{i \in \BC} \varphi(i)$. It is not difficult to see from the closure relation \ref{cls:p-adic-closure-relations} that $\BO_\varphi$ is open in $E_\varphi$. 

When $\varphi$ is a regular integral infinitesimal character (see \ref{cls:integral-pieces}), the full rank part is called the \textit{symmetric part} in \cite{Deng:symm} since in this case there is an isomorphism of posets $G_\varphi \backslash \BO_\varphi \cong S_n$ (see Theorem \ref{thm:symm-orbit-structure} below), where the left side is equipped with the closure order of orbits and the right side is the symmetric group and is equipped with the Bruhat order.

Now consider 
\begin{align*}
	\ubar E_{\varphi_s} 
	&:= \bigoplus_{\substack{i \in \supp \varphi_s \\ i \neq \floor {r_s}}} \Hom_\BC(V_{\varphi,i},V_{\varphi,i+1}),\\
	\ubar E_\varphi
	&:= \prod_s \ubar E_{\varphi_s},
\end{align*}
and the canonical projections $E_{\varphi_s} \surj \ubar E_{\varphi_s}$, $E_\varphi \surj \ubar E_\varphi$, $T \mapsto \ubar T$ by forgetting the component $T_{\floor {r_s}} \in \Hom_\BC(V_{\varphi,\floor {r_s}}, V_{\varphi,\ceil {r_s}})$. Let $\ubar \BO_{\varphi_s}$, $\ubar \BO_\varphi$ denote the images of $\BO_{\varphi_s}$, $\BO_\varphi$ under this projection. For any element $\ubar T \in \ubar \BO_\varphi$, we write $\BO_\varphi|_{\ubar T}$ for its preimage along the projection $\BO_\varphi \surj \ubar \BO_\varphi$.

The following theorems outline the key properties that enable the construction of the desired immersion. Write $G_{\varphi, \ubar T}$ for the stabilizer of $\ubar T$, and write $G_{\varphi, \ubar T}^l$ and $G_{\varphi, \ubar T}^r$ for the images of the projection $G_\varphi \surj \GL(V_{\varphi, \floor r})$ and $G_\varphi \surj \GL(V_{\varphi, \ceil r})$, respectively (recall that $G_\varphi = \prod_{i \in \supp \varphi} \GL(V_{\varphi,i})$ is equipped with natural projections to each $\GL(V_{\varphi,i})$).

\begin{theorem}[Deng]\label{thm:symm-orbit-structure}
	Suppose $\varphi$ is integral and satisfies Assumption \ref{assump:r}.
	\begin{enumerate}
		\item $\ubar \BO_\varphi$ is a single $G_\varphi$-orbit in $\ubar E_\varphi$.
		
		\item For any point $\ubar T \in \ubar \BO_\varphi$, the fiber $\BO_\varphi|_{\ubar T}$ is canonically isomorphic to $\GL(V_{\varphi,\floor r}, V_{\varphi,\ceil r})$, the set of full rank elements in $\Hom_\BC(V_{\varphi,\floor r}, V_{\varphi, \ceil r})$.
		
		\item 
		$G_{\varphi, \ubar T}^l$ is a parabolic subgroup of $\GL(V_{\varphi, \floor r})$ equal to the stabilizer of the flag
		\begin{equation*}
			V_{\varphi, \floor r} \supset \ubar T (V_{\varphi, \floor r-1}) \supset \ubar T^2 (V_{\varphi, \floor r-2}) \supset \cdots \supset 0.
		\end{equation*}
		The Levi of $G_{\varphi, \ubar T}^l$ is isomorphic to 
		\begin{equation*}
			\cdots \times \GL_{\varphi(\floor r-1) - \varphi(\floor r -2)} \times \GL_{\varphi(\floor r) - \varphi(\floor r -1)}
		\end{equation*}
		(when $\varphi(i) = \varphi(i-1)$, we ignore the factor $\GL_{\varphi(i) - \varphi(i-1)}$). Similarly, 
		$G_{\varphi, \ubar T}^r$ is a parabolic subgroup of $\GL(V_{\varphi, \ceil r})$ equal to the stabilizer of the flag 
		\begin{equation*}
			0 \subset \ker \ubar T|_{V_{\varphi, \ceil r}} \subset \ker \ubar T^2|_{V_{\varphi, \ceil r}} \subset \cdots \subset V_{\varphi,\ceil r}.
		\end{equation*}
		The Levi of $G_{\varphi, \ubar T}^r$ is isomorphic to 
		\begin{equation*}
			\GL_{\varphi(\ceil r) - \varphi(\ceil r+1)} \times \GL_{\varphi(\ceil r +1) - \varphi(\ceil r+2)} \times \cdots.
		\end{equation*}
		
		The product of the two projections
		\begin{equation*}
			G_{\varphi, \ubar T} \surjects G_{\varphi, \ubar T}^l \times G_{\varphi, \ubar T}^r
		\end{equation*}
		is an isomorphism.
		
		\item We have a commutative diagram
		\begin{equation*}
			\begin{tikzcd}
				G_{\varphi, \ubar T} \ar[r, phantom, "\acts" description] \ar[d, "\cong"]
				& \BO_\varphi|_{\ubar T} \ar[d, "\cong"]\\
				G_{\varphi, \ubar T}^l \times G_{\varphi, \ubar T}^r \ar[r, phantom, "\acts"]
				& \GL(V_{\varphi, \floor r}, V_{\varphi, \ceil r})
			\end{tikzcd}
		\end{equation*}
		and hence an isomorphism  
		\begin{equation*}
			\big[G_\varphi \backslash \BO_\varphi \big] =
			\big[ G_{\varphi, \ubar T} \backslash (\BO_\varphi|_{\ubar T}) \big] \cong 
			\big[ G_{\varphi, \ubar T}^r \backslash \GL(V_{\varphi, \floor r}, V_{\varphi, \ceil r}) / G_{\varphi, \ubar T}^l \big].
		\end{equation*}
	\end{enumerate}
	
	If $\varphi$ is non-integral, let $\varphi = \sum_s \varphi_s$ be its decomposition into integral pieces (\ref{cls:integral-pieces}). Then the above properties hold for each $\varphi_s$.
\end{theorem}

Note that the parabolic $G_{\varphi, \ubar T}^l$ labeled by the superscript ``\textit{left}'' actually acts on $\GL(V_{\varphi, \floor r}, V_{\varphi, \ceil r})$ on the \textit{right} since it acts by twisting maps $V_{\varphi, \floor r} \to V_{\varphi, \ceil r}$ on the source. Similarly the group $G_{\varphi, \ubar T}^r$ acts on the \textit{left} of $\GL(V_{\varphi, \floor r}, V_{\varphi, \ceil r})$.

If $\varphi$ is integral and regular, then the parabolic subgroups $G_{\varphi, \ubar T}^l$ and $G_{\varphi, \ubar T}^r$ are Borels. In this case the proof is contained in \cite[Theorem 4.12]{Deng:symm} (see also \cite[\S 10.7]{Kirillov:quiver} for a treatment in the language of quiver representations). We include an argument for integral $\varphi$ for completeness. The non-integral cases can be easily reduced to the integral case.

\begin{proof}[Proof of \ref{thm:symm-orbit-structure} for integral $\varphi$]~
	\begin{enumerate}
		\item We may decompose $\ubar E_\varphi = \ubar E_{\varphi, <r} \times \ubar E_{\varphi, >r}$, where $\ubar E_{\varphi,<r} = \prod_{i <r-1} \Hom_\BC(V_{\varphi,i}, V_{\varphi,i+1})$ and $\ubar E_{\varphi,>r} = \prod_{i >r} \Hom_\BC(V_{\varphi,i}, V_{\varphi,i+1})$. We have similar decompositions for $V_\varphi$, $\ubar \BO_\varphi$, and $G_\varphi$. Then the action $G_\varphi \acts \ubar \BO_\varphi$ is isomorphic to the product of $G_{\varphi,<r} \acts \ubar \BO_{\varphi,<r}$ with $G_{\varphi,>r} \acts \ubar \BO_{\varphi,>r}$. It is enough to show that the latter two actions are transitive. By taking $\BC$-linear dual, the $>r$ case is reduced to the $<r$ case.
		
		Let $i_0 \in \supp \varphi$ be the smallest number in $\supp \varphi$. Suppose $\ubar T, \ubar S \in \ubar \BO_{\varphi,<r}$ are two operators, each representing a chain of inclusions $V_{\varphi,i_0} \inj V_{\varphi,i_0+1} \inj \cdots \inj V_{\varphi,\floor r}$. Choose a homogeneous basis $\cE_i$ (viewed as an ordered set) of $V_{\varphi,i}$. Then $\ubar T(\cE_{i_0})$ and $\ubar S(\cE_{i_0})$ are two (possibly different) linearly independent subsets of $V_{\varphi,i_0+1}$, and so there exists an element $g_{i_0+1} \in \GL(V_{\varphi,i_0+1})$ so that $g_{i_0+1} (\ubar T( \cE_{i_0})) = \ubar S( \cE_{i_0})$. Similarly, $\ubar T( g_{i_0+1}\inv (\cE_{i_0+1}))$ and $\ubar S(\cE_{i_0+1})$ are two linearly independent subsets of $V_{\varphi, i_0+2}$, and there exists $g_{i_0+2} \in \GL(V_{\varphi,i_0+2})$ so that $g_{i_0+2}( \ubar T( g_{i_0+1}\inv (\cE_{i_0+1}))) = \ubar S(\cE_{i_0+1})$. Repeating this process, we obtain an element $g = (1,g_{i_0+1}, g_{i_0+2},\ldots, g_{\floor r}) \in G_{\varphi,<r}$ so that $g \circ \ubar T \circ g\inv$ and $\ubar S$ agree on the bases $\cE_i$'s. This implies $g \circ \ubar T \circ g\inv = \ubar S$, and (1) is proved.
		
		\item By construction, the fiber of $E_\varphi \surj \ubar E_\varphi$ is equal to $\{\ubar T\} \times \Hom_\BC(V_{\varphi,\floor r}, V_{\varphi, \ceil r})$. In other words, different elements $T \in E_\varphi$ in the fiber of $\ubar T$ can only be different at $T_{\floor r} : V_{\varphi, \floor r} \to V_{\varphi, \ceil r}$ since the other degrees are already determined by $\ubar T$. By definition, such a $T$ has full rank if and only if $T_{\floor r}$ is both injective and surjective, i.e. is an isomorphism. So the fiber $\BO_\varphi|_{\ubar T}$ is equal to $\{\ubar T\} \times \GL(V_{\varphi,\floor r}, V_{\varphi, \ceil r})$.
		
		\item We first choose a basis of $V$ that plays well with $\ubar T$. 
		
		\begin{clause}[Basis adapted to $\ubar T$]\label{cls:adapted-basis}
			For a fixed element $\ubar T$, we fix a basis $\cE_i = \{v_i^1,\ldots, v_i^{\varphi(i)}\}$ of $V_{\varphi,i}$ satisfying the following properties: 
			\begin{itemize}
				\item For $i <r$, we require $\ubar T(v_i^j) = v_{i+1}^j$ for any $j$, and
				\item for $i >r$, we require $\ubar T(v_i^j) = 
				\begin{cases}
					v_{i+1}^j & \text{if } j \le \varphi(i+1),\\
					0 & \text{if } j > \varphi(i+1).
				\end{cases}$.
			\end{itemize}
			This is possible since $\ubar T$ is injective on degrees $<r$ and surjective in degrees $>r$.
		\end{clause}
		
		Returning to the proof, we first consider the statement for $G_{\varphi, \ubar T}^l$.
		The matrix of $\ubar T_i := \ubar T|_{V_{\varphi,i}}$ in the above bases is 
		\begin{equation*}
			\ubar T_i = 
			\begin{pmatrix}
				I_{\varphi(i)}\\
				0
			\end{pmatrix}_{\varphi(i+1) \times \varphi(i)}
		\end{equation*}
		where $I_d$ stands for the identity matrix of size $d \times d$. Let $g = (\ldots, g_i, g_{i+1}, \ldots) \in G_{\varphi,\ubar T,<r}$ (the subscript $<r$ is defined at the beginning of the proof of (1) above). By definition $g_{i+1} \circ \ubar T_i \circ g_i\inv = \ubar T_i$. Identify $g_i$ with its matrix in the above chosen basis, and write $g_{i+1}$ as a block matrix as
		\begin{equation*}
			g_{i+1} = 
			\begin{pmatrix}
				g_{i+1}^{11} & g_{i+1}^{12}\\
				g_{i+1}^{21} & g_{i+1}^{22}
			\end{pmatrix}, \quad
			g_{i+1}^{11} \text{ has size } \varphi(i) \times \varphi(i).
		\end{equation*}
		Then the equation $g_{i+1} \circ \ubar T_i \circ g_i\inv = \ubar T_i$ can be rewritten in matrix form as
		\begin{equation*}
			\begin{pmatrix}
				g_{i+1}^{11} g_i\inv\\
				g_{i+1}^{21} g_i\inv
			\end{pmatrix}
			=
			\begin{pmatrix}
				g_{i+1}^{11} & g_{i+1}^{12}\\
				g_{i+1}^{21} & g_{i+1}^{22}
			\end{pmatrix}
			\begin{pmatrix}
				I_{\varphi(i)}\\
				0
			\end{pmatrix}
			g_i\inv = 
			\begin{pmatrix}
				I_{\varphi(i)}\\
				0
			\end{pmatrix},
		\end{equation*}
		which implies $g_{i+1}^{21} = 0$ and both $g_{i+1}^{11} = g_i$ and $g_{i+1}^{22}$ are invertible. In other words, 
		\begin{equation*}
			g_{i+1} = 
			\begin{pmatrix}
				g_i & *\\
				0 & g_{i+1}'
			\end{pmatrix}
			\text{ for some } g_{i+1}' := g_{i+1}^{22} \in \GL_{\varphi(i+1)-\varphi(i)}.			
		\end{equation*}
		Running this argument inductively, we see that there exist $g_{i+1}' \in \GL_{\varphi(i+1) - \varphi(i)}$ so that
		\begin{equation*}
			g = \left(
			g_{i_0}',
			\begin{pmatrix}
				g_{i_0}' & *\\
				0 & g_{i_0+1}'
			\end{pmatrix},
			\begin{pmatrix}
				g_{i_0}' & * & *\\
				0 & g_{i_0+1}' & *\\
				0 & 0 & g_{i_0+2}'
			\end{pmatrix},
			\ldots,
			\begin{pmatrix}
				g_{i_0}' & * & \cdots & *\\
				0 & g_{i_0+1}' & \ddots & \vdots\\
				\vdots & \ddots & \ddots &  *\\
				0 & \cdots & 0 & g_{\floor r}'
			\end{pmatrix}
			\right)
		\end{equation*}
		where $i_0$ is the smallest number in $\supp \varphi$. Here if $\varphi(i) = \varphi(i+1)$, we delete the rows and the columns containing $g_{i+1}'$. Therefore, the map
		\begin{equation*}
			G_{\varphi, \ubar T, <r} \aro \GL(V_{\varphi,\floor r}),\quad
			g \mapsto 
			\begin{pmatrix}
				g_{i_0}' & * & \cdots & *\\
				0 & g_{i_0+1}' & \ddots & \vdots\\
				\vdots & \ddots & \ddots &  *\\
				0 & \cdots & 0 & g_{\floor r}'
			\end{pmatrix}
		\end{equation*}
		is injective whose image $G_{\varphi, \ubar T}^l$ is a parabolic subgroup whose Levi is $\prod_{i < r} \GL_{\varphi(i)- \varphi(i-1)}$. 
		
		The case for $G_{\varphi, \ubar T}^r$ is similar, except the matrices are now block \textit{lower} triangular.
		
		Since $G_{\varphi, \ubar T} = G_{\varphi, \ubar T, <r} \times G_{\varphi, \ubar T, >r}$, we see that the combination of the two projections $G_{\varphi, \ubar T} \to G_{\varphi, \ubar T}^l \times G_{\varphi, \ubar T}^r$ is an isomorphism.
		
		\item The two commutative diagrams can be checked straightforwardly, for example by using the matrix description in the proof of part (3). Then the isomorphism 
		$\big[ G_{\varphi, \ubar T} \backslash (\BO_\varphi|_{\ubar T}) \big] \cong 
		\big[ G_{\varphi, \ubar T}^r \backslash \GL(V_{\varphi, \floor r}, V_{\varphi, \ceil r}) / G_{\varphi, \ubar T}^l \big]$ follows from the first diagram. 
		
		For the identification 
		\begin{equation}\label{eqn:slice-in-full-rank-part}
			\big[G_\varphi \backslash \BO_\varphi \big] =
			\big[ G_{\varphi, \ubar T} \backslash (\BO_\varphi|_{\ubar T}) \big],
		\end{equation}
		recall the following general fact: if $f: X \to Y$ is a $G$-equivariant morphism of $G$-varieties and $Y$ is a homogeneous space for $G$, then for any point $y \in Y$ we have an isomorphism $X \cong G \times_{\operatorname{Stab}_G(y)} f\inv(y)$ so that the map $f$ is isomorphic to the projection $G \times_{\operatorname{Stab}_G(y)} f\inv(y) \surj G/\operatorname{Stab}_G(y) \cong Y$. As a result,
		\begin{equation*}
			[G \backslash X] \cong [G \backslash (G \times_{\operatorname{Stab}_G(y)} f\inv(y))] \cong [\operatorname{Stab}_G(y) \backslash f\inv(y)]
		\end{equation*}
		where the second isomorphism is the induction equivalence  \cite[Theorem 6.5.10, Exercise 6.8.1]{Achar:book}. Applying this to $G = G_\varphi$, $X = \BO_\varphi$, $Y = \ubar \BO_\varphi$ and $y = \ubar T$, we obtain (\ref{eqn:slice-in-full-rank-part}), as desired. \qedhere
	\end{enumerate}
\end{proof}

\begin{example}\label{eg:geom-comparison}
	Let $\varphi$ be the weight function of $\bm = [-1,4] + [-1,3]+[0,2]+[1,2]$ which can be depicted by 
	\begin{equation*}
		\bm =
		\begin{tikzcd}[start anchor = real center, end anchor = real center, row sep=1ex, column sep=1ex]
			{\makebox[0pt]{$-1$}} & {\makebox[0pt]{$0$}} &{\makebox[0pt]{$1$}} & {\makebox[0pt]{$2$}} & {\makebox[0pt]{$3$}} & {\makebox[0pt]{$4$}}\\
			\bullet \ar[r, dash] & \bullet \ar[r, dash] & \bullet \ar[r, dash] & \bullet \ar[r, dash] & \bullet \ar[r,dash] & \bullet \\
			\bullet \ar[r,dash] & \bullet \ar[r,dash] & \bullet \ar[r, dash] & \bullet \ar[r,dash] & \bullet\\
			& \bullet \ar[r, dash] & \bullet \ar[r,dash] & \bullet\\
			& & \bullet \ar[r,dash] & \bullet
		\end{tikzcd}.
	\end{equation*}
	Then under the basis \ref{cls:adapted-basis}, we have
	\begin{equation*}
		G_{\varphi, \ubar T}^l \cong P_{2,1,1} := \left\{
		\begin{pmatrix}
			* & * & * & *\\
			* & * & * & *\\
			&   & * & *\\
			&   &   & *
		\end{pmatrix} \in \GL_4 \right\},
		\quad
		G_{\varphi, \ubar T}^r \cong P_{1,1,2}^- := \left\{
		\begin{pmatrix}
			*\\
			* & *\\
			* & * & * & *\\
			* & * & * & *
		\end{pmatrix} \in \GL_4 \right\}.
	\end{equation*}
	Hence
	\begin{equation*}
		[G_\varphi \backslash \BO_\varphi] \cong [P_{1,1,2}^- \backslash \GL_4 / P_{2,1,1}].
	\end{equation*}
\end{example}

We now explain the connection with real groups. Assume $\varphi$ is integral and let $n = \varphi(\floor r) = \varphi(\ceil r)$. Each choice of basis for $V_{\varphi, \floor r}$ and $V_{\varphi, \ceil r}$ provides identifications of the groups $\GL(V_{\varphi, \floor r})$, $\GL(V_{\varphi, \ceil r})$ and the space $\GL(V_{\varphi, \floor r}, V_{\varphi, \ceil r})$ with $\GL_n$ so that the natural actions
\begin{equation*}
	\GL(V_{\varphi, \ceil r}) \acts \GL(V_{\varphi, \floor r}, V_{\varphi, \ceil r}) \racts \GL(V_{\varphi, \floor r})
\end{equation*}
are identified with the left and right multiplication action of $\GL_n$ on itself. The parabolic subgroups $G_{\varphi, \ubar T}^l \subset \GL(V_{\varphi, \floor r})$ and $G_{\varphi, \ubar T}^r \subset \GL(V_{\varphi, \ceil r})$ are identified with parabolic subgroups in $\GL_n$ denoted by $P_{\varphi,l}$ and $P_{\varphi,r}$. Write $\cP_{\varphi,l} = \GL_n/P_{\varphi,l}$, $\cP_{\varphi,r} = \GL_n/P_{\varphi,r}$ for the corresponding flag varieties. Then we have
\begin{equation*}
	[ G_{\varphi, \ubar T}^r \backslash \GL(V_{\varphi, \floor r}, V_{\varphi, \ceil r}) / G_{\varphi, \ubar T}^l ] 
	\cong 
	[ P_{\varphi,r} \backslash \GL_n /P_{\varphi,l} ].
\end{equation*}
Further, the stack on the right hand side can be written as
\begin{multline*}
	[ P_{\varphi,r} \backslash \GL_n /P_{\varphi,l} ]
	\cong
	[ \GL_n \backslash (\GL_n \times_{P_{\varphi,r}} (\GL_n/P_{\varphi,l}))]\\
	\cong [\Delta \GL_n \backslash (\GL_n/ P_{\varphi,r} \times \GL_n/P_{\varphi,l})]
	\cong
	[\Delta \GL_n \backslash (\cP_{\varphi,r} \times \cP_{\varphi,l})].
\end{multline*}
Here the third isomorphism is by the definitions of $\cP_{\varphi,l}$ and $\cP_{\varphi,r}$, and the first isomorphism is the induction equivalence. For the second isomorphism, consider the projection
\begin{equation*}
	p_2: \GL_n/ P_{\varphi,r} \times \GL_n/P_{\varphi,l} \surjects \GL_n/P_{\varphi,l}
\end{equation*}
which is $\GL_n$-equivariant for the diagonal action on the domain and the (transitive) left multiplication action on the codomain. By the general fact recalled in the proof of part (4) above, we have
\begin{equation*}
	\GL_n \times_{P_{\varphi,r}} (\GL_n/P_{\varphi,l}) \bijects
	\GL_n/ P_{\varphi,r} \times \GL_n/P_{\varphi,l},\quad
	(g_1, g_2 P_{\varphi,l}) \modulo P_{\varphi,r} 
	\mapsto (g_1 P_{\varphi,r}, g_1 g_2 P_{\varphi,l})
\end{equation*}
from which the second isomorphism follows.

Note that $[\Delta \GL_n \backslash (\cP_{\varphi,r} \times \cP_{\varphi,l})]$ is the ABV space for the integral infinitesimal character $\Lambda_\BR = (\lambda_L, \lambda_R)$ of $\GL_n(\BC)$ whenever $\cP_{\varphi,r} = \cP_{\lambda_L}$ and $\cP_{\varphi,l} = \cP_{\lambda_R}$. Combined with Theorem \ref{thm:symm-orbit-structure}(4), we obtain the desired open immersion from the ABV space into the Vogan variety, which we record as a corollary for future reference.

\begin{corollary}\label{cor:zeta}
	Assume $\varphi$ is integral and Assumption \ref{assump:r}. Let $n = \varphi(\lfloor r \rfloor) = \varphi(\lceil r \rceil)$. Then there exists an integral infinitesimal character $\Lambda_\BR = (\lambda_L, \lambda_R)$ of $\GL_n(\BC)$ and an open immersion
	\begin{equation}\label{eqn:zeta}
		\zeta: [\Delta \GL_n \backslash (\cP_{\lambda_L} \times \cP_{\lambda_R})] \cong [G_\varphi \backslash \BO_\varphi] \injects [G_\varphi \backslash E_\varphi].
	\end{equation}
\end{corollary}

In the non-integral case, we have a similar immersion
\begin{equation*}
	\zeta: [\Delta \GL_{n,\Lambda_\BR} \backslash (\cP_{\lambda_L} \times \cP_{\lambda_R})] \cong [G_\varphi \backslash \BO_\varphi] \injects [G_\varphi \backslash E_\varphi]
\end{equation*}
for some non-integral infinitesimal character $\Lambda_\BR$.

\begin{clause}[Consequences of $\zeta$]
	The open immersion $\zeta$ has a few formal but important consequences that are worth spelling out. For simplicity, we restrict to the case where $\varphi$ is integral. 
	\begin{enumerate}[label=(\alph*)]
		\item The map $\zeta$ induces an isomorphism of posets of orbits $\Delta \GL_n \backslash (\cP_{\lambda_L} \times \cP_{\lambda_R}) \cong G_\varphi \backslash \BO_\varphi$. In particular, we obtain an injection between conjugacy classes of Langlands parameters
		\begin{equation*}
			\Phi(\GL_n(\BC),\Lambda_\BR) \injects \Phi(\GL_m(\BQ_p),\Lambda),
		\end{equation*}
		where $\Lambda$ is the infinitesimal character that corresponds to the weight function $\varphi$.
		
		\item For each orbit $Q \in \Delta \GL_n \backslash (\cP_{\lambda_L} \times \cP_{\lambda_R})$ write $Q' \in G_\varphi \backslash \BO_\varphi$ for the corresponding orbit, and write $\phi$, $\phi'$ for the L-parameters. Then there is a bijection between simple $\GL_n$-equivariant local systems on $Q$ and simple $G_\varphi$-equivariant local systems on $Q'$. In particular, we obtain a bijection of L-packets
		\begin{equation*}
			\Pi_\phi(\GL_n(\BC)) \bijects \Pi_{\phi'}(\GL_m(\BQ_p))
		\end{equation*}		
		and an injection between conjugacy classes of complete Langlands parameters
		\begin{equation*}
			\Xi(\GL_n(\BC),\Lambda_\BR) \injects \Xi(\GL_m(\BQ_p),\Lambda),
		\end{equation*}
		extending the injection in part (a) above. Indeed, the isomorphism $[\Delta \GL_n \backslash (\cP_{\lambda_L} \times \cP_{\lambda_R})] \cong [G_\varphi \backslash \BO_\varphi]$ implies an isomorphism between component groups
		\begin{equation*}
			A_{\GL_n}(x) \cong A_{G_\varphi}(x')
		\end{equation*}
		whose irreducible representations parameterize the set of simple local systems on the respective orbits.
		
		In the current setting of general linear groups, (b) is rather trivial, since each orbit supports only one local system on both real and $p$-adic side. However, in the sequel to this paper, we will construct a similar open immersion for real unitary groups and $p$-adic symplectic or orthogonal groups, where (b) becomes meaningful. 
		
		\item Given a complete Langlands parameter $\xi \in \Xi(\GL_n(\BC),\Lambda_\BR)$, write $\xi'$ for the corresponding complete parameter on the $p$-adic side. Since $\zeta$ is an open immersion, a certain shift of the pullback functor $\zeta^*$
		sends the standard sheaf $\cM(\xi')$ and the simple perverse sheaf $\cL(\xi')$ to $\cM(\xi)$ and $\cL(\xi)$, respectively, and the pushforward $\zeta_!$, up to a shift, sends $\cM(\xi)$ to $\cM(\xi')$. If $\xi'' \in \Xi(\GL_m(\BQ_p),\Lambda)$ is a parameter not in the image of the injection of part (b), then $\zeta^*$ sends $\cM(\xi')$ and $\cL(\xi')$ to zero. As a consequence, the adjoint of $\zeta^*$ under the pairing (\ref{eqn:LLC-GLn-pairing}) sends standard modules to standard ones and irreducible modules to irreducible ones up to some signs.
		
		\item For complete parameters $\xi_1$ and $\xi_2$, the Kazhdan-Lusztig polynomials $P_{\xi_1,\xi_2}$ and $P_{\xi_1',\xi_2'}$ are equal. This is because the base change along $\zeta$ of the inclusion map $i_2: Q_2 \inj \cP_{\lambda_L} \times \cP_{\lambda_R}$ is the inclusion map $i_2': Q_2' \inj \BO_\varphi$, and vice versa. Therefore the pullback $i_2^*$ (resp. $i_2'{}^*$), which computes $P_{\xi_1,\xi_2}$ (resp. $P_{\xi_1',\xi_2'}$), can be computed using $i_2'{}^*$ (resp. $i_2^*$) instead.
		
		\item Let $mic_{Q_2}(\xi_1)$ denote the \textit{microlocal multiplicity} of $\xi_1$ at $Q_2$, i.e. the multiplicity of the conormal bundle to $Q_2$ in the characteristic cycle of the simple perverse sheaf $\cL(\xi_1)$ attached to $\xi_1$ (see \cite[Theorem 1.31]{ABV}). Then $mic_{Q_2}(\xi_1) = mic_{Q_2'}(\xi_1')$. This follows from \cite[Proposition 20.2]{ABV}. In particular, the ABV-packets $\Pi_{Q_2}^{ABV}$ and $\Pi_{Q_2'}^{ABV}$ match, where
		\begin{equation*}
			\Pi_{Q_2}^{ABV} = \{\xi_1 \in \Xi(\GL_n(\BC),\Lambda_\BR) \mid mic_{Q_2}(\xi_1) \neq 0\}
		\end{equation*}
		and similarly for $\Pi_{Q_2}^{ABV}$.
	\end{enumerate}
	
	We mention the inspiring work of Barchini-Trapa \cite{Barchini-Trapa}, in which the authors construct a map from the Vogan variety to the ABV space for the same group (in opposite direction than ours). It enjoys the properties (a), (d) and (e), but property (b) fails for them in general. 
	
	Finally, we remark that the geometric comparison theorem above is quite flexible, in the sense that for a fixed choice of $n$ and $\Lambda_\BR$, the isomorphism $[\Delta \GL_n \backslash (\cP_{\lambda_L} \times \cP_{\lambda_R})] \cong [G_\varphi \backslash \BO_\varphi]$ holds for infinitely many $\varphi$'s. Indeed, the isomorphism relies only on the increments of values of $\varphi$, and there are different $\varphi$'s with the same increments. For example, the ABV space $[P_{1,1,2}^- \backslash \GL_4 / P_{2,1,1}]$ in Example \ref{eg:geom-comparison} is also isomorphic to $[G_{\varphi'} \backslash \BO_{\varphi'}]$ where $\varphi'$ is the weight function of the multisegment
	\begin{equation*}
		\bm' =
		\begin{tikzcd}[start anchor = real center, end anchor = real center, row sep=1ex, column sep=1ex]
			{\makebox[0pt]{$-2$}} & {\makebox[0pt]{$-1$}} & {\makebox[0pt]{$0$}} &{\makebox[0pt]{$1$}} & {\makebox[0pt]{$2$}} & {\makebox[0pt]{$3$}} & {\makebox[0pt]{$4$}} & {\makebox[0pt]{$5$}} 
			\\
			\bullet \ar[r, dash] & \bullet \ar[r, dash] & \bullet \ar[r, dash] & \bullet \ar[r, dash] & \bullet \ar[r, dash] & \bullet \ar[r,dash] & \bullet \ar[r, dash] & \bullet
			\\
			\bullet \ar[r, dash] & \bullet \ar[r,dash] & \bullet \ar[r,dash] & \bullet \ar[r, dash] & \bullet \ar[r,dash] & \bullet \ar[r,dash] & \bullet
			\\
			&& \bullet \ar[r, dash] & \bullet \ar[r,dash] & \bullet
			\\
			&& & \bullet \ar[r,dash] & \bullet
		\end{tikzcd}.
	\end{equation*}
\end{clause}

\subsection{Algebraic comparison} \label{subsec:alg-comparison}

In this subsection, we review the construction of Chan-Wong \cite{Chan-Wong:GLn}. Write 
\begin{equation*}
	\Rep(\GL_n(\BC)) = 
	\text{category of Harish-Chandra modules of the real group } \GL_n(\BC),
\end{equation*}
\begin{equation*}
	\Rep^I (\GL_m(\BQ_p)) = 
	\begin{array}{l}
		\text{category of finite length smooth representations of $\GL_m(\BQ_p)$}\\
		\text{whose composition factors admit nonzero Iwahori fixed vectors}
	\end{array}	
\end{equation*}
It is well-known (see \cite{Borel:Iwahori} and \cite{Lusztig:grHecke}) that $\Rep^I(\GL_m(\BQ_p))$ is equivalent to the category $\mod(\BH_m)$ of finite dimensional $\BH_m$-modules, where $\BH_m$ is the graded Hecke algebra of type $A_{m-1}$. For the definition of $\BH_m$, we refer the readers to \cite[\textsection 4]{Lusztig:grHecke} and \cite[\textsection 2.3]{Chan-Wong:GLn}. 

Let $\triv$ denote the trivial representation of $\GL_n(\BC)$, and let $V$ denote the conjugate standard representation of $\GL_n(\BC)$, i.e. $V = \BC^n$ with $g \in \GL_n(\BC)$ acting by the matrix $\bar g$. In the notations of \ref{cls:RT-notations-GLn}, $V = \bar X(\rho, \rho + e_1)$, where $e_1: \fh \cong \BC^n \to \BC$ takes the first coordinate. Then there exists an $\BH_m$-module structure on $\Hom_{\bU(n)}(\triv, M \dotimes_\BC V^{\otimes m})$. We may define the functor
\begin{equation*}
	\Gamma_{n,m}: \Rep(\GL_n(\BC)) \aro \mod(\BH_m),\quad
	M \mapsto \Hom_{\bU(n)}(\triv, M \dotimes_\BC V^{\otimes m}),
\end{equation*}
The functor $\Gamma_{n,m}$ relates certain translation functors on $\Rep(\GL_n(\BC))$ and certain derivatives on $\mod(\BH_m)$. To be more precise, for each $i \le m$ consider the subgroup $S_i \subset S_m$ of permutations on the first $i$ letters. It induces an inclusion $\BC[S_i] \inj \BC[S_m] \subset \BH_m$. For any irreducible representation $\tau$ of $S_i$, the \textit{generalized Bernstein-Zelevinskii derivative} is defined to be
\begin{equation*}
	\mathbf{BZ}_\tau: \mod(\BH_m) \aro \mod(\BH_{m-i}), \quad
	\pi \mapsto \Hom_{S_i}(\tau, \pi).
\end{equation*}
When $\tau$ is the sign representation, the operator $\mathbf{BZ}_\tau$ agrees with Bernstein-Zelevinskii's original derivatives \cite{Bernstein-Zelevinskii} up to some normalizations, see \cite{Chan-Savin:BZ}. On the real side, consider 
\begin{equation*}
	T_\tau: \Rep(\GL_n(\BC)) \aro \Rep(\GL_n(\BC)),\quad
	\pi \mapsto \pi \otimes \Hom_{S_i}(\tau, V^{\otimes i})
\end{equation*}
where $S_i$ acts on $V^{\otimes i}$ by permutating the factors, twisted by the sign representation.

\begin{theorem}[{\cite[Theorem 1.1, 1.2]{Chan-Wong:GLn}}]
	The functor $\Gamma_{n,m}$ satisfies the following properties.
	\begin{itemize}
		\item $\Gamma_{n,m}$ is exact. It commutes with parabolic inductions and preserves unitarity.
		\item $\Gamma_{n,m}$ sends standard (resp. simple) modules to standard (resp. simple) modules.
		\item For each $i$ and each irreducible representation $\tau$ of $S_i$, $\Gamma_{n,m}$ intertwines $\mathbf{BZ}_\tau$ and $T_\tau$, that is
		\begin{equation*}
			\begin{tikzcd}
				\mod(\BH_m) \ar[d, "\mathbf{BZ}_\tau"']
				& \Rep(\GL_n(\BC)) \ar[l, "\Gamma_{n,m}"'] \ar[d, "T_\tau"]
				\\
				\mod(\BH_{m-i})
				& \Rep(\GL_n(\BC)) \ar[l, "\Gamma_{n,m-i}"']
			\end{tikzcd}
		\end{equation*}
	\end{itemize}
\end{theorem}

We may compose $\Gamma_{n,m}$ with the equivalence between $\BH_m$-modules and $\Rep^I(\GL_m(\BQ_p))$. By abuse of notation, we denote the resulting functor also by $\Gamma_{n,m}$:
\begin{equation*}
	\Gamma_{n,m}: \Rep(\GL_n(\BC)) \aro \Rep^I(\GL_m(\BQ_p)).
\end{equation*}
The image of standard representations are described in \textit{loc. cit.}, Theorem 6.4:

\begin{theorem}\label{thm:CW:std}
	Fix positive integers $n$ and $m$. Let $\fh \subset \fgl_n(\BC)$ be the diagonal Cartan. Let $\Lambda_\BR = (\lambda_L, \lambda_R) \in \fh^* \times \fh^*$ be an infinitesimal character satisfying 
	\begin{itemize}
		\item $\lambda_{L,i} \ge \lambda_{R,i}$ for each $i$, and
		\item $m = \sum_i (\lambda_{L,i} - \lambda_{R,i})$.
	\end{itemize} 
	Let $\bm = \sum_i \big[ \lambda_{R,i} + \tfrac12, \lambda_{L,i} - \tfrac12 \big]$ be a multisegment. Then
	\begin{equation*}
		\Gamma_{n,m}(X(\lambda_L, \lambda_R)) = X_\bm.
	\end{equation*}
\end{theorem}

\subsection{Compatibility of two approaches}\label{subsec:alg-vs-geom}

Let $\Lambda_\BR= (\lambda_L, \lambda_R) \in \fh^* \oplus \fh^*$ be a dominant infinitesimal character for $\GL_n(\BC)$ ($\fh$ denotes the diagonal Cartan in $\fgl_n(\BC)$). In particular $\lambda_{L,i} - \lambda_{R,i}$ is an integer for any $i$. Here dominance translates to the condition that whenever $i < j$ and $\lambda_{L,i} - \lambda_{L,j}$ is an integer, we have $\lambda_{L,i} \ge \lambda_{L,j}$. 

If $\Lambda_\BR$ is integral, we assume 
\begin{equation}\label{eqn:LambdaR-apart}
	\min_i \{ \lambda_{L,i}\} > \max_i \{ \lambda_{R,i}\}+1
\end{equation}
(see \textsection\ref{subsec:gen-notn} for the meaning of $\min$ and $\max$). If $\Lambda_\BR$ is non-integral, then we require
\begin{equation}\label{eqn:LambdaR-apart-nint}
	\min_{i \in \sS} \{ \lambda_{L,i}\} > \max_{i \in \sS} \{ \lambda_{R,i}\} +1
\end{equation}
for any subset $\sS \subset \{1,\ldots,n\}$ so that $\lambda_{L,i} - \lambda_{L,j} \in \BZ$ whenever $i,j \in \sS$. Note that any $\Lambda_\BR$ can be transformed into one of this form by adding a twist by the determinant character.

Write
\begin{align}
	m &:= \sum_i (\lambda_{L,i} - \lambda_{R,i}) \label{eqn:m-from-LambdaR}\\ 
	\bm &:= \sum_i \big[ \lambda_{R,i} + \tfrac12, \lambda_{L,i} - \tfrac12 \big]\\
	\varphi &:= \text{the weight function of } \bm
	\; (\text{see \ref{cls:p-adic-closure-relations}})\\
	&\quad \text{i.e. } \varphi(a) = \text{the number of times $a$ appears in $\bm$}, \nonumber\\
	\Lambda &:= \text{infinitesimal character of $\GL_m(\BQ_p)$ corresponding to $\varphi$} 
	\; (\text{see \ref{cls:ELambda-as-Evarphi}}), \label{eqn:Lambda-from-LambdaR}\\ 
	\ubar T &:= \text{any element in } \ubar \BO_\varphi. \label{eqn:ubarT-from-LambdaR}
\end{align}
Then $\Lambda$ is the infinitesimal character of $X_\bm$, and $O_\bm \subset \BO_\varphi$ is the smallest (i.e. closed) $G_\varphi$-orbit in $\BO_\varphi$.

\begin{clause}\label{cls:cond-for-both-comparisons}
	If $\Lambda_\BR$ is integral (i.e. the differences $\lambda_{L,i} - \lambda_{L,j}$ are all integers, for any $i$ and $j$), then by our assumption, we have
	\begin{equation}\label{eqn:cond-on-LambdaR}
		\lambda_{R,n} \le \lambda_{R,n-1} \le \cdots \le \lambda_{R,1} < \lambda_{R,1}+1 < \lambda_{L,n} \le \cdots \le \lambda_{L,2} \le \lambda_{L,1}.
	\end{equation}
	The weight function $\varphi$ is also integral (in the sense of \ref{cls:integral-pieces}). Observe the following facts:
	\begin{enumerate}[label=(\roman*)]
		\item $(w \lambda_L, \lambda_R)$ (for $w \in W(\GL_n)$) satisfies the condition of Theorem \ref{thm:CW:std}. 
		
		Indeed, we have $\lambda_{L,i} \ge \lambda_{R,i}$ for each $i$ by (\ref{eqn:cond-on-LambdaR}), and the condition on $m$ of Theorem \ref{thm:CW:std} is satisfied by our construction of $m$.
		
		\item $\varphi$ satisfies Assumption \ref{assump:r}, where one may choose $r$ to a suitable number $r \in \supp \varphi + \BQ$ so that $\lambda_{L,n} > r > \lambda_{R,1}$. 
		
		Indeed, the open orbit $O_{\bm'}$ in $E_\varphi$ corresponds to the multisegment
		\begin{equation*}
			\bm' 
			= [\lambda_{R,1} + \tfrac12, \lambda_{L,n} - \tfrac12] 
			+ [\lambda_{R,2} + \tfrac12, \lambda_{L,n-1} - \tfrac12] + \cdots
			+ [\lambda_{R,n} + \tfrac12, \lambda_{L,1} - \tfrac12]
		\end{equation*}
		which is nested and the shortest segment $[\lambda_{R,1} + \tfrac12, \lambda_{L,n} - \tfrac12]$ has length $\ge 1$ (see the discussion after Assumption \ref{assump:r}).
	\end{enumerate}
	
	Hence the conditions for the geometric and the algebraic comparison are both satisfied. 
	
	We now want to write down the precise relation between real and $p$-adic parameter spaces in this specific setup. 
	
	\begin{lemma}\label{lem:compare-G/Ps}
		Assume we are in the above setup. Let $\tau$ be the involution on $\fh$ given by $h \mapsto -w_0 h$. Recall the partial flag varieties $\cP_{\lambda_L}$, $\cP_{\lambda_R}$ defined in \ref{cls:real-geom-setup}. Let $\cP_{\varphi, \floor r}$ (resp. $\cP_{\varphi,\ceil r}$) be the partial flag variety $\GL(V_{\varphi, \floor r})/ G_{\varphi, \ubar T}^l$ (resp. $\GL(V_{\varphi, \ceil r})/ G_{\varphi, \ubar T}^r$).
		
		Choose an ordered basis for $V_{\varphi, \floor r}$ and $V_{\varphi, \ceil r}$ and use it to build isomorphisms between $\GL(V_{\varphi,\floor r})$, $\GL(V_{\varphi, \ceil r})$ and $\GL(V_{\varphi,\floor r}, V_{\varphi, \ceil r})$ with $\GL_n$ so that the following diagram commutes
		\begin{equation*}
			\begin{tikzcd}
				\GL(V_{\varphi, \ceil r}) \ar[d, "\cong"] \ar[r, phantom, "\acts" description]
				& \GL(V_{\varphi, \floor r}, V_{\varphi, \ceil r}) \ar[d, "\cong"]
				& \GL(V_{\varphi, \floor r}) \ar[d, "\cong"] \ar[l, phantom, "\racts" description]\\
				\GL_n \ar[r, phantom, "\acts_l" description]
				& \GL_n
				& \GL_n \ar[l, phantom, "\racts_r" description]
			\end{tikzcd}.
		\end{equation*}
		where $\acts_l$ and $\racts_r$ denote the left and right multiplication actions. Then the two outer vertical maps induce isomorphisms
		\begin{equation*}
			\cP_{\varphi,\floor r} \bijects \cP_{\tau(\lambda_R)},\quad
			\cP_{\varphi,\ceil r} \bijects \cP_{\tau(\lambda_L)}.
		\end{equation*}
		Consequently, we have an isomorphism
		\begin{equation*}
			\big[ \Delta \GL_n \backslash (\cP_{\tau(\lambda_L)} \times \cP_{\tau(\lambda_R)}) \big]
			\cong
			\big[ G_{\varphi, \ubar T}^r \backslash \GL(V_{\varphi, \floor r}, V_{\varphi, \ceil r}) / G_{\varphi, \ubar T}^l \big].
		\end{equation*}
		Combined with Theorem \ref{thm:symm-orbit-structure}(4), we obtain
		\begin{equation}\label{eqn:quotients-stacks-by-basis}
			\big[ \Delta \GL_n \backslash (\cP_{\tau(\lambda_L)} \times \cP_{\tau(\lambda_R)}) \big]
			\cong
			[G_\varphi \backslash \BO_\varphi] \injects [G_\varphi \backslash E_\varphi].
		\end{equation}
	\end{lemma}
	
	\begin{proof}
		The isomorphisms of partial flag varieties follows by unraveling the definitions. We write down details for $\cP_{\varphi,\floor r} \bij \cP_{\tau(\lambda_R)}$; the other isomorphism $\cP_{\varphi,\ceil r} \bij \cP_{\tau(\lambda_L)}$ is similar. An example is provided at the end of the proof.
		
		Let $1 \le j_1, j_2,\ldots, j_k \le n$ be integers such that as $j$ increases, the value of $\lambda_{R,j}$ jumps precisely at the $j_i+1$'s, i.e.
		\begin{itemize}
			\item $\lambda_{R, 1} = \lambda_{R,2} = \cdots = \lambda_{R,j_1}$,
			\item $\lambda_{R, j_1} > \lambda_{R, j_1 +1}$,
			\item $\lambda_{R, j_1 +1} = \lambda_{R, j_1 + 2} = \cdots = \lambda_{R,j_2}$,
			\item $\lambda_{R, j_2} > \lambda_{R, j_2+1}$
		\end{itemize}
		and so on. By definition, $\cP_{\lambda_R}$ contains the upper-triangular parabolic $P_{\lambda_R}$ which consists of block upper-triangular matrices whose diagonal blocks are (read from top-left to bottom-right)
		\begin{equation*}
			\GL_{j_1} \times \GL_{j_2-j_1} \times \GL_{j_3-j_2} \times \cdots \times \GL_{n-j_k}.
		\end{equation*}
		On the other hand, if $i_0$ is the smallest number in the support of $\varphi$, then $G_{\varphi, \ubar T}^l$ also consists of block upper-triangular matrices whose diagonal blocks are
		\begin{equation*}
			\GL_{\varphi(i_0)} \times \GL_{\varphi(i_0+1) - \varphi(i_0)} \times \GL_{\varphi(i_0+2) - \varphi(i_0+1)} \times \cdots \times \GL_{\varphi(\floor r) - \varphi(\floor r -1)},
		\end{equation*}
		see \ref{thm:symm-orbit-structure}(3). Some of the blocks here are redundant since we could have $\varphi(i_0+j+1) - \varphi(i_0+j) = 0$ which happens when the value $\varphi(i)$ does not change as $i$ increases. This value jumps at $i$ (i.e. $\varphi(i) > \varphi(i-1)$) precisely when there is a segment in $\bm$ that starts at $i$. The smallest such $i$ is $i_0$, which equals $\lambda_{R,n} +\frac12 = \lambda_{R,n-1} + \frac12 = \cdots = \lambda_{R,j_k+1} + \frac12$, since $\lambda_{R,n}$ is the smallest number among the $\lambda_{R,i}$'s. In this case $\varphi(i_0) = n-j_k$. The next such $i$ is $\lambda_{R,j_k} + \frac12 = \cdots = \lambda_{R,j_{k-1}+1} + \frac12$, in which case $\varphi(i) = n - j_{k-1}$. The pattern continues. Hence the diagonal blocks in $G_{\varphi, \ubar T}^l$ are
		\begin{equation*}
			\GL_{n-j_k} \times \GL_{j_k-j_{k-1}} \times \GL_{j_{k-1}-j_{k-2}} \times \cdots \times \GL_{j_2 - j_1} \times \GL_{j_1}.
		\end{equation*}
		These are the same diagonal blocks as $P_{\lambda_R}$ but in reverse order. It is now clear that $G_{\varphi, \ubar T}^l$ is equal to the upper-triangular parabolic subgroup $P_{\tau(\lambda_R)}$ determined by $\tau(\lambda_R) = (-\lambda_{R,n}, \ldots, -\lambda_{R,1})$.
		
		For example, consider $\lambda_R = (3,3,3,2,1,0,0,0,0)$. Then $P_{\lambda_R}$ is the upper-triangular parabolic whose Levi is $\GL_3 \times \GL_1 \times \GL_1 \times \GL_4$. The left half of the multisegment $\bm$ is
		
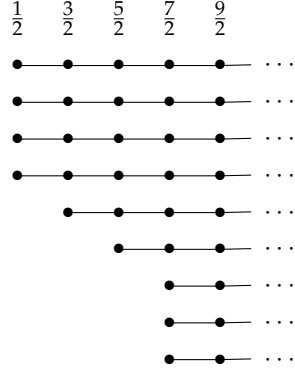
\begin{equation*}
			\begin{tikzcd}[start anchor = real center, end anchor = real center, row sep=0.3ex, column sep=1ex]
				{\makebox[0pt]{$\tfrac12$}} & {\makebox[0pt]{$\tfrac32$}} &{\makebox[0pt]{$\tfrac52$}} & {\makebox[0pt]{$\tfrac72$}} & {\makebox[0pt]{$\tfrac92$}} \\
				\bullet \ar[r, dash] & \bullet \ar[r, dash] & \bullet \ar[r, dash] & \bullet \ar[r, dash] & \bullet \ar[r,dash, end anchor=west] & \cdots 
				\\
				\bullet \ar[r, dash] & \bullet \ar[r, dash] & \bullet \ar[r, dash] & \bullet \ar[r, dash] & \bullet \ar[r,dash, end anchor=west] & \cdots 
				\\
				\bullet \ar[r, dash] & \bullet \ar[r, dash] & \bullet \ar[r, dash] & \bullet \ar[r, dash] & \bullet \ar[r,dash, end anchor=west] & \cdots 
				\\
				\bullet \ar[r, dash] & \bullet \ar[r, dash] & \bullet \ar[r, dash] & \bullet \ar[r, dash] & \bullet \ar[r,dash, end anchor=west] & \cdots 
				\\
				& \bullet \ar[r, dash] & \bullet \ar[r, dash] & \bullet \ar[r, dash] & \bullet \ar[r,dash, end anchor=west] & \cdots 
				\\
				&& \bullet \ar[r, dash] & \bullet \ar[r, dash] & \bullet \ar[r,dash, end anchor=west] & \cdots 
				\\
				&&& \bullet \ar[r, dash] & \bullet \ar[r,dash, end anchor=west] & \cdots 
				\\
				&&& \bullet \ar[r, dash] & \bullet \ar[r,dash, end anchor=west] & \cdots
				\\
				&&& \bullet \ar[r, dash] & \bullet \ar[r,dash, end anchor=west] & \cdots
			\end{tikzcd}
		\end{equation*}	
		and so $G_{\varphi, \ubar T}^l$ is the upper-triangular parabolic whose Levi is $\GL_4 \times \GL_1 \times \GL_1 \times \GL_3$, which equals the standard parabolic $P_{\tau(\lambda_R)}$ of $\tau(\lambda_R) = (0,0,0,0,-1,-2,-3,-3,-3)$.	
	\end{proof}
\end{clause}

We may now state the relation between geometric and algebraic comparisons. Continue to work under the setup of (\ref{eqn:LambdaR-apart})-(\ref{eqn:ubarT-from-LambdaR}) and assume $\Lambda_\BR$ is integral. Consider the exact functor $\Gamma_{n,m}$ restricted to the full subcategory of modules with infinitesimal character $\Lambda_\BR$. It descends to a homomorphism on the level of Grothendieck groups
\begin{equation*}
	\Gamma_{n,m}: K \Rep(\GL_n(\BC))_{\Lambda_\BR} \injects K \Rep(\GL_m(\BQ_p))_\Lambda
\end{equation*}
which is injective by Theorem \ref{thm:CW:std} (the conditions of \ref{thm:CW:std} are satisfied by \ref{cls:cond-for-both-comparisons}(i)). On the dual side, 
we precompose (\ref{eqn:quotients-stacks-by-basis}) with the map $\tau \times \tau$ sending $\cP_{\lambda_L}$ (resp. $\cP_{\lambda_R}$) to $\cP_{\tau(\lambda_L)}$ (resp. $\cP_{\tau(\lambda_R)}$):
\begin{equation}\label{eqn:iota}
	\imath : [ \Delta \GL_n \backslash (\cP_{\lambda_L} \times \cP_{\lambda_R})]
	\xrightarrow[\sim]{\tau \times \tau} [ \Delta \GL_n \backslash (\cP_{\tau(\lambda_L)} \times \cP_{\tau(\lambda_R)})]
	\xhookrightarrow{(\ref{eqn:quotients-stacks-by-basis})}
	[G_\varphi \backslash E_\varphi].
\end{equation}
The exact pullback functor
\begin{equation*}
	\imath^*: D^b(E_\varphi, G_\varphi) \aro D^b(\cP_{\lambda_L} \times \cP_{\lambda_R}, \Delta \GL_n)
\end{equation*}
sends standard modules to standard modules or zero, and descends to a surjective homomorphism on Grothendieck groups.

\begin{theorem}\label{thm:alg-vs-geom}
	Suppose $\Lambda_\BR = (\lambda_L, \lambda_R)$ is a dominant integral infinitesimal character for $\GL_n(\BC)$ satisfying
	\begin{equation}
		\min_i \{ \lambda_{L,i}\} > \max_i \{\lambda_{R,i}\} +1. \tag{\ref{eqn:LambdaR-apart}} 
	\end{equation}
	Define $m$, $\varphi$ and $\Lambda$ as in (\ref{eqn:m-from-LambdaR})-(\ref{eqn:Lambda-from-LambdaR}). Choose an ordered basis for $V_{\varphi,\floor r}$ and $V_{\varphi, \ceil r}$ and use it to construct the map $\imath$ as in (\ref{eqn:iota}). Then with respect to the perfect pairings 
	\begin{equation}
		K \Rep (\GL_n(\BC))_{\Lambda_\BR} \times K D^b(\cP_{\lambda_L} \times \cP_{\lambda_R}, \Delta \GL_n) \aro \BZ, \tag{\ref{eqn:LLC-GLn-pairing}}
	\end{equation}
	\begin{equation}
		K \Rep(\GL_m(\BQ_p))_\Lambda \times K D^b(E_\varphi, G_\varphi) \aro \BZ, \tag{Proposition \ref{prop:LLC-GLmQp}}
	\end{equation}
	the functors $\Gamma_{n,m}$ and $\imath^*$ are adjoint to each other, i.e.
	\begin{equation*}
		\langle \Gamma_{n,m} M, \cF\rangle =  \langle M,  \imath^* \cF \rangle
	\end{equation*}
	for any $M \in K \Rep(\GL_n(\BC))_{\Lambda_\BR}$ and any $\cF \in K D^b(E_\varphi, G_\varphi)$.
\end{theorem}

\begin{proof}
	Since both functors sends standards to standards and all Kottwitz signs are equal to $1$, it suffices to check the following: $\Gamma_{n,m}$ sends a standard object $X$ to a standard object $X'$ if and only if $\imath^*$ sends the dual $I'$ of $X'$ to the dual $I$ of $X$. For $w \in W$, write $\bm(w) = \sum_i [(w\lambda_R)_i + \frac12, \lambda_{L,i} - \frac12]$ (so that the $\bm$ in (\ref{eqn:m-from-LambdaR}) is equal to $\bm(1)$).  Then $\Gamma_{n,m}(X(\lambda_L, w \lambda_R)) = X_{\bm(w)}$ (Theorem \ref{thm:CW:std}). The dual of $X(\lambda_L, w \lambda_R)$ (resp. $X_{\bm(w)}$) under Langlands correspondence is $(-1)^{\dim Z_{w\inv}} \cM_{w\inv}$ (resp. $(-1)^{\dim O_{\bm(w)}} \cM_{\bm(w)}$) by Proposition \ref{prop:LLC-GLnC} (resp. \ref{prop:LLC-GLmQp}). It remains to show that the image of $(-1)^{\dim O_{\bm(w)}} \cM_{\bm(w)}$ under $\imath^*$ is equal to $(-1)^{\dim Z_{w\inv}} \cM_{w\inv}$. 
	
	We first match the orbits. Recall that the chosen basis of $V_{\varphi, \floor r}$ and $V_{\varphi, \ceil r}$ are used identify $\GL(V_{\varphi,\floor r},V_{\varphi,\ceil r})$ with $\GL_n$ and $G_{\varphi,\ubar T}^l$, $G_{\varphi,\ubar T}^r$ with parabolic subgroups $P_l$, $P_r$ in $\GL_n$ (see Lemma \ref{lem:compare-G/Ps}). Fix an element $\ubar T \in \ubar \BO_\varphi$ so that both $P_l$ and $P_r$ contain the upper-triangular Borel $B$. We then need to check that for any $T \in O_{\bm(w)}$ in the fiber over $\ubar T$, the image of the map $T_{\floor r} \in \GL(V_{\varphi,\floor r},V_{\varphi,\ceil r})$ in $\GL_n$ lies in the Bruhat cell $P_r w_0w w_0 P_l$. This is straightforward but tedious to write down, and we omit the details. Therefore the orbit $O_{\bm(w)}$ is sent to the orbit $Z_{(w_0 w w_0)\inv} = Z_{w_0 w\inv w_0}$ in  $\cP_{\tau(\lambda_L)} \times \cP_{\tau(\lambda_R)}$ under (\ref{eqn:quotients-stacks-by-basis}) (the inverse $(-)\inv$ comes from our convention on relative positions in the definition of $Z_w$). Finally, the map $\tau \times \tau$ in the definition of $\imath$ sends $Z_{w_0 w\inv w_0}$ to $Z_{w\inv}$. Thus $\imath\inv(O_{\bm(w)}) = Z_{w\inv}$ as required.
	
	It remains to check the signs. Writing $\iota$ as a composition of maps, the dimension of the underlying space changes only at the induction equivalences. For a general induction equivalence $[H \backslash X] \cong [G \backslash (G \times_H X)]$ and any $H$-orbit $\cO$ in $X$, we have a Cartesian diagram
	\begin{equation*}
		\begin{tikzcd}
			\cO \ar[d, "i"'] \ar[r, "j_\cO"]
			& G \times_H \cO\ar[d, "\tilde i"] \\
			X \ar[r, "j_X"]
			& G \times_H X
		\end{tikzcd}
	\end{equation*}
	and base change says $i_! j_\cO^* \ubar \BC_{G \times_H \cO} \cong j_X^* \tilde i_! \ubar \BC_{G \times_H \cO}$, which simplifies to $\cM_\cO [-\dim G/H] \cong j_X^* \cM_{G \times_H X}$. Here $\dim G/H = \dim G \times_H X - \dim X = \dim G \times_H \cO - \dim \cO$. Apply this to each induction equivalence in $\iota^*$, we see 
	\begin{equation*}
		\iota^* \cM_{\bm(w)} [\dim O_{\bm(w)}] = \cM_{Z_{w\inv}} [\dim Z_{w\inv}]
	\end{equation*}
	and so in the Grothendieck group $\iota^*$ sends $(-1)^{\dim O_{\bm(w)}} \cM_{\bm(w)}$ to $(-1)^{\dim Z_{w\inv}} \cM_{Z_{w\inv}}$.
\end{proof} 

In the non-integral case, let $\varphi = \sum_s \varphi_s$ be its decomposition into integral pieces. Then the above proof goes through for each integral piece. We omit the details.

\section{Comparison between Lusztig's inductions and push-pull functors}\label{sec:compare-functors}

Throughout this section, except in \S \ref{subsec:left-to-right}, let $\varphi$ and $\psi$ be integral weight functions so that
\begin{itemize}
	\item $\varphi$ and $\psi$ satisfy Assumption \ref{assump:r} for the same number $r$,
	\item $k \in \supp \varphi$, $k > \ceil r$; $c \in \BZ_{\ge 1}$, and
	\item $\varphi = \psi + c[k]$.
\end{itemize}
Here $[k]:\BC \to \BZ_{\ge 0}$ is the characteristic function at $k$. These requirements forces $c \le \varphi(k) - \varphi(k+1)$. We fix elements 
\begin{equation*}
	\ubar S \in \ubar \BO_\psi \text{ and } \ubar T \in \ubar \BO_\varphi.
\end{equation*}

The goal of this section is to prove Corollary \ref{cor:compare-functors} which says Lusztig's induction on Vogan varieties and the push-pull functors on the ABV spaces are intertwined by the geometric comparison map $\zeta$ (\ref{eqn:zeta}). In more detail, recall from \textsection \ref{subsec:geom-comparison} that the full rank part $[G_\varphi \backslash \BO_\varphi]$ can be described as $[P_{\varphi,r} \backslash \GL_n / P_{\varphi,l}]$ for some parabolic subgroups $P_{\varphi,l}, P_{\varphi,r} \subset \GL_n$ (here $n = \varphi(\floor r) = \varphi(\ceil r)$). Suppose 
the parabolics $P_{\psi,l}$ and $P_{\varphi,l}$ (resp. $P_{\psi,r}$ and $P_{\varphi,r}$) contain the same Borel (this can be done by choosing a suitable basis for the vector spaces involved). Let $Q_r := P_{\psi,r} \cap P_{\varphi,r}$ and $Q_l := P_{\psi,l} = P_{\varphi,l}$. Corollary  \ref{cor:compare-functors} says that the restriction of Lusztig's induction $\LInd_{\psi, c[k]}^\varphi$ to the full rank part can be described as $\pi_{\varphi*} \circ \pi_\psi^*$ via $\zeta$, where $\pi_\psi$ and $\pi_\varphi$ are natural projections
\begin{equation*}
	[P_{\psi,r} \backslash \GL_n / P_{\psi,l}] \xleftarrow{\pi_\psi} [Q_r \backslash \GL_n / Q_l] \xrightarrow{\pi_\varphi} [P_{\varphi,r} \backslash \GL_n / P_{\varphi,l}].
\end{equation*}
There is an analogous result for the case for $k < \floor r$, see \S \ref{subsec:left-to-right}.

The above statement requires a choice of basis for the vector spaces $V_{\varphi, \floor r}$, $V_{\varphi, \ceil r}$, $V_{\psi, \floor r}$ and $V_{\psi,\ceil r}$ so that the fibers $\BO_\varphi|_{\ubar T}$, $\BO_\psi|_{\ubar S}$ in the full rank parts $\BO_\varphi$ and $\BO_\psi$ can be identified with $\GL_n$. In fact, to define the maps $\pi_\psi$ and $\pi_\varphi$, it is enough to choose a graded extension $0 \to V_{c[k]} \to V_\varphi \to V_\psi \to 0$ compatible with $\ubar S$ and $\ubar T$, but there is no canonical choice of such extension. It turns out that the space of all such extensions, called the extension space, appears naturally in Lusztig's induction, and its properties are the keys to the proof of the desired identification of functors.

In \textsection \ref{subsec:symm-ind}, we rewrite the restriction of Lusztig's induction functor to the fibers. In \textsection\ref{subsec:ext-space}, we study properties of the extension space $\Ext_{\ubar S, 0}^{\ubar T}$ which will be used to, among other things, define the maps $\pi_\psi$ and $\pi_\varphi$ in a basis-free way. Some examples of $\pi_\psi$ and $\pi_\varphi$ are discussed at the end of \textsection \ref{subsec:ext-space}. The identification of $\LInd$ with $\pi_{\varphi*} \circ \pi_\psi^*$ is proven in \textsection\ref{subsec:comparison-thms}. The last subsection \textsection\ref{subsec:left-to-right} reduces the case $k < \floor r$ to the case $k > \ceil r$.

\subsection{Lusztig's induction and restriction to fibers}\label{subsec:symm-ind}

In this subsection we restrict Lusztig's induction functor to the fibers $\BO_\psi|_{\ubar S}$ and $\BO_\varphi|_{\ubar T}$. We introduce notations for several useful spaces along the way.

Consider the following diagram
\begin{equation}\label{diag:base-change}
	\begin{tikzcd}[column sep=small, cramped]
		&&&& \widetilde{\BO_{\varphi,c[k]}}|_{\ubar S, \ubar T} 
		\ar[dl, hook] \ar[dr, hook] \ar[dlll, "\gamma_l"'] \ar[drr, "\gamma_r"]\\
		\{\ubar S\} \ar[d, hook] 
		& \BO_\psi|_{\ubar S} \ar[l, two heads] \ar[d, hook, "i_\psi"'] 
		& \BO_\psi|_{\ubar S} \times E_{c[k]} \ar[l, two heads] \ar[d, hook]
		& \widetilde{\BO_{\varphi,c[k]}}|_{\ubar S} \ar[l] \ar[dr, hook]
		&& \widetilde{\BO_{\varphi,c[k]}}|_{\ubar T} \ar[dl, hook] \ar[r, two heads]
		& \BO_{\varphi,c[k]}|_{\ubar T} \ar[d, hook] \ar[r, two heads, "p|_{\ubar T}"]
		& \BO_\varphi|_{\ubar T} \ar[d, hook, "i_\varphi"] \ar[r, two heads]
		& \{ \ubar T \} \ar[d, hook]\\
		\ubar \BO_\psi \ar[d, hook]
		& \BO_\psi \ar[l, two heads] \ar[d, hook, "j_\psi"']
		& \BO_\psi \times E_{c[k]} \ar[l, two heads, "q|_\BO"'] \ar[d, hook]
		& \beta\inv(\BO_\psi \times E_{c[k]}) \ar[l, two heads, "\beta|_\BO"'] \ar[dr, hook]
		& \widetilde{\BO_{\varphi,c[k]}} \ar[l, hook, dashed, "i"] \ar[rr, two heads, "\tilde p|_\BO"] \ar[d, hook]
		&& \BO_{\varphi,c[k]} \ar[r, two heads, "p|_\BO"] \ar[d, hook]
		& \BO_\varphi \ar[r, two heads] \ar[d, hook, "j_\varphi"]
		& \ubar \BO_\varphi  \ar[d, hook]\\
		\ubar E_\psi
		& E_\psi \ar[l, two heads]
		& E_\psi \times E_{c[k]} \ar[l, two heads, "q"']
		&& \widetilde{E_{\varphi,c[k]}} \ar[ll, two heads, "\beta"'] \ar[rr, two heads, "\tilde p"]
		&& E_{\varphi,c[k]} \ar[r, two heads, "p"]
		& E_\varphi \ar[r, two heads]
		& \ubar E_\varphi
	\end{tikzcd}.
\end{equation}
Here the bottom row is obtained by combining the maps of (\ref{eqn:induction-diag}) with the projections $E_\varphi \surj \ubar E_\varphi$, $E_\psi \surj \ubar E_\psi$ defined in \ref{cls:symm-part} and the projection $E_\psi \xtwoheadleftarrow{q} E_\psi \times E_{c[k]}$ onto the first factor. The maps in the two leftmost and rightmost columns are also defined in \ref{cls:symm-part}. All other squares, trapezoids, and solid maps are obtained from the bottom row and the two outer columns by taking fiber products. The maps $\gamma_l$ and $\gamma_r$ are defined by compositions. It remains to explain the dashed map $i$. Note that both $\widetilde{\BO_{\varphi,c[k]}}$ and $\beta\inv(\BO_\psi \times E_{c[k]})$ are open subsets of $\widetilde{E_{\varphi,c[k]}}$ (they are preimages of open subsets along the maps). We claim that $\widetilde{\BO_{\varphi,c[k]}}$ is contained in $\beta\inv(\BO_\psi \times E_{c[k]})$, whence an inclusion map $i$. Indeed, if $(V_{c[k]} \xrightarrow{\iota} V_\varphi \xrightarrow{\eps} V_\psi, T) \in \widetilde{\BO_{\varphi,c[k]}}$, then $T$ has maximal possible rank at each degree. Then the induced map $\eps \circ \bar T \circ \eps\inv \in E_\psi$ also has maximal possible rank. So $q \circ \beta$ sends any element in $\widetilde{\BO_{\varphi,c[k]}}$ into $\BO_\psi$, and hence $\widetilde{\BO_{\varphi,c[k]}}$ is a subset of $\beta\inv(\BO_\psi \times E_{c[k]})$, whence the dashed inclusion map $i$.

It is easy to check that the arrows on the top of the diagram are equivariant with respect to group actions
\begin{equation*}
	\begin{tikzcd}
		G_{\psi, \ubar S} \times G_{c[k]} \ar[d, phantom, "\racts" {rot90, description}]
		& G_{\psi, \ubar S} \times G_{c[k]} \times G_{\varphi, \ubar T} \ar[d, phantom, "\racts" {rot90, description}] \ar[l, two heads, "pr_{12}"'] \ar[r, two heads, "pr_3"]
		& G_{\varphi, \ubar T} \ar[d, phantom, "\racts" {rot90, description}] \ar[r, equal]
		& G_{\varphi, \ubar T} \ar[d, phantom, "\racts" {rot90, description}]\\
		\BO_\psi|_{\ubar S}
		& \widetilde{\BO_{\varphi, c[k]}}|_{\ubar S, \ubar T} \ar[l, "\gamma_l"'] \ar[r, "\gamma_r"]
		& \BO_{\varphi, c[k]}|_{\ubar T} \ar[r, "p|_{\ubar T}"]
		& \BO_\varphi|_{\ubar T}
	\end{tikzcd}.
\end{equation*}
Hence we obtain maps 
\begin{equation*}
	[(G_{\psi, \ubar S} \times G_{c[k]}) \backslash \BO_\psi|_{\ubar S}] 
	\xleftarrow{\gamma_l} 
	[ (G_{\psi, \ubar S} \times G_{c[k]} \times G_{\varphi, \ubar T}) \backslash \widetilde{\BO_{\varphi, c[k]}}|_{\ubar S, \ubar T}] 
	\xrightarrow{\gamma_r}
	[ G_{\varphi, \ubar T} \backslash \BO_{\varphi, c[k]}|_{\ubar T}]
	\xrightarrow{p|_{\ubar T}}
	[ G_{\varphi, \ubar T} \backslash \BO_\varphi|_{\ubar T}].
\end{equation*}
It will be shown in Proposition \ref{prop:ext-space}(1)(3) below that the map $\gamma_r$ is a principal $G_{\psi, \ubar S} \times G_{c[k]}$-bundle, and hence $\gamma_r$ is an isomorphism at the level of quotient stacks.

\begin{definition}\label{def:symm-ind}
	The induction functor from $\ubar S$ to $\ubar T$ is 
	\begin{equation*}
		\Ind_{\ubar S}^{\ubar T} := (p|_{\ubar T})_* \circ (\gamma_r^*)\inv \circ \gamma_l^*:
		D^b(\BO_\psi|_{\ubar S}, G_{\psi, \ubar S} \times G_{c[k]}) \aro D^b(\BO_\varphi|_{\ubar T}, G_{\varphi, \ubar T}).
	\end{equation*}
	Here $G_{c[k]}$ acts on $\BO_\psi|_{\ubar S}$ trivially.
\end{definition}

By base change and descent, it is easy to see:

\begin{lemma}\label{lem:Lind-on-fiber}
	We have a natural isomorphism of functors
	\begin{equation*}
		i_\varphi^* \circ j_\varphi^* \circ \LInd_{\psi,c[k]}^\varphi \circ q^* = \Ind_{\ubar S}^{\ubar T} \circ i_\psi^* \circ j_\psi^*
	\end{equation*}
	from $D^b(E_\psi, G_\psi \times G_{c[k]})$ to $D^b(\BO_\varphi|_{\ubar T}, G_{\varphi, \ubar T})$.
\end{lemma}


\subsection{The extension space}\label{subsec:ext-space}

\begin{definition}
	The \textbf{extension space} of $\varphi = \psi+c[k]$ is the space of graded extensions
	\begin{equation*}
		\Ext_{\psi,c[k]}^\varphi = \big\{ \text{graded short exact sequences } \cE = (0 \to V_{c[k]} \xrightarrow{\iota} V_\varphi \xrightarrow{\eps} V_\psi \to 0)  \big\}.
	\end{equation*}
	This is equipped with the $G_\psi \times G_{c[k]}$-action defined by twisting $\iota$ on the source and twisting $\eps$ on the target. It is also equipped with the $G_\varphi$-action defined by simultaneously twisting $\iota$ on the target and $\eps$ on the source. These two actions commute, and hence we have an action
	\begin{equation*}
		G_\psi \times G_{c[k]} \times G_\varphi \acts \Ext_{\psi,c[k]}^\varphi.
	\end{equation*}
	
	We define
	\begin{equation*}
		\Ext_{\ubar S,0}^{\ubar T} = \big\{ \cE \in \Ext_{\psi,c[k]}^\varphi \mid \cE \text{ intertwines the maps } \ubar S \text{ and } \ubar T \big\},
	\end{equation*}
	equipped with the action of $G_{\psi, \ubar S} \times G_{c[k]} \times G_{\varphi, \ubar T}$.
	By ``$\cE$ intertwines $\ubar S$ and $\ubar T$'', we mean that the following diagram commutes for any $i \neq \floor r$:
	\begin{equation}\label{diag:ext-space}
		\begin{tikzcd}
			V_{c[k],i+1} \ar[r, "\iota"] &
			V_{\varphi, i+1} \ar[r, "\eps"] &
			V_{\psi,i+1} \\
			V_{c[k],i} \ar[u, "0"] \ar[r, "\iota"] &
			V_{\varphi,i} \ar[u, "\ubar T"] \ar[r, "\eps"] &
			V_{\psi,i} \ar[u, "\ubar S"]
		\end{tikzcd}
	\end{equation}
	where the subscripts $(-)_i$ denote the restrictions on the degree $i$ part.
	
\end{definition}

\begin{proposition}\label{prop:ext-space}~
	\begin{enumerate}
		\item The action $G_{\psi,\ubar S} \times G_{c[k]} \acts \Ext_{\ubar S, 0}^{\ubar T}$ is free. The quotient map is represented by
		\begin{equation*}
			\Ext_{\ubar S, 0}^{\ubar T} \surjects \Gr_c(\ker \ubar T_k),\quad
			\big( V_{c[k]} \xrightarrow{\iota} V_\varphi \xrightarrow{\eps} V_\psi \big) \mapsto \im\iota
		\end{equation*}
		where $\Gr_c(\ker \ubar T_k)$ is the Grassmannian variety of $c$-dimensional subspaces in $\ker \ubar T_k$. The quotient map is a Zariski locally trivial $G_{\psi,\ubar S} \times G_{c[k]}$-principal bundle.
		
		\item The action $G_{\varphi, \ubar T} \acts \Ext_{\ubar S, 0}^{\ubar T}$ is free. Along the quotient map in (1), it descends to a transitive action on $\Gr_c(\ker \ubar T_k)$. In particular, the action
		\begin{equation*}
			G_{\psi, \ubar S} \times G_{c[k]} \times G_{\varphi, \ubar T} \acts \Ext_{\ubar S, 0}^{\ubar T}
		\end{equation*}
		is transitive.
		
		\item We have natural isomorphisms
		\begin{equation*}
			\BO_\psi|_{\ubar S} \times \Ext_{\ubar S, 0}^{\ubar T} 
			\lbijects 
			\widetilde{\BO_{\varphi,c[k]}}|_{\ubar S, \ubar T} 
			\bijects
			\Ext_{\ubar S, 0}^{\ubar T} \times \BO_\varphi|_{\ubar T},
		\end{equation*}
		\begin{equation*}
			\big( \eps \circ \bar T \circ \eps\inv \big) \times \big( V_{c[k]} \xrightarrow{\iota} V_\varphi \xrightarrow{\eps} V_\psi \big) 
			\;\mapsfrom\;
			\big( V_{c[k]} \xrightarrow{\iota} V_\varphi \xrightarrow{\eps} V_\psi, T \big) 
			\mapsto
			\big( V_{c[k]} \xrightarrow{\iota} V_\varphi \xrightarrow{\eps} V_\psi \big) \times T
		\end{equation*}
		where $\bar{T}$ is the endomorphism on $V_\varphi/\iota(V_{c[k]})$ induced by $T$. Consider the composition
		\begin{equation*}
			\BO_\psi|_{\ubar S} \lsurjects
			\BO_\psi|_{\ubar S} \times \Ext_{\ubar S, 0}^{\ubar T} \lbijects \widetilde{\BO_{\varphi,c[k]}}|_{\ubar S, \ubar T}
		\end{equation*}
		where the map on the left is the projection, and the composition
		\begin{equation*}
			\widetilde{\BO_{\varphi,c[k]}}|_{\ubar S, \ubar T} \bijects
			\Ext_{\ubar S, 0}^{\ubar T} \times \BO_\varphi|_{\ubar T} \surjects
			\Gr_c(\ker \ubar T_k) \times \BO_\varphi|_{\ubar T} = \BO_{\varphi,c[k]}|_{\ubar T}
		\end{equation*}
		where the middle map is defined in part (1), and the last equality is by definition. They agree with the maps $\gamma_l$ and $\gamma_r$ in (\ref{diag:base-change}), respectively.
		
		\item Consider the actions of $G_{\psi,\ubar S} \times G_{c[k]} \times G_{\varphi, \ubar T}$ on $\BO_\psi|_{\ubar S} \times \Ext_{\ubar S, 0}^{\ubar T}$ and on $\Ext_{\ubar S, 0}^{\ubar T} \times \BO_\varphi|_{\ubar T}$ inherited from the action on $\widetilde{\BO_{\varphi,c[k]}}|_{\ubar S, \ubar T}$. Then
		\begin{enumerate}
			\item $G_{\varphi,\ubar T} \acts \BO_\psi|_{\ubar S} \times \Ext_{\ubar S, 0}^{\ubar T}$ only affects the second component;
			
			\item $G_{\psi,\ubar S} \times G_{c[k]} \acts \Ext_{\ubar S, 0}^{\ubar T} \times \BO_\varphi|_{\ubar T}$ only affects the first component.
		\end{enumerate}
	\end{enumerate}	
\end{proposition}

\begin{proof}~
	\begin{enumerate}
		\item From the commutative diagram (\ref{diag:ext-space}), $\ubar T$ restricts to the zero map on the image of $\iota$. This explains why $\im \iota \in \Gr_c(\ker \ubar T_k)$.
		
		Let $\cE = (\iota,\eps) \in \Ext_{\ubar S, 0}^{\ubar T}$ be any extension. For any $g_\psi \in G_{\psi,\ubar S}$ that is not the identity, there is a vector $v_\psi \in V_\psi$ that is not fixed by $g_\psi$. Since $\eps$ is surjective, there is a vector $v_\varphi$ so that $\eps(v_\varphi) = v_\psi$. Then $(g_\psi \cdot \eps) (v_\varphi) = g_\psi \cdot v_\psi \neq v_\psi = \eps(v_\varphi)$ and hence $g_\psi \cdot \cE \neq \cE$. Therefore any non-identity element in $G_{\psi, \ubar S}$ has no fixed point in $\Ext_{\ubar S, 0}^{\ubar T}$.
		
		The same argument together with injectivity of $\iota$ shows that the $G_{c[k]}$-action is also free.
		
		The map $\Ext_{\ubar S, 0}^{\ubar T} \surj \Gr_c(\ker \ubar T_k)$ is clearly constant on $G_{\psi, \ubar S} \times G_{c[k]}$-orbits: twisting $\iota$ on the source or $\eps$ on the target does not change $\im \iota$. Now suppose two extensions $\cE = (\iota,\eps)$, $\cE' = (\iota',\eps')$ in $\Ext_{\ubar S, 0}^{\ubar T}$ give the same $\im \iota$. Write $W= \im \iota$. Then the composition $V_{c[k]} \xrightarrow{\iota'} W \xrightarrow{\iota\inv} V_{c[k]}$ is an automorphism, so $\iota$ and $\iota'$ differ by an element in $G_{c[k]}$. Similarly, the maps $\eps$ and $\eps'$ both factor through the quotient $V_\varphi/W$ followed by isomorphisms $\bar\eps$, $\bar \eps'$ from $V_\varphi/W$ to $V_\psi$. So the composition $V_\psi \xrightarrow{\bar\eps\inv} V_\varphi/W \xrightarrow{\bar \eps'} V_\psi$ is an automorphism, and hence $\bar\eps$ and $\bar \eps'$ differ by an element $g_\psi \in G_\psi$. We need to show $g_\psi \in G_{\psi, \ubar S}$. Indeed, since $\cE$ and $\cE'$ are assumed to intertwine $\ubar S$ and $\ubar T$, we have commutative diagrams
		\begin{equation*}
			\begin{tikzcd}
				V_\varphi/W \ar[r, "\bar \eps"]
				& V_\psi\\
				V_\varphi/W \ar[u, "\ubar T"] \ar[r, "\bar \eps"] 
				& V_\psi \ar[u, "\ubar S"']
			\end{tikzcd},\quad
			\begin{tikzcd}
				V_\varphi/W \ar[r, "\bar \eps'"]
				& V_\psi\\
				V_\varphi/W \ar[u, "\ubar T"] \ar[r, "\bar \eps'"]
				& V_\psi \ar[u, "\ubar S"']
			\end{tikzcd}.
		\end{equation*}
		Hence $g_\psi \circ \ubar S =\bar \eps' \circ \bar \eps\inv \circ \ubar S = \bar \eps' \circ \ubar T \circ \bar \eps\inv = \ubar S \circ \bar \eps' \circ \bar \eps\inv = \ubar S \circ g_\psi$, and so $g_\psi \in G_{\psi,\ubar S}$. Thus $\cE$ and $\cE'$ differ by an element of $G_{\psi, \ubar S} \times G_{c[k]}$. Hence the fibers of $\Ext_{\ubar S, 0}^{\ubar T} \surj \Gr_c(\ker \ubar T_k)$ are precisely the $G_{\psi ,\ubar S} \times G_{c[k]}$-orbits.
		
		We then show that the quotient map is Zariski locally trivial. Fix $W \in \Gr_c(\ker \ubar T_k)$. The Grassmannian $\Gr_c(\ker \ubar T_k)$ is a partial flag variety $\GL(\ker \ubar T_k)/P$ of $\GL(\ker \ubar T_k)$, where the parabolic $P$ can be taken to be the stabilizer of $W$. Let $\bar U$ be the unipotent radical of a parabolic opposite to $P$. Then $\bar UP/P$ is an open set in $\Gr_c(\ker \ubar T_k)$ containing the point $W$ which is isomorphic to $\bar U$. We want to show that the quotient map is trivial over $\bar UP/P$. Choose ordered homogeneous bases for $V_{c[k]}$, $V_\varphi$ and $V_\psi$ so that the first a few basis vectors of $V_\varphi$ span $W$. These bases provide a graded extension $\cE = (\iota,\eps) \in \Ext_{\ubar S, 0}^{\ubar T}$ lying over the point $W$. Then we have an isomorphism
		\begin{equation*}
			G_{\psi,\ubar S} \times G_{c[k]} \times \bar U \bijects \Ext_{\ubar S, 0}^{\ubar T}|_{\bar UP/P}
		\end{equation*}
		that sends $(g_\psi, g_{c[k]}, u)$ to the extension $(u \circ \iota \circ g_{c[k]}\inv, g_\psi \circ \eps \circ u\inv)$. This gives the desired trivialization around $W$.
		
		\item Suppose $g_\varphi = (g_{i}) \in G_{\varphi, \ubar T}$ stabilizes $\cE = (\iota, \eps)$, where $g_{i} \in \GL(V_{\varphi,i})$. This means that the following diagram commutes for any $i$:
		\begin{equation*}
			\begin{tikzcd}
				0 \ar[r]
				& V_{c[k],i} \ar[r, "\iota_i"] \ar[d, equal]
				& V_{\varphi, i} \ar[r, "\eps_i"] \ar[d, "g_{i}"]
				& V_{\psi, i} \ar[r] \ar[d, equal] 
				& 0\\
				0 \ar[r]
				& V_{c[k],i} \ar[r, "\iota_i"]
				& V_{\varphi, i} \ar[r, "\eps_i"]
				& V_{\psi, i} \ar[r]
				& 0
			\end{tikzcd}.
		\end{equation*}
		When $i \neq k$, $\eps_i$ is an isomorphism. Hence the second square forces $g_{i}$ to be the identity. When $i = k$, we have a commutative square
		\begin{equation*}
			\begin{tikzcd}
				V_{\varphi, k-1} \ar[r, "{\ubar T_{k-1}}"] \ar[d, equal, "g_{k-1}"']
				& V_{\varphi,k} \ar[d, "g_{ k}"] \\
				V_{\varphi, k-1} \ar[r, "{\ubar T_{k-1}}"]
				& V_{\varphi, k} 
			\end{tikzcd}
		\end{equation*}
		by the fact that $g_\varphi$ stabilizes $\ubar T$. Since $\ubar T$ has full rank, $\ubar T_{k-1}$ is surjective. So the above square forces $g_{k}$ to be the identity as well. Therefore $g_\varphi$ is the identity element. Hence the action $G_{\varphi, \ubar T} \acts \Ext_{\ubar S, 0}^{\ubar T}$ is free.
		
		We now show that $G_{\varphi, \ubar T} \acts \Gr_c(\ker \ubar T_k)$ is transitive. Take the basis \ref{cls:adapted-basis} of $V$ and write $G_{\varphi, \ubar T}$ in matrix terms as in the proof of Theorem \ref{thm:symm-orbit-structure}(3). Then $\ker \ubar T_k$ is spanned by $v_k^{\varphi(k+1)+1},\ldots, v_k^{\varphi(k)}$. The component of $G_{\varphi, \ubar T}$ in $\GL(V_{\varphi,k})$ is a lower-triangular parabolic subgroup whose Levi contains $\GL_{\varphi(k)-\varphi(k+1)}$ in the block corresponding to $v_k^{\varphi(k+1)+1},\ldots, v_k^{\varphi(k)}$, i.e. it contains $\GL(\ker \ubar T_k)$. Since the action $G_{\varphi, \ubar T} \acts \Gr_c(\ker \ubar T_k)$ factors through the projection onto the $\GL(V_{\varphi,k})$-component and since $\GL(\ker \ubar T_k)$ acts transitively on $\Gr_c(\ker \ubar T_k)$, the statement follows.
		
		
		As a result, $G_{\varphi, \ubar T}$ acts transitively on the base of the fibration $\Ext_{\ubar S, 0}^{\ubar T} \surj \Gr_c(\ker \ubar T_k)$ of part (1), and $G_{\psi, \ubar S} \times G_{c[k]}$ acts (simply) transitively on its fibers. Hence the product $G_{\psi, \ubar S} \times G_{c[k]} \times G_{\varphi, \ubar T}$ acts transitively on $\Ext_{\ubar S, 0}^{\ubar T}$.
		
		\item Recall from (\ref{diag:base-change}) that $\widetilde{\BO_{\varphi,c[k]}}|_{\ubar S, \ubar T}$ consists of pairs $\big( \cE: V_{c[k]} \xrightarrow{\iota} V_\varphi \xrightarrow{\eps} V_\psi, T \big)$ so that the subspace $\iota(V_{c[k]}) \subset V_\varphi$ is $T$-stable, $T$ has full rank, $T$ is sent to $\ubar T$ under the projection $\BO_\varphi \surj \ubar \BO_\varphi$, and $S := \eps\inv \circ \bar T \circ \eps$ is sent to $\ubar S$ under the projection $\BO_\psi \surj \ubar \BO_\psi$. Therefore by definition $\cE$ intertwines $\ubar S$ and $\ubar T$, i.e. $\cE \in \Ext_{\ubar S,0}^{\ubar T}$, and the projections to $\cE$ and $T$ define a map
		\begin{equation*}
			\widetilde{\BO_{\varphi,c[k]}}|_{\ubar S, \ubar T} \aro
			\Ext_{\ubar S, 0}^{\ubar T} \times \BO_\varphi|_{\ubar T}
		\end{equation*}
		which is clearly injective. Conversely, for any $\cE = (\iota,\eps) \in \Ext_{\ubar S, 0}^{\ubar T}$, the image $\iota(V_{c[k]}) \subset V_\varphi$ is concentrated in degree $k$ and is $\ubar T_k$-stable (by (\ref{diag:ext-space})). Since $k$ is distinct from $\floor r$, we have $\ubar T_k = T_k$ for any $T \in \BO_\varphi|_{\ubar T}$. Hence $\iota(V_{c[k]})$ is $T$-stable for any $T \in \BO_\varphi|_{\ubar T}$, and $(\cE,T)$ is an element of $\widetilde{\BO_{\varphi,c[k]}}|_{\ubar S, \ubar T}$. Therefore the above map from $\widetilde{\BO_{\varphi,c[k]}}|_{\ubar S, \ubar T}$ to $\Ext_{\ubar S, 0}^{\ubar T} \times \BO_\varphi|_{\ubar T}$ is in fact an isomorphism.
		
		A similar argument establishes the isomorphism from $\BO_\psi|_{\ubar S} \times \Ext_{\ubar S, 0}^{\ubar T}$ to $\widetilde{\BO_{\varphi,c[k]}}|_{\ubar S, \ubar T}$.
		
		The rest of (3) is straightforward to check.
		
		\item The $G_{\psi,\ubar S} \times G_{c[k]}$-action on an element $\big( \cE: V_{c[k]} \xrightarrow{\iota} V_\varphi \xrightarrow{\eps} V_\psi, T \big)$ in $\widetilde{\BO_{\varphi,c[k]}}|_{\ubar S,\ubar T}$ is by twisting $\eps$ and $\iota$. In particular, the $T$ component is not affected. So if we identify $\widetilde{\BO_{\varphi,c[k]}}|_{\ubar S,\ubar T}$ with $\Ext_{\ubar S, 0}^{\ubar T} \times \BO_\varphi|_{\ubar T}$, the second component is not affected by the $G_{\psi,\ubar S} \times G_{c[k]}$-action. This proves (b). Part (a) is similar. \qedhere
	\end{enumerate}	
\end{proof}

We examine more closely the transitive action of $G_{\psi, \ubar S} \times G_{c[k]} \times G_{\varphi, \ubar T}$ on $\Ext_{\ubar S, 0}^{\ubar T}$. 

\begin{proposition}\label{prop:SE}
	Fix any element $\cE = (\iota,\eps) \in \Ext_{\ubar S, 0}^{\ubar T}$. By assumption, $\eps_i: V_{\varphi,i} \to V_{\psi,i}$ is an isomorphism for any $i \neq k$.
	\begin{enumerate}
		\item The isomorphisms $\eps_{\floor r}$ and $\eps_{\ceil r}$ induce isomorphisms
		\begin{equation*}
			\theta_\cE: \GL(V_{\psi, \floor r}) \times \GL(V_{\psi, \ceil r}) \bijects 
			\GL(V_{\varphi,\floor r}) \times \GL(V_{\varphi,\ceil r})
		\end{equation*}
		\begin{equation*}
			\tau_\cE: \BO_\psi|_{\ubar S} \bijects \BO_\varphi|_{\ubar T}, \quad
			S \mapsto T.
		\end{equation*}
		Namely, $\theta_\cE$ is equal to conjugation by $\eps_{\floor r} \times \eps_{\ceil r}$, and $T = \tau_\cE(S)$ is determined by the following commutative diagram
		\begin{equation}\label{diag:S->T}
			\begin{tikzcd}
				V_{\varphi, \ceil r} \ar[r, "\eps_{\ceil r}", "\cong"'] &
				V_{\psi,\ceil r} \\
				V_{\varphi,\floor r} \ar[u, "T"] \ar[r, "\eps_{\floor r}", "\cong"'] &
				V_{\psi,\floor r} \ar[u, "S"]
			\end{tikzcd}.
		\end{equation}
		
		\item Let $S_\cE$ be the stabilizer of $\cE$ in $G_{\psi, \ubar S} \times G_{c[k]} \times G_{\varphi, \ubar T}$. Under the projections
		\begin{equation*}
			G_{\psi, \ubar S} \lsurjects G_{\psi, \ubar S} \times G_{c[k]} \lsurjects G_{\psi, \ubar S} \times G_{c[k]} \times G_{\varphi, \ubar T} \surjects G_{\varphi, \ubar T}
		\end{equation*}
		the group $S_\cE$ maps injectively onto its images. Denote the images by $S_{\cE,\psi} \subseteq G_{\psi, \ubar S}$, $S_{\cE, \psi,c[k]} \subseteq G_{\psi, \ubar S} \times G_{c[k]}$, and $S_{\cE, \varphi} \subseteq G_{\varphi, \ubar T}$. Then we have
		\begin{equation*}
			S_{\cE, \varphi} = \theta_\cE(S_{\cE, \psi}) = \theta_\cE(G_{\psi, \ubar S}) \cap G_{\varphi, \ubar T}.
		\end{equation*}
		Here we are identifying $G_{\psi, \ubar S}$ and $S_{\cE,\psi}$ with their images under the isomorphism $G_{\psi, \ubar S} \bij G_{\psi, \ubar S}^l \times G_{\psi, \ubar S}^r$ so that the map $\theta_\cE$ can be applied.
		
		\item We have a commutative diagram
		\begin{equation*}
			\begin{tikzcd}
				G_{\psi, \ubar S} \times G_{c[k]} \ar[d, phantom, "\racts" {rot90, description}] 
				& S_\cE \ar[l, hook] \ar[r, hook]
				& G_{\varphi, \ubar T} \ar[d, phantom, "\racts" {rot90, description}] 
				\\
				\BO_\psi|_{\ubar S} \ar[rr, "\tau_\cE", "\cong"']
				&& \BO_\varphi|_{\ubar T}
			\end{tikzcd},
		\end{equation*}
		where the inclusions of $S_\cE$ to the two groups are the maps of (2). In particular, we have maps of quotient stacks
		\begin{equation*}
			[(G_{\psi, \ubar S} \times G_{c[k]}) \backslash (\BO_\psi|_{\ubar S})]
			\xleftarrow{\pi_{\psi,c[k]}} [S_\cE \backslash (\BO_\psi|_{\ubar S})]
			\xrightarrow[\sim]{\tau_\cE} [S_\cE \backslash (\BO_\varphi|_{\ubar T})]
			\xrightarrow{\pi_\varphi} [G_{\varphi,\ubar T} \backslash (\BO_\varphi|_{\ubar T})]
		\end{equation*}
		and hence a pullback-pushforward functor
		\begin{equation*}
			\bI_\cE = \pi_{\varphi*} \circ \pi_{\psi,c[k]}^* : 
			D^b(\BO_\psi|_{\ubar S}, G_{\psi, \ubar S} \times G_{c[k]}) 
			\aro
			D^b(\BO_\varphi|_{\ubar T}, G_{\varphi, \ubar T}).		 
		\end{equation*}
	\end{enumerate}
\end{proposition}

\begin{remark}
	More explicitly, the pushforward functor $\pi_{\varphi*}$ is defined as the averaging functor $\Av_*^\cE$ \cite[Table 6.8.1]{Achar:book}. It is the composition 
	\begin{equation*}
		\Av_*^\cE:
		D^b(\BO_\varphi|_{\ubar T}, S_\cE)
		\xrightarrow[\sim]{((\{1\} \times \id)^*)\inv} 
		D^b(G_{\varphi, \ubar T} \times_{S_\cE} \BO_\varphi|_{\ubar T}, G_{\varphi, \ubar T})
		\xrightarrow{act_*}
		D^b(\BO_\varphi|_{\ubar T}, G_{\varphi, \ubar T})
	\end{equation*}
	where the maps are
	\begin{equation*}
		\BO_\varphi|_{\ubar T} \xhookrightarrow{\{1\} \times \id} G_{\varphi, \ubar T} \times_{S_\cE} \BO_\varphi|_{\ubar T} \xtwoheadrightarrow{act} \BO_\varphi|_{\ubar T},\quad
		(\{1\} \times \id)(x) = (1,x),\quad
		act(g,x) = g \cdot x.
	\end{equation*}
\end{remark}

\begin{remark}
	We may precompose $\bI_\cE$ with the inflation map 
	\begin{equation*}
		\operatorname{Infl}_1^{G_{c[k]}}: D^b(\operatorname{pt}) \aro D^b(\operatorname{pt}, G_{c[k]})
	\end{equation*}
	which is the $*$-pullback along the natural map $[G_{c[k]} \backslash \operatorname{pt}] \to \operatorname{pt}$. The two sides of the resulting functor
	\begin{equation*}
		\bI_\cE \circ \operatorname{Infl}_1^{G_{c[k]}}: D^b(\BO_\psi|_{\ubar T}, G_{\psi, \ubar S}) \aro D^b(\BO_\varphi|_{\ubar S}, G_{\varphi, \ubar T})
	\end{equation*}
	are now equivariant derived categories on ABV spaces for $\GL_n$.
\end{remark}

\begin{proof}
	Part (1) is clear. Parts (2) and (3) can be proven by choosing nice bases for the graded vector spaces and write the stabilizer groups in matrix terms as in the proof of Theorem \ref{thm:symm-orbit-structure}(3). Below we give a basis-free proof.
	
	We look at part (2). For convenience, we identify $G_{\varphi, \ubar T}$ with $G_{\varphi, \ubar T}^l \times G_{\varphi, \ubar T}^r$ by using Theorem \ref{thm:symm-orbit-structure}, and similarly for $G_{\psi, \ubar S}$. We write $g_\psi \in G_{\psi, \ubar S}$, $g_{c[k]} \in G_{c[k]}$ and $g_\varphi \in G_{\varphi, \ubar T}$ for typical elements.
	
	Consider first the composition
	\begin{equation*}
		S_\cE \subseteq G_{\psi, \ubar S} \times G_{c[k]} \times G_{\varphi, \ubar T} \surjects G_{\psi, \ubar S} \times G_{\varphi, \ubar T}.
	\end{equation*}
	Let $(g_\psi, g_\varphi)$ be in the image. Then in particular $g_{\psi, \ceil r} \circ \eps_{\ceil r} \circ g_{\varphi, \ceil r}\inv = \eps_{\ceil r}$. In other words $\theta_\cE(g_{\psi, \ceil r}) = g_{\varphi, \ceil r}$. Similarly, $\theta_\cE(g_{\psi, \floor r}) = g_{\varphi, \floor r}$. Hence $\theta_\cE(g_\psi) = g_\varphi$. As a result, the image of the above composition is contained in the group
	\begin{equation*}
		A_{\psi, \varphi} := \big\{ (g_\psi, g_\varphi) \in G_{\psi, \ubar S} \times G_{\varphi, \ubar T} \mid
		\theta_\cE(g_\psi) = g_\varphi \big\} \subseteq G_{\psi, \ubar S} \times G_{\varphi, \ubar T}
	\end{equation*}
	For any $(g_\psi,g_\varphi) \in A_{\psi, \varphi}$, the two components $g_\psi$ and $g_\varphi$ determine each other. Hence this group maps injectively under the projections from $G_{\psi, \ubar S} \times G_{\varphi, \ubar T}$ to the two factors, and the images are equal to $\theta_\cE\inv(\theta_\cE(G_{\psi, \ubar S}) \cap G_{\varphi, \ubar T})$ and $\theta_\cE(G_{\psi, \ubar S}) \cap G_{\varphi, \ubar T}$, respectively.
	
	Next, we show that the map $S_\cE \to A_{\psi, \varphi}$ is an isomorphism. This amounts to showing that for each pair $(g_\psi, g_\varphi)$ with $\theta_\cE(g_\psi) = g_\varphi$, there is a unique element $g_{c[k]} \in G_{c[k]}$ so that $(g_\psi, g_{c[k]}, g_\varphi)$ stabilizes $\cE$. The condition $\theta_\cE(g_\psi) = g_\varphi$ is equivalent to saying that $g_\varphi$ and $g_\psi$ stabilize both $\eps_{\floor r}$ and $\eps_{\ceil r}$. We claim that $g_\varphi$ and $g_\psi$ intertwine $\eps$ in all degrees. To see this, consider the diagram
	\begin{equation*}
		\begin{tikzcd}[row sep=small, column sep=small]
			& 
			&& V_{\varphi,i+1} \ar[rr, "\eps"]  \ar[from=dd, "\ubar T", near end]
			&& V_{\psi, i+1}\\
			&& V_{\varphi,i+1} \ar[rr, "\eps", near end, crossing over] \ar[ur, "g_{\varphi}"]
			&& V_{\psi,i+1} \ar[ur, "g_{\psi}"]\\
			& 
			&& V_{\varphi,i} \ar[rr, "\eps", near start]  
			&& V_{\psi,i} \ar[uu, "\ubar S"]\\
			&& V_{\varphi,i} \ar[rr, "\eps"] \ar[ur, "g_{\varphi}"] \ar[uu, crossing over, "\ubar T", near end]
			&& V_{\psi,i} \ar[ur, "g_{\psi}"] \ar[uu, crossing over, "\ubar S", near end]
		\end{tikzcd}
	\end{equation*}
	where all the vertical squares commute. If all vertical arrows are surjections (resp. injections) and the bottom (resp. the top) square commutes, then so does the top (resp. the bottom) square. The claim follows by applying this argument inductively starting from $i = \ceil r$ (resp. $i+1 = \floor r$). Finally, there is a unique element $g_{c[k]} \in G_{c[k]}$ so that $g_{c[k]}$ and $g_\varphi$ intertwine $\iota$. Indeed, consider the following diagram of solid arrows
	\begin{equation*}
		\begin{tikzcd}
			0 \ar[r]
			& V_{c[k]} \ar[d, dashed, "g_{c[k]}"] \ar[r, "\iota"] 
			& V_\varphi \ar[d, "g_\varphi"] \ar[r, "\eps"]
			& V_\psi \ar[d, "g_\psi"] \ar[r]
			& 0\\
			0 \ar[r]
			& V_{c[k]} \ar[r, "\iota"] 
			& V_\varphi \ar[r, "\eps"]
			& V_\psi \ar[r]
			& 0
		\end{tikzcd}.
	\end{equation*}
	The rows are exact, and the right square commutes by the above discussion. Hence there is a unique dashed map $g_{c[k]}$ making the left square commute. Therefore, there is a unique element $g_{c[k]}$ so that $(g_\psi, g_{c[k]}, g_\varphi)$ stabilizes $\cE$, as required.
	
	Consequently, the projections from $G_{\psi, \ubar S} \times G_{c[k]} \times G_{\varphi, \ubar T}$ to $G_{\psi, \ubar S}$ and to $G_{\varphi, \ubar T}$ are injective on $S_\cE$, and the images satisfy $S_{\cE, \varphi} = \theta_\cE(S_{\cE, \psi}) = \theta_\cE(G_{\psi, \ubar S}) \cap G_{\varphi, \ubar T}$. Moreover, since the element $g_{c[k]}$ in any $(g_\psi, g_{c[k]}, g_\varphi) \in S_\cE$ is completely determined by $g_\varphi$ and $g_\varphi = \theta_\cE(g_\psi)$ is determined by $g_\psi$, we see that the the projection from $G_{\psi, \ubar S} \times G_{c[k]}$ to $G_{\psi, \ubar S}$ induces an isomorphism from $S_{\cE, \psi, c[k]}$ to $S_{\cE, \psi}$. This completes the proof of (2).
	
	For part (3), suppose $(g_\psi, g_{c[k]}, g_\varphi) \in S_\cE$ is any element. We need to show that for any $S \in \BO_\psi|_{\ubar S}$, $\tau_\cE((g_\psi, g_{c[k]}) \cdot S) = g_\varphi \cdot \tau_\cE(S)$. Indeed,
	\begin{align*}
		\tau_\cE((g_\psi, g_{c[k]}) \cdot S)
		&= \eps_{\ceil r} \circ (g_\psi \circ S \circ g_\psi\inv) \circ \eps_{\floor r}\\
		&= (\eps_{\ceil r} \circ g_\psi \circ \eps_{\ceil r}\inv) \circ (\eps_{\ceil r} \circ S \circ \eps_{\floor r}) \circ (\eps_{\floor r}\inv \circ g_\psi\inv \circ \eps_{\floor r})\\
		&= \theta_\cE(g_\psi) \circ \tau_\cE(S) \circ \theta_\cE(g_\psi)\inv\\
		&= g_\varphi \circ \tau_\cE(S) \circ g_\varphi\inv = g_\varphi \cdot \tau_\cE(S). \qedhere
	\end{align*}
\end{proof}

Let us examine some special cases and examples. Recall that $0 < c \le \varphi(k) - \varphi(k+1)$ (see the beginning of this section).

\begin{clause}[The ``regular to singular'' case]\label{cls:reg-to-sing}
	Suppose $c = \varphi(k) - \varphi(k+1)$, or equivalently $\psi(k) = \psi(k+1)$. We claim that the parabolic subgroup $G_{\varphi, \ubar T}$ is contained in $\theta_\cE(G_{\psi, \ubar S})$, and we have $S_\cE \cong S_{\cE, \varphi} = G_{\varphi, \ubar T}$. As a result, the functor $\pi_{\varphi*}$ is the identity, and $\bI_\cE \circ \operatorname{Infl}_1^{G_{c[k]}}$ is isomorphic to the forgetful functor from $G_{\psi, \ubar S}$-equivariant sheaves to $G_{\varphi, \ubar T}$-equivariant sheaves.
	
	To see the claim, it suffices to show that any element $g_\varphi \in G_{\varphi, \ubar T}$ can be extended to a unique triple $(g_\psi, g_{c[k]}, g_\varphi)$ stabilizing the extension $\cE$. Note that $\dim \ker \ubar T_k = \dim V_{\varphi, k} - \dim V_{\varphi, k+1} = \varphi(k) - \varphi(k+1) = c$. Therefore the Grassmannian $\Gr_c(\ker \ubar T_k)$ is a point. Since the quotient of the free action $G_{\psi, \ubar S} \times G_{c[k]} \acts \Ext_{\ubar S, 0}^{\ubar T}$ is represented by $\Ext_{\ubar S, 0}^{\ubar T} \surj \Gr_c(\ker \ubar T_k) = pt$ (Proposition \ref{prop:ext-space}(1)), the action $G_{\psi, \ubar S} \times G_{c[k]} \acts \Ext_{\ubar S, 0}^{\ubar T}$ is simply transitive. Hence there is a unique pair $(g_\psi, g_{c[k]}) \in G_{\psi, \ubar S} \times G_{c[k]}$ so that $(g_\psi, g_{c[k]}) \cdot (g_\varphi \cdot \cE) = \cE$. In other words, $(g_\psi, g_{c[k]}, g_\varphi)$ stabilizes $\cE$, as required.
	
	Alternatively, the claim can be seen by writing $G_{\psi, \ubar S}$ and $G_{\varphi, \ubar T}$ as matrices using the basis \ref{cls:adapted-basis} for $V_\varphi$.
	
	Let $\Lambda_\BR$ and $\Lambda_\BR'$ be infinitesimal characters for the real group $\GL_n(\BC)$ ($n = \varphi(\floor r) = \varphi(\ceil r)$) so that the Langlands parameter space for $\Lambda_\BR$ (resp. $\Lambda_\BR'$) matches $[G_\varphi \backslash \BO_\varphi]$ (resp. $[G_\psi \backslash \BO_\psi]$) as in Corollary \ref{cor:zeta}. Then $\Lambda_\BR$ is more regular than $\Lambda_\BR'$ in the sense that any simple root orthogonal to $\Lambda_\BR$ is orthogonal to $\Lambda_\BR'$. Under local Langlands correspondence, the forgetful functor $\bI_\cE \circ \operatorname{Infl}_1^{G_{c[k]}}$ is adjoint to the translation functor $T_{\Lambda_\BR}^{\Lambda_\BR'}$ from the more regular infinitesimal character $\Lambda_\BR$ to the more singular infinitesimal character $\Lambda_\BR'$, see Proposition \ref{prop:adjoint-of-translation}. This justifies the name ``regular to singular''.
	
	Here is an example
	\begin{equation*}
		\psi = 
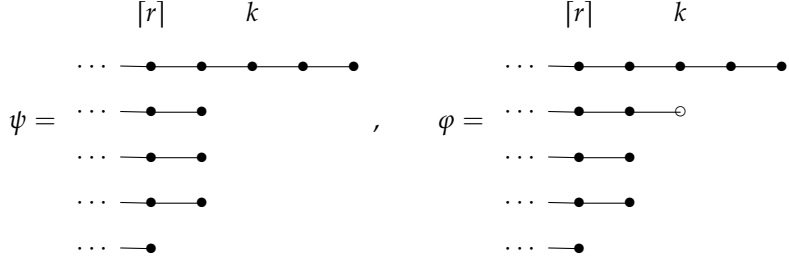
		\begin{tikzcd}[start anchor = real center, end anchor = real center, row sep=1ex, column sep=1ex]
			& {\makebox[0pt]{$\ceil r$}} && {\makebox[0pt]{$k$}}\\
			\cdots \ar[r, dash, start anchor=east] & \bullet \ar[r, dash] & \bullet \ar[r, dash] & \bullet \ar[r, dash] & \bullet \ar[r, dash] & \bullet \\
			\cdots \ar[r, dash, start anchor=east] & \bullet \ar[r, dash] & \bullet \\
			\cdots \ar[r, dash, start anchor=east] & \bullet \ar[r, dash] & \bullet	\\
			\cdots \ar[r, dash, start anchor=east] & \bullet \ar[r, dash] & \bullet	\\
			\cdots \ar[r, dash, start anchor=east] & \bullet	
		\end{tikzcd},
		\qquad
		\varphi = 
		\begin{tikzcd}[start anchor = real center, end anchor = real center, row sep=1ex, column sep=1ex]
			& {\makebox[0pt]{$\ceil r$}} && {\makebox[0pt]{$k$}}\\
			\cdots \ar[r, dash, start anchor=east] & \bullet \ar[r, dash] & \bullet \ar[r, dash] & \bullet \ar[r, dash] & \bullet \ar[r, dash] & \bullet \\
			\cdots \ar[r, dash, start anchor=east] & \bullet \ar[r, dash] & \bullet \ar[r, dash] & \circ\\
			\cdots \ar[r, dash, start anchor=east] & \bullet \ar[r, dash] & \bullet	\\
			\cdots \ar[r, dash, start anchor=east] & \bullet \ar[r, dash] & \bullet	\\
			\cdots \ar[r, dash, start anchor=east] & \bullet
		\end{tikzcd}
	\end{equation*}
	where $\varphi = \psi + [k]$. The picture depicts the right half of a multisegment $\bm$ whose weight function is equal to $\varphi$ (resp. $\psi$), and the extra node in $\varphi$ corresponding to $[k]$ is marked with a $\circ$. Under the basis \ref{cls:adapted-basis}, we have
	\begin{equation*}
		\theta_\cE(G_{\psi, \ubar S}^r) = 
		\left\{
		\begin{pmatrix}
			*\\
			* & * & * & *\\
			* & * & * & *\\
			* & * & * & *\\
			* & * & * & * & *\\
		\end{pmatrix}
		\in \GL_5(\BC) \right\} \supset
		G_{\varphi, \ubar T}^r = 
		\left\{
		\begin{pmatrix}
			*\\
			* & * & \\
			* & * & * & *\\
			* & * & * & *\\
			* & * & * & * & *\\
		\end{pmatrix}
		\in \GL_5(\BC) \right\}.
	\end{equation*}
\end{clause}

\begin{clause}[The ``singular to regular'' case]
	Suppose $\varphi(k-1) = \varphi(k)$, or equivalently $c = \psi(k-1) - \psi(k)$. We claim that $\theta_\cE(G_{\psi, \ubar S})$ is contained in $G_{\varphi, \ubar T}$, and we have $S_{\cE, \psi} = G_{\psi, \ubar S}$. As a result, the functor $\bI_\cE \circ \operatorname{Infl}_1^{G_{c[k]}}$ is isomorphic to the averaging functor from $G_{\psi, \ubar S}$-equivariant sheaves to $G_{\varphi, \ubar T}$-equivariant sheaves.
	
	To see the claim, it suffices to show that any element $g_\psi \in G_{\psi, \ubar S}$ can be extended to a unique triple $(g_\psi, g_{c[k]}, g_\varphi)$ stabilizing the extension $\cE= (\iota, \eps)$. Consider the following commutative diagram consisting of solid arrows
	\begin{equation*}
		\begin{tikzcd}[cramped]
			& 
			&& V_{\varphi,i+1} \ar[rr, "\eps_{i+1}"]  \ar[from=dd, "\ubar T", near end]
			&& V_{\psi, i+1}\\
			&& V_{\varphi,i+1} \ar[rr, "\eps_{i+1}", near end, crossing over] \ar[ur, dashed, "g_{\varphi,i+1}"]
			&& V_{\psi,i+1} \ar[ur, "g_{\psi, i+1}"]\\
			& 
			&& V_{\varphi,i} \ar[rr, "\eps_i", near start]  
			&& V_{\psi,i} \ar[uu, "\ubar S"]\\
			&& V_{\varphi,i} \ar[rr, "\eps_i"] \ar[ur, dashed, "g_{\varphi,i}"] \ar[uu, crossing over, "\ubar T", near end]
			&& V_{\psi,i} \ar[ur, "g_{\psi,i}"] \ar[uu, crossing over, "\ubar S", near end]
		\end{tikzcd}.
	\end{equation*}
	Suppose by induction that $g_{\varphi,i}$ is uniquely determined by $g_\psi$. If $i+1 \neq k$, then $\eps_{i+1}$ is an isomorphism, and the top square uniquely determines an element $g_{\varphi,i+1} \in \GL(V_{\varphi,i+1})$. A straightforward diagram chasing shows that the left square commutes, i.e. $g_{\varphi,i}$ and $g_{\varphi,i+1}$ commutes with $\ubar T$. If $i+1 = k$, then our assumption $\varphi(k-1) = \varphi(k)$ and that $\ubar T$ has full rank ensures $\ubar T_{k-1}: V_{\varphi,k-1} \to V_{\varphi,k}$ is an isomorphism, and the left square uniquely determines an element $g_{\varphi,k} \in \GL(V_{\varphi,k})$, namely $g_{\varphi,k} = \ubar T_{k-1} \circ g_{\varphi,k-1} \circ \ubar T_{k-1}\inv$. Then the top square commutes by a straightforward diagram chasing. Hence there is a unique element $g_\varphi = (g_{\varphi,i})_i$ that commutes with $\ubar T$ and intertwines $\eps$ with $g_\psi$. As in the proof of Proposition \ref{prop:SE}(2), given $g_\psi$ and $g_\varphi$ there is a unique $g_{c[k]} \in G_{c[k]}$ so that $(g_\psi, g_{c[k]}, g_\varphi)$ stabilizes $\cE$, as required.
	
	Alternatively, the claim can be seen by writing $G_{\psi, \ubar S}$ and $G_{\varphi, \ubar T}$ as matrices using the basis \ref{cls:adapted-basis}.
	
	If we choose infinitesimal characters $\Lambda_\BR$ and $\Lambda_\BR'$ matching with $\varphi$ and $\psi$ respectively as in Corollary \ref{cor:zeta}, then $\Lambda_\BR$ is more singular than $\Lambda_\BR'$ in the sense that any simple root orthogonal to $\Lambda_\BR'$ will be orthogonal to $\Lambda_\BR$. Under local Langlands correspondence, the averaging functor $\bI_\cE \circ \operatorname{Infl}_1^{G_{c[k]}}$ is adjoint to the translation functor $T_{\Lambda_\BR}^{\Lambda_\BR'}$ by Proposition \ref{prop:adjoint-of-translation}, whence the name ``singular to regular''.
	
	Here is an example
	\begin{equation*}
		\psi = 
		\begin{tikzcd}[start anchor = real center, end anchor = real center, row sep=1ex, column sep=1ex]
			& {\makebox[0pt]{$\ceil r$}} &&& {\makebox[0pt]{$k$}}\\
			\cdots \ar[r, dash, start anchor=east] & \bullet \ar[r, dash] & \bullet \ar[r, dash] & \bullet \ar[r, dash] & \bullet \ar[r, dash] & \bullet \\
			\cdots \ar[r, dash, start anchor=east] & \bullet \ar[r, dash] & \bullet \ar[r, dash] & \bullet \ar[r, dash] & \bullet\\
			\cdots \ar[r, dash, start anchor=east] & \bullet \ar[r, dash] & \bullet \ar[r, dash] & \bullet \ar[r, dash] & \bullet\\
			\cdots \ar[r, dash, start anchor=east] & \bullet \ar[r, dash] & \bullet \ar[r, dash] & \bullet\\
			\cdots \ar[r, dash, start anchor=east] & \bullet 	
		\end{tikzcd},
		\qquad
		\varphi = 
		\begin{tikzcd}[start anchor = real center, end anchor = real center, row sep=1ex, column sep=1ex]
			& {\makebox[0pt]{$\ceil r$}} &&& {\makebox[0pt]{$k$}}\\
			\cdots \ar[r, dash, start anchor=east] & \bullet \ar[r, dash] & \bullet \ar[r, dash] & \bullet \ar[r, dash] & \bullet \ar[r, dash] & \bullet \\
			\cdots \ar[r, dash, start anchor=east] & \bullet \ar[r, dash] & \bullet \ar[r, dash] & \bullet \ar[r, dash] & \bullet\\
			\cdots \ar[r, dash, start anchor=east] & \bullet \ar[r, dash] & \bullet \ar[r, dash] & \bullet \ar[r, dash] & \bullet\\
			\cdots \ar[r, dash, start anchor=east] & \bullet \ar[r, dash] & \bullet \ar[r, dash] & \bullet \ar[r, dash] & \circ\\
			\cdots \ar[r, dash, start anchor=east] & \bullet 	
		\end{tikzcd}
	\end{equation*}
	where $\varphi = \psi + [k]$. Under the basis \ref{cls:adapted-basis},
	\begin{equation*}
		\theta_\cE(G_{\psi, \ubar S}^r) = 
		\left\{
		\begin{pmatrix}
			*\\
			* & * & * & \\
			* & * & * & \\
			* & * & * & *\\
			* & * & * & * & *\\
		\end{pmatrix}
		\in \GL_5(\BC) \right\} \subset
		G_{\varphi, \ubar T}^r = 
		\left\{
		\begin{pmatrix}
			*\\
			* & * & * & *\\
			* & * & * & *\\
			* & * & * & *\\
			* & * & * & * & *\\
		\end{pmatrix}
		\in \GL_5(\BC) \right\}.
	\end{equation*}
\end{clause}

\begin{clause}[The general case]\label{cls:general_ex}
	In general $\theta_\cE(G_{\psi, \ubar S}^r)$ and $G_{\varphi, \ubar T}^r$ do not contain each other. However, since $\cE$ intertwines $\ubar S$ and $\ubar T$, the parabolic subgroups $G_{\varphi, \ubar T}^r$ and $\theta_\cE(G_{\psi, \ubar S}^r)$ contain a common Borel. Hence the intersection $S_{\cE, \varphi} = \theta_\cE(S_{\cE, \psi}) = \theta_\cE(G_{\psi, \ubar S}) \cap G_{\varphi, \ubar T}$ is also a parabolic subgroup.
	
	In this case, we may choose dominant integral infinitesimal characters $\Lambda_\BR$ and $\Lambda_\BR'$ satisfying (\ref{eqn:lambdaL-and-lambdaL'-a})-(\ref{eqn:lambdaL-and-lambdaL'-c}) that match $\varphi$ and $\psi$, respectively. With such $\Lambda_\BR$ and $\Lambda_\BR'$, the functor $\bI_\cE \circ \operatorname{Infl}_1^{G_{c[k]}}$ is adjoint to the translation functor $T_{\Lambda_\BR}^{\Lambda_\BR'}$ by Proposition \ref{prop:adjoint-of-translation-mixed}. For example, we may take $\Lambda_\BR$ in the following way. Let $\bm = \sum_{i=1}^n [a_i,b_i]$ be the multisegment corresponding to the open $G_\varphi$-orbit in $E_\varphi$, and order the segments $[a_i,b_i]$ in $\bm$ so that those with larger endpoints come first, i.e. $b_i \ge b_j \implies i < j$. Then we may take 
	\begin{equation*}
		\Lambda_\BR = (\lambda_L, \lambda_R) = (b_1, b_2, \ldots, b_n, a_n, \ldots, a_2,a_1).
	\end{equation*}
	We may define $\Lambda_\BR'$ similarly using $\psi$ instead of $\varphi$. More generally, for any numbers $e_L$, $e_R$ with $e_L-e_R \in \BZ$, the elements
	\begin{equation*}
		\Lambda_\BR = (b_1 + e_L, b_2+ e_L, \ldots, b_n + e_L, a_n + e_R, \ldots, a_2 + e_R,a_1 + e_R)
	\end{equation*}
	and similarly defined $\Lambda_\BR'$ work equally well. For $e_L = \frac12$ and $e_R = -\frac12$, we recover the cases discussed in \S \ref{subsec:alg-vs-geom}.
	
	Here is an example
	\begin{equation*}
		\psi = 
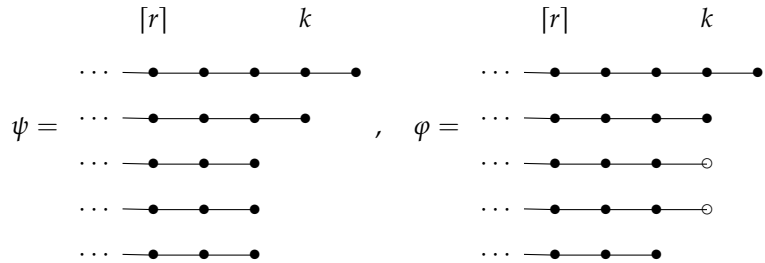
		\begin{tikzcd}[start anchor = real center, end anchor = real center, row sep=1ex, column sep=1ex]
			& {\makebox[0pt]{$\ceil r$}} &&& {\makebox[0pt]{$k$}}\\
			\cdots \ar[r, dash, start anchor=east] & \bullet \ar[r, dash] & \bullet \ar[r, dash] & \bullet \ar[r, dash] & \bullet \ar[r, dash] & \bullet \\
			\cdots \ar[r, dash, start anchor=east] & \bullet \ar[r, dash] & \bullet \ar[r, dash] & \bullet \ar[r, dash] & \bullet\\
			\cdots \ar[r, dash, start anchor=east] & \bullet \ar[r, dash] & \bullet \ar[r, dash] & \bullet \\
			\cdots \ar[r, dash, start anchor=east] & \bullet \ar[r, dash] & \bullet \ar[r, dash] & \bullet\\
			\cdots \ar[r, dash, start anchor=east] & \bullet  \ar[r, dash] & \bullet \ar[r, dash] & \bullet	
		\end{tikzcd},
		\quad
		\varphi = 
		\begin{tikzcd}[start anchor = real center, end anchor = real center, row sep=1ex, column sep=1ex]
			& {\makebox[0pt]{$\ceil r$}} &&& {\makebox[0pt]{$k$}}\\
			\cdots \ar[r, dash, start anchor=east] & \bullet \ar[r, dash] & \bullet \ar[r, dash] & \bullet \ar[r, dash] & \bullet \ar[r, dash] & \bullet \\
			\cdots \ar[r, dash, start anchor=east] & \bullet \ar[r, dash] & \bullet \ar[r, dash] & \bullet \ar[r, dash] & \bullet\\
			\cdots \ar[r, dash, start anchor=east] & \bullet \ar[r, dash] & \bullet \ar[r, dash] & \bullet \ar[r, dash] & \circ\\
			\cdots \ar[r, dash, start anchor=east] & \bullet \ar[r, dash] & \bullet \ar[r, dash] & \bullet \ar[r, dash] & \circ\\
			\cdots \ar[r, dash, start anchor=east] & \bullet \ar[r, dash] & \bullet \ar[r, dash] & \bullet
		\end{tikzcd}
	\end{equation*}
	where $\varphi = \psi + 2[k]$. The infinitesimal characters $\Lambda_\BR$ and $\Lambda_\BR'$ chosen in the above fashion will have $\lambda_R = \lambda_R'$ and
	\begin{equation*}
		\lambda_L = (k+1,k, k,k,k-1),\quad
		\lambda_L' = (k+1,k, k-1,k-1,k-1).
	\end{equation*}
	It is easy to verify that they satisfy the conditions (\ref{eqn:lambdaL-and-lambdaL'-a})-(\ref{eqn:lambdaL-and-lambdaL'-c}) for $j = 3$, $c=2$. Under the basis \ref{cls:adapted-basis} for $V_\varphi$,
	\begin{equation*}
		\theta_\cE(G_{\psi, \ubar S}^r) = 
		\left\{
		\begin{pmatrix}
			*\\
			* & * \\
			* & * & * & * & *\\
			* & * & * & * & *\\
			* & * & * & * & *\\
		\end{pmatrix}
		\in \GL_5(\BC) \right\},\quad
		G_{\varphi, \ubar T}^r = 
		\left\{
		\begin{pmatrix}
			*\\
			* & * & * & *\\
			* & * & * & *\\
			* & * & * & *\\
			* & * & * & * & *\\
		\end{pmatrix}
		\in \GL_5(\BC) \right\},
	\end{equation*}
	\begin{equation*}
		S_{\cE,\varphi} = \theta_\cE(S_{\cE,\psi}) = 
		\left\{
		\begin{pmatrix}
			*\\
			* & * \\
			* & * & * & *\\
			* & * & * & *\\
			* & * & * & * & *\\
		\end{pmatrix}
		\in \GL_5(\BC) \right\}.
	\end{equation*}
\end{clause}

\subsection{The comparison theorem}\label{subsec:comparison-thms}

We are now ready to compare the two functors. 

Let us recall our setup. Suppose $\varphi$ and $\psi$ are integral weight functions satisfying the following conditions:
\begin{itemize}
	\item $\varphi$ and $\psi$ satisfy Assumption \ref{assump:r} for the same number $r$.
	\item $k \in \supp \varphi$, $k > \ceil r$; $c \in \BZ_{\ge 1}$.
	\item $\varphi = \psi + c[k]$.
\end{itemize}
These conditions imply in particular that $c \le \varphi(k) - \varphi(k+1)$. Choose elements $\ubar S \in \ubar \BO_\psi$ and $\ubar T \in \ubar \BO_\varphi$. Recall the induction functor
\begin{equation*}
	\Ind_{\ubar S}^{\ubar T} = (p|_{\ubar T})_* \circ (\gamma_r^*)\inv \circ \gamma_l^*: D^b(\BO_\psi|_{\ubar S}, G_{\psi, \ubar S} \times G_{c[k]}) 
	\aro
	D^b(\BO_\varphi|_{\ubar T}, G_{\varphi, \ubar T})
\end{equation*}
defined in \ref{def:symm-ind} as the restriction of Lusztig's induction (see (\ref{diag:base-change}) for the maps $p|_{\ubar T}$, $\gamma_r$, and $\gamma_l$). Recall from Proposition \ref{prop:SE} that for any choice of extension $\cE \in \Ext_{\ubar S, 0}^{\ubar T}$, there are maps of quotient stacks
\begin{equation*}
	[(G_{\psi, \ubar S} \times G_{c[k]}) \backslash (\BO_\psi|_{\ubar S})]
	\xleftarrow{\pi_{\psi,c[k]}} [S_\cE \backslash (\BO_\psi|_{\ubar S})]
	\xrightarrow[\sim]{\tau_\cE} [S_\cE \backslash (\BO_\varphi|_{\ubar T})]
	\xrightarrow{\pi_\varphi} [G_{\varphi,\ubar T} \backslash (\BO_\varphi|_{\ubar T})]
\end{equation*}
and hence a pullback-pushforward functor
\begin{equation*}
	\bI_\cE = \pi_{\varphi*} \circ \pi_{\psi,c[k]}^* : 
	D^b(\BO_\psi|_{\ubar S}, G_{\psi, \ubar S} \times G_{c[k]}) 
	\aro
	D^b(\BO_\varphi|_{\ubar T}, G_{\varphi, \ubar T})		 
\end{equation*}
where $S_\cE \subset G_{\psi, \ubar S} \times G_{c[k]} \times G_{\varphi, \ubar T}$ is the stabilize of $\cE$.

\begin{theorem}\label{thm:compare-functors}
	Suppose we are in the above setup. Then for any extension $\cE\in \Ext_{\ubar S, 0}^{\ubar T}$, there is a natural isomorphism of functors
	\begin{equation*}
		\Ind_{\ubar S}^{\ubar T} \bijects \bI_\cE
	\end{equation*}	
	from $D^b(\BO_\psi|_{\ubar S}, G_{\psi, \ubar S} \times G_{c[k]})$ to $D^b(\BO_\varphi|_{\ubar T}, G_{\varphi, \ubar T})$.
\end{theorem}

\begin{proof}
	To ease notations, let us write
	\begin{itemize}
		\item $H_1 := G_{\psi, \ubar S} \times G_{c[k]}$;
		\item $H_2 := G_{\varphi, \ubar T}$;
		\item $S := S_\cE \subset H_1 \times H_2$, and we identify $S$ with its images in $H_1$ and $H_2$ along the projections;
		\item $X := \BO_\psi|_{\ubar S}$.
	\end{itemize}
	Recall from Proposition \ref{prop:ext-space}(2) that the action of $G_{\psi, \ubar S} \times G_{c[k]} \times G_{\varphi, \ubar T}$ on $\Ext_{\ubar S, 0}^{\ubar T}$ is transitive. Hence $\Ext_{\ubar S, 0}^{\ubar T} = (H_1 \times H_2)/S$. 
	
	Let us first examine the map $\gamma_l$ in the definition of $\Ind_{\ubar S}^{\ubar T}$. By \ref{prop:ext-space}(2), it factors as
	\begin{equation*}
		\BO_\psi|_{\ubar S} \lsurjects \BO_\psi|_{\ubar S} \times \Ext_{\ubar S, 0}^{\ubar T} \lbijects \widetilde{\BO_{\varphi, c[k]}}|_{\ubar S, \ubar T}.
	\end{equation*}
	In the new notations, this is the projection map to $X$:
	\begin{equation*}
		X \lsurjects X \times (H_1 \times H_2)/S.
	\end{equation*}
	It doesn't hurt to swap $X$ and $(H_1 \times H_2)/S$. So we are looking at the projection
	\begin{equation*}
		X \lsurjects (H_1 \times H_2)/S \times X : q_X
	\end{equation*}
	where the $X$ in the target is equipped with the left action of $H_1$, and $(H_1 \times H_2)/S \times X$ is equipped with the diagonal action of $H_1$ on $H_1$ by multiplication and on $X$, and the left action of $H_2$ on $H_2$ by multiplication. Now this projection map fits into the following diagram
	\begin{equation*}
		\begin{tikzcd}
			&&& H_1 \times H_2 \times X 
			\ar[start anchor = 160, end anchor = 170, out=100, in=150, distance={3ex}, "{}_l H_1"']
			\ar[start anchor = 80, end anchor = 120, out=90, in=135, distance={3ex}, "{}_l H_2"']
			\ar[start anchor = 0, end anchor = -20, out=-45, in=-60, distance={3ex}, "\Delta_{123} S"]
			\\[3ex]
			X \ar[start anchor = 120, end anchor = 150, out=100, in=170, distance={3ex}, "H_1"']
			& X \ar[l, "\id"'] \ar[start anchor = 120, end anchor = 150, out=100, in=170, distance={3ex}, "S"']
			& H_1 \times X \ar[l, "pr_X"'] \ar[d, "\text{quotient by }S"'] 
			\ar[start anchor = 140, end anchor = 150, out=100, in=170, distance={3ex}, "\Delta H_1"']
			\ar[start anchor = 100, end anchor = 120, out=45, in=90, distance={3ex}, "{}_r S"']
			& H_1 \times H_2 \times X \ar[l, "pr_{13}"'] \ar[d, "\text{quotient by }S"'] 
			\ar[start anchor = 160, end anchor = 170, out=100, in=170, distance={3ex}, "\Delta_{13} H_1"']
			\ar[start anchor = 80, end anchor = 120, out=90, in=135, distance={3ex}, "{}_l H_2"']
			\ar[start anchor = -30, end anchor = -60, out=-45, in=-80, distance={3ex}, "\Delta_{12,r} S"]
			\ar[u, xshift={2ex}, "\cong"']
			\\
			&& (H_1/S) \times X
			\ar[start anchor = -140, end anchor = -150, out=-100, in=-170, distance={3ex}, "\Delta H_1"] \ar[ul, bend left, "pr_X"]
			& (H_1 \times H_2)/S \times X \ar[l, "pr_{13}"']
			\ar[start anchor = -160, end anchor = -170, out=-100, in=-170, distance={3ex}, "\Delta_{13} H_1"]
			\ar[start anchor = -80, end anchor = -120, out=-90, in=-135, distance={3ex}, "{}_l H_2"]
		\end{tikzcd}
	\end{equation*}
	where the maps are given by
	\begin{equation*}
		\begin{tikzcd}
			&&& (h_1, h_2, h_1\inv \cdot x)
			\\
			x 
			& x \ar[l, mapsto]
			& (h_1,x) \ar[l, mapsto] \ar[d, mapsto] 
			& (h_1, h_2, x) \ar[d, mapsto] \ar[u, mapsto] \ar[l, mapsto]
			\\
			&& {[h_1,x]}
			& {[h_1,h_2,x]} \ar[l, mapsto]
		\end{tikzcd}
	\end{equation*}
	These maps are equivariant with respect to the actions marked in the diagram. Here $\Delta$, $\Delta_{13}$, $\Delta_{12}$, and $\Delta_{123}$ indicate diagonal actions on the factors given by the subscripts, and the left subscripts ``$l$'' and ``$r$'' denote the left and right multiplication actions, respectively. All the group actions on each object commute with each other. Also, all the maps except the leftmost map $X \leftarrow X$ are isomorphic to one of the form $Z \surj Z/G$ where $Z$ is a variety with a free $G$ action. Therefore, after passing to quotient stacks, the diagram becomes
	\begin{equation*}
		\begin{tikzcd}[column sep=small]
			&&& {[ ({}_l H_1 \times {}_l H_2) \backslash (H_1 \times H_2 \times X)/ \Delta_{123} S]}
			\\
			{[H_1 \backslash X]} 
			& {[S \backslash X]} \ar[l, "\pi_1"']
			& {[\Delta H_1 \backslash (H_1 \times X) / S]} \ar[l, "\cong"'] \ar[d, "\cong"'] 
			& {[(\Delta_{13} H_1 \times {}_l H_2) \backslash (H_1 \times H_2 \times X) / \Delta_{12} S]} \ar[l, "\cong"'] \ar[d, "\cong"'] \ar[u, "\cong"]
			\\
			&& {[\Delta H_1 \backslash (H_1/S) \times X]}
			& {[( \Delta_{13} H_1 \times {}_l H_2) \backslash (H_1 \times H_2)/S \times X]} \ar[l, "\cong"']
		\end{tikzcd}
	\end{equation*}
	As a result, the pullback functor $\gamma_l^*$, or equivalently $q_X^*$, is isomorphic to
	\begin{equation*}
		\pi_1^*: 
		D^b(X, H_1) \aro 
		D^b(X,S)
		\cong D^b( (H_1 \times H_2)/S \times X, \Delta_{13} H_1 \times {}_l H_2).
	\end{equation*}
	Note that $\pi_1$ is the same as the map $\pi_{\psi,c[k]}$ in the definition of $\bI_\cE$. Hence we see that
	\begin{equation*}
		\gamma_l^* \cong \pi_{\psi,c[k]}^*.
	\end{equation*}
	
	A similar argument applies to the map $p|_{\ubar T} \circ \gamma_r$. By Proposition \ref{prop:ext-space}(3), it can be rewritten as the composition
	\begin{equation*}
		\widetilde{\BO_{\varphi,c[k]}}|_{\ubar S, \ubar T} \bijects
		\Ext_{\ubar S, 0}^{\ubar T} \times \BO_\varphi|_{\ubar T} \surjects
		\Gr_c(\ker \ubar T_k) \times \BO_\varphi|_{\ubar T} 
		\surjects \BO_\varphi|_{\ubar T}.
	\end{equation*}
	Note that by means of the isomorphism $\tau_\cE$ of Proposition \ref{prop:SE}(1), the space $X = \BO_\psi|_{\ubar S}$ can be identified with $\BO_\varphi|_{\ubar T}$. Hence in the new notations, the above map is the same as
	\begin{equation*}
		(H_1 \times H_2)/S \times X \xrightarrow{q_{H_1} = \text{quotient by } H_1} (H_2/S) \times X \xtwoheadrightarrow{pr_X} X.
	\end{equation*}
	Here $(H_1 \times H_2)/S \times X$ is equipped with the left multiplication of $H_1$ on $H_1$ and the diagonal action of $H_2$ on $H_2$ by left multiplication and on $X$ on the left, $(H_2/S) \times X$ is equipped with the diagonal $H_2$-action, and the last $X$ is equipped with the left $H_2$-action. By the same argument as above (swapping the roles of $H_1$ and $H_2$), after passing to quotient stacks, these maps become
	\begin{equation*}
		[ ({}_l H_1 \times {}_l H_2) \backslash (H_1 \times H_2 \times X)/ \Delta_{123} S]
		\cong [S \backslash X] \xrightarrow{\pi_2} [ H_2 \backslash X].
	\end{equation*}
	Hence we see that $(p|_{\ubar T})_* \circ (\gamma_r^*)\inv$, or equivalently $pr_{X*} \circ (q_{H_1}^*)\inv$ in the new notations, is isomorphic to $\pi_{2*} = \pi_{\varphi*}$. Therefore
	\begin{equation*}
		(p|_{\ubar T})_* \circ (\gamma_r^*)\inv  \cong \pi_{\varphi*}.
	\end{equation*}
	As a result,
	\begin{equation*}
		\Ind_{\ubar S}^{\ubar T} = (p|_{\ubar T})_* \circ (\gamma_r^*)\inv \circ \gamma_l^*
		\cong 
		\pi_{\varphi*} \circ \pi_{\psi,c[k]}^* = \bI_\cE
	\end{equation*}
	as desired.
\end{proof}

Let us put everything together. Fix a choice of $\cE = (\iota,\eps) \in \Ext_{\ubar S,0}^{\ubar T}$. Choose a basis for $V_{\varphi, \floor r}$ and for $V_{\varphi, \ceil r}$. We get a basis for $V_{\psi, \floor r}$ and $V_{\psi, \ceil r}$ via the isomorphisms $\eps_{\floor r}$ and $\eps_{\ceil r}$. Write $n = \varphi( \floor r)$. As in Corollary \ref{cor:zeta} we have maps
\begin{gather*}
	\zeta: [\Delta \GL_n \backslash (\cP_{\lambda_L} \times \cP_{\lambda_R})] \cong [P_{\varphi,r} \backslash \GL_n / P_{\varphi, l}] \cong [G_{\varphi, \ubar T} \backslash \BO_\varphi|_{\ubar T}] \cong [G_\varphi \backslash \BO_\varphi] \injects [G_\varphi \backslash E_\varphi]
	\\
	\zeta': [\Delta \GL_n \backslash (\cP_{\lambda_L'} \times \cP_{\lambda_R'})] \cong [P_{\psi,r} \backslash \GL_n / P_{\psi, l}] \cong [G_{\psi, \ubar S} \backslash \BO_\psi|_{\ubar S}] \cong [G_\psi \backslash \BO_\psi] \injects [G_\psi \backslash E_\psi]
\end{gather*}
As demonstrated in \ref{cls:general_ex}, the parabolics $P_{\varphi, r}$ and $P_{\psi, r}$ contain a common Borel, and $P_{\varphi, r} = P_{\psi, r}$. Hence the intersections
\begin{equation*}
	Q_l := P_{\varphi, l} = P_{\psi, l} \text{ and } Q_r := P_{\varphi, r} \cap P_{\psi, r}
\end{equation*}
are parabolic subgroups of $\GL_n$, and the functor $\bI_\cE \circ \operatorname{Infl}_1^{G_{c[k]}}$ is the same as the pullback-pushforward functor
\begin{equation*}
	\bI_\psi^\varphi: D^b(\GL_n, P_{\psi, r} \times P_{\psi,l}) \aro D^b(\GL_n, P_{\varphi, r} \times P_{\varphi,l})
\end{equation*}
along the projections
\begin{equation*}
	[P_{\psi, r} \backslash \GL_n / P_{\psi,l}] \longleftarrow [Q_r \backslash \GL_n / Q_l] \longrightarrow [P_{\varphi, r} \backslash \GL_n / P_{\varphi,l}].
\end{equation*}

We then have the following diagram
\begin{equation*}
	{
	\begin{tikzcd}[row sep=large, column sep=small, cramped]
		D^b(E_\varphi, G_\varphi) \ar[r, "\text{rest.}"] 
		& D^b(\BO_\varphi, G_\varphi) \ar[r, dash, "\cong"]
		& D^b(\BO_\varphi|_{\ubar T}, G_{\varphi, \ubar T}) \ar[r, dash, "\cong"]
		& D^b(\GL_n, P_{\varphi,r} \times P_{\varphi,l})
		\\
		D^b(E_\psi, G_\psi \times G_{c[k]}) \ar[r, "\text{rest.}"] \ar[u, "\LInd_{\psi, c[k]}^\varphi"']
		& D^b(\BO_\psi, G_\psi \times G_{c[k]}) \ar[r, dash, "\cong"] \ar[u, phantom, "\sharp" description]
		& D^b(\BO_\psi|_{\ubar S}, G_{\psi, \ubar S} \times G_{c[k]}) \ar[r, phantom, "\natural" description] \ar[u, bend left, "\Ind_{\ubar S}^{\ubar T}"] \ar[u, bend right, "\bI_\cE"'] \ar[u, phantom, "\star" description]
		& D^b(\GL_n, Q_r \times Q_l) \ar[u, "\text{pushforward}"]
		\\
		D^b(E_\psi, G_\psi) \ar[r, "\text{rest.}"] \ar[u, "\operatorname{Infl}_1^{G_{c[k]}}"'] \ar[uu, out=150, in=-150, distance=2cm, "\LInd_\psi^\varphi"]
		& D^b(\BO_\psi, G_\psi) \ar[r, dash, "\cong"] \ar[u, "\operatorname{Infl}_1^{G_{c[k]}}"]
		& D^b(\BO_\psi|_{\ubar S}, G_{\psi, \ubar S}) \ar[r, dash, "\cong"] \ar[u, "\operatorname{Infl}_1^{G_{c[k]}}"]
		& D^b(\GL_n, P_{\psi,r} \times P_{\psi,l}) \ar[u, "\text{pullback}"] \ar[uu, out=30, in=-30, distance=2cm, "\bI_\psi^\varphi"']
	\end{tikzcd}
	}
\end{equation*}
Here rest. denotes the restriction functor, and $\LInd_\psi^\varphi$ is defined to be the composition $\LInd_{\psi, c[k]}^\varphi \circ \operatorname{Infl}_1^{G_{c[k]}}$. The composition of the top and bottom rows are the pullback functors $\zeta^*$ and $(\zeta')^*$, respectively. The rectangle $\natural$ commutes by the above discussion; the rectangle $\sharp$ commutes by Lemma \ref{lem:Lind-on-fiber}; the two maps around the oval $\star$ agree by Theorem \ref{thm:compare-functors} above, and the first two squares in the bottom row clearly commutes. As a result, the whole diagram commutes.

As explained in \ref{cls:general_ex}, we may choose $\Lambda_\BR$ and $\Lambda_\BR'$ that satisfy (\ref{eqn:lambdaL-and-lambdaL'-a})-(\ref{eqn:lambdaL-and-lambdaL'-c}). Then $\bI_\psi^\varphi$ is the same as the functor $\bI_{\Lambda_\BR'}^{\Lambda_\BR}$ defined prior to Proposition \ref{prop:adjoint-of-translation-mixed}. Consequently, we have the following corollary.

\begin{corollary}\label{cor:compare-functors}
	Under the above setup, the following diagram commutes
	\begin{equation*}
		\begin{tikzcd}
			D^b(E_\varphi, G_\varphi) \ar[r, "\zeta^*"]
			& D^b(\GL_n, P_{\varphi,r} \times P_{\varphi,l}) \ar[r, dash, "\cong"]
			& D^b(\cP_{\lambda_L} \times \cP_{\lambda_R}, \Delta \GL_n)
			\\
			D^b(E_\psi, G_\psi) \ar[r, "(\zeta')^*"] \ar[u, "\LInd_{\psi}^\varphi"]
			&  D^b(\GL_n, P_{\psi,r} \times P_{\psi,l}) \ar[u, "\bI_\psi^\varphi"'] \ar[r, dash, "\cong"]
			& D^b(\cP_{\lambda_L'} \times \cP_{\lambda_R'}, \Delta \GL_n) \ar[u, "\bI_{\Lambda_\BR'}^{\Lambda_\BR}"']
		\end{tikzcd}.
	\end{equation*}
	Taking adjoint under the perfect pairings (\ref{eqn:LLC-GLn-pairing}) and (\ref{eqn:LLC-Qp-pairing}), we obtain a commutative diagram
	\begin{equation*}
		\begin{tikzcd}
			K\Rep_{\Lambda}(\GL_m(\BQ_p)) \ar[d, "{}^kD"']
			& K\Rep_{\Lambda_\BR}(\GL_n(\BC)) \ar[l]  \ar[d, "T_{\Lambda_\BR}^{\Lambda_\BR'}"]
			\\
			K\Rep_{\Lambda'}(\GL_{m-c}(\BQ_p)) 
			& K\Rep_{\Lambda_\BR'}(\GL_n(\BC)) \ar[l]
		\end{tikzcd}.
	\end{equation*}
	where $\Lambda$, $\Lambda'$ are infinitesimal characters corresponding to $\varphi$, $\psi$, respectively, and ${}^kD$ is ${}^k\sD$ postcomposed with the projection onto the summand $K\Rep_{\Lambda'}(\GL_{m-c}(\BQ_p))$ of $\cR$.
\end{corollary}

\subsection{The left case}\label{subsec:left-to-right}

At the end of this section, we explain how to reduce the ``left case'' $k < \floor r$ to the ``right case'' $k > \ceil r$.

Suppose we have weight functions $\psi_a: \BZ \to \BN$ so that $\psi_0 = \psi_1 + \psi_2$.
\begin{itemize}
	\item Set $\psi_a^*(i) := \psi_a(-i)$;
	\item Identify $V_{\psi_a^*}$ with $(V_{\psi_a})^*$, so that $V_{\varphi^*,-i} = (V_{\varphi,i})^*$ has degree $-i$; 
	
	\item For $T \in \End_\BC(V_{\psi_a})$, write $T^* \in \End_\BC(V_{\psi_a^*})$ for its adjoint. Then $T$ is homogeneous of degree 1 if and only if $T^*$ is so, that is, $T \in E_{\psi_a}$ if and only if $T^* \in E_{\psi_a^*}$;
	
	\item Equip $V_{\psi_a^*}$ and $E_{\psi_a^*}$ with the induced actions of $G_{\psi_a}$; 
%
\end{itemize}
Then for $\{a,b\} = \{1,2\}$, if $\big( V_{\psi_a} \to V_{\psi_0} \to V_{\psi_b}, T\big)$ is an element in $\widetilde{E_{\psi_0,\psi_a}}$ (see \ref{cls:Lind} for the notation), $\big(V_{\psi_b^*} \to V_{\psi_0^*} \to V_{\psi_a^*}, T^*\big)$ is an element of $\widetilde{E_{\psi_0^*,\psi_b^*}}$. Similarly, if $\big(W \subset V_{\psi_0}, T\big)$ is an element of $E_{\psi_0,\psi_a}$, $\big( (V_{\psi_0}/W)^* \subset V_{\psi_0}^* = V_{\psi_0^*}, T^*\big)$ is an element of $E_{\psi_0^*,\psi_b^*}$.

Suppose $k < \floor r$, $c \in \BZ_{>0}$ and $\varphi = c[k] + \psi$ so that both $\varphi$ and $\psi$ satisfy Assumption \ref{assump:r} (in particular $c \le \varphi(k) - \varphi(k-1)$). Then the correspondence diagram for Lusztig induction is
\begin{equation*}
	\begin{tikzcd}
		G_{c[k]} \times G_\psi \ar[d, phantom, "\racts" {rot90, description}] 
		& G_{c[k]} \times G_\psi \times G_\varphi \ar[d, phantom, "\racts" {rot90, description}] \ar[l, two heads] \ar[r, two heads]
		& G_\varphi \ar[d, phantom, "\racts" {rot90, description}] \ar[r, equal]
		& G_\varphi \ar[d, phantom, "\racts" {rot90, description}]\\
		E_{c[k]} \times E_\psi 
		& \widetilde{E_{\varphi,\psi}} \ar[l, two heads, "\beta"'] \ar[r, two heads, "\tilde p"]
		& E_{\varphi,\psi} \ar[r, two heads, "p"]
		& E_\varphi
	\end{tikzcd}.
\end{equation*}
Applying $\BC$-linear dual to the objects involved, we obtain isomorphisms
\begin{align*}
	E_\varphi &\bijects E_{\varphi^*}, & T &\mapsto T^*,\\
	E_\psi &\bijects E_{\psi^*}, & S &\mapsto S^*,\\
	E_{\varphi,\psi} &\bijects E_{\varphi^*, c[k]^*}, & 
	\big(W \subset V_\varphi, T\big) &\mapsto \big((V_\varphi/W)^* \subset V_{\varphi^*}, T^*\big)\\
	\widetilde{E_{\varphi,\psi}} &\bijects \widetilde{E_{\varphi^*, c[k]^*}}, & 
	\big(V_\psi \to V_\varphi \to V_{c[k]}, T\big) &\mapsto \big(V_{c[k]^*} \to V_{\varphi^*} \to V_{\psi^*}, T^*\big)
\end{align*}
These isomorphisms evidently preserve the full rank parts and are equivariant for the various group actions. They form a commutative diagram
\begin{equation*}
	\begin{tikzcd}
		E_{c[k]} \ar[r, phantom, "\times" description] \ar[dr, "\sim"' {near start}] 
		&[-3ex] E_\psi \ar[dl, crossing over, "\sim" {near start}]
		& \widetilde{E_{\varphi,\psi}} \ar[l, two heads] \ar[d, "\sim"] \ar[r, two heads]
		& E_{\varphi,\psi} \ar[d, "\sim"] \ar[r, two heads]
		& E_\varphi \ar[d, "\sim"]\\
		E_{\psi^*} \ar[r, phantom, "\times" description]
		& E_{c[k]^*} 
		& \widetilde{E_{\varphi^*,c[k]^*}} \ar[l, two heads] \ar[r, two heads]
		& E_{\varphi^*,c[k]^*} \ar[r, two heads]
		& E_{\varphi^*}
	\end{tikzcd}
\end{equation*}
where all maps are equivariant with respect to appropriate groups. As a result, Lusztig induction with respect to the top row is isomorphic to the induction with respect to the bottom row via the vertical isomorphisms. 

Finally, we have $\varphi^* = (c[k])^* + \psi^* =\psi^* + c[-k]$, both $\varphi^*$ and $\psi^*$ satisfy Assumption \ref{assump:r} at position $-r$, and we have $-k > -r$. So $\varphi^*$ and $\psi^*$ belong to the situation studied in previous subsections.

\appendix
\section{Translation functors and push-pull functors}\label{sec:trans-pushpull}

Since this appendix only concerns real groups, to ease notation we will use the lowercase letter $\lambda$ to denote a real infinitesimal character, rather than $\Lambda_\BR$. We also drop the subscript $(-)_\BC$ for the complex algebraic objects. For example, we write $G$, $\fg$, and $\fh$ for $G_\BC$, $\fg_\BC$, and $\fh_\BC$.

\subsection{Push-pull functors}\label{subsec:trans_vs_pushpull}

The aim of this subsection is to prove Proposition \ref{prop:adjoint-of-translation} (Proposition \ref{prop:sing_to_sing} below). This should be well-known to experts, but we could not find a proof in the literature.

\subsubsection{The setup}

Since Recall from \ref{cls:LLC/R} the local Langlands correspondence at a dominant regular infinitesimal character $\lambda$ is a perfect pairing
\begin{equation}\label{eqn:LLC/R}
	\langle -,- \rangle_\lambda: \bigoplus_{x \in H^1(\Gamma, \check G_\BC)} K \Rep_{\lambda}(\check  G_{\BR,x}) \times 
	K \Perv(\cY(\lambda), G_{\lambda})
	\aro \BZ,
\end{equation}
where $\cY({\lambda})$ is the product of the partial flag variety $\cP_{\lambda}$ with the set $\cI(\lambda) := \big\{ y \in G^\Gamma- G \mid y^2 = \exp(2\pi i {\lambda}) \big\}$, usually called the set of \textit{strong involutions} attached to $\lambda$. The set $\cI(\lambda)$ is the disjoint union of finitely many closed $G_{\lambda}$-orbits, and hence we have a decomposition
\begin{equation*}
	[G_{\lambda} \backslash \cY(\lambda)] = \bigsqcup_i [K_{\lambda,y_i} \backslash \cP_{\lambda}]
\end{equation*}
where the disjoint union is over a choice of representatives $y_i$ for the $G_{\lambda}$-orbits in $\cI(\lambda)$, and $K_{\lambda,y_i} = G_{\lambda}{}^{\Ad y_i}$. In particular,
\begin{equation*}
	\Perv(\cY(\lambda), G_{\lambda}) = \bigoplus_i \Perv(\cP_{\lambda}, K_{\lambda,y_i}).
\end{equation*}

Now let $\nu \in \fh^*$ be a dominant regular infinitesimal character. Let $\lambda, \mu \in \fh^*$ be dominant and possibly singular infinitesimal characters so that $\lambda, \mu$ and $\nu$ are in the same coset of the character lattice. Let $W_\lambda$, $W_\mu$ denote the stabilizers of $\lambda$ and $\mu$ in $W$. Assume that $\nu$ is regular and that $\lambda$ is at least as singular as $\mu$, i.e. $W_\lambda \supseteq W_\mu \supseteq W_\nu = 1$. We then have perfect pairings $\langle -, - \rangle_\lambda$, $\langle -, - \rangle_\mu$ and $\langle -, - \rangle_\nu$. There are natural maps 
\begin{equation*}
	\begin{tikzcd}[row sep=small, column sep=small]
		\cP_\nu \ar[dr, two heads, "p_\mu"] \ar[dd, two heads, "p_\lambda"'] \\
		& \cP_\mu \ar[dl, two heads, "\bar p_\lambda"]\\
		\cP_\lambda
	\end{tikzcd}
\end{equation*} 
and hence functors
\begin{equation*}
	\begin{tikzcd}
		\Rep_{[\nu]}(\check G_{\BR,x})
		\ar[dd, "T_\nu^\lambda", xshift=8pt] \ar[dr, "T_\nu^\mu", xshift=5pt, yshift=5pt]
		&& \Perv(\cP_\nu, K_{\nu,y_i})
		\ar[dd, "p_{\lambda*}", xshift=8pt] \ar[dr, "p_{\mu*}", xshift=5pt, yshift=5pt] \\
		&\Rep_{[\mu]}(\check G_{\BR,x})
		\ar[dl, "T_\mu^\lambda", xshift=5pt, yshift=-5pt] \ar[ul, "T_\mu^\nu"]
		&& \Perv(\cP_\mu,K_{\mu,i})
		\ar[dl, "\bar p_{\lambda*}", xshift=5pt, yshift=-5pt] \ar[ul, "p_\mu^*"]\\
		\Rep_{[\lambda]}(\check G_{\BR,x})
		\ar[uu, "T_\lambda^\nu"] \ar[ur, "T_\lambda^\mu"]
		&& \Perv(\cP_\lambda, K_{\lambda,i})
		\ar[uu, "p_\lambda^*"] \ar[ur, "\bar p_\lambda^*"]
	\end{tikzcd}
\end{equation*}
Here the subscript $(-)_{[\nu]}$ denotes generalized infinitesimal character $\nu$, similarly for $\mu$ and $\lambda$. Note that the Grothendieck groups of $\Rep_{[\nu]}(\check G_{\BR,x})$ and $\Rep_{\nu}(\check G_{\BR,x})$ are the same. Also, since $\lambda$, $\mu$ and $\nu$ have the same integrality, $K_{\lambda,i} = K_{\mu,i} = K_{\nu,i}$.

\begin{proposition}[Proposition \ref{prop:adjoint-of-translation}]\label{prop:sing_to_sing}
	Consider the Grothendieck groups over $\BC$ (i.e. we apply $- \otimes_\BZ \BC$ to the Grothendieck groups). Along the perfect pairings $\langle -, - \rangle_\lambda$ and $\langle -, - \rangle_\mu$, $T_\lambda^\mu$ (resp. $T_\mu^\lambda$) is adjoint to $\bar p_{\lambda*}$ (resp. $\bar p_\lambda^*$), i.e.
	\begin{equation*}
		\langle T_\lambda^\mu [M], [\cF] \rangle_\mu = \langle [M], \bar p_{\lambda*} [\cF] \rangle_\lambda \quad
		(\text{resp. } \langle T_\mu^\lambda [M], [\cF] \rangle_\lambda = \langle [M], \bar p_\lambda^* [\cF] \rangle_\mu).
	\end{equation*}
\end{proposition}

The statement about $T_\mu^\lambda$ is known: it is enough to check this on simple or standard objects, in which case this follows from \cite[16.6]{ABV}. The goal is to prove the statement about $T_\lambda^\mu$.

\subsubsection{Weyl group actions}\label{subsec: W-actions}

Write $W_{int}$ for the integral Weyl group of $\lambda$, $\mu$, and $\nu$, and write $\ell_{int}$ for the length function on $W_{int}$. Suppose $\nu$ is regular integral dominant. The two Grothendieck groups in (\ref{eqn:LLC/R}) admit actions of $W_{int}$: it acts on $K \Rep_\nu(\check G_{\BR,x})$ by coherent continuation, and on $K \Perv(\cP_\nu, K_{\nu,i}) = K \Perv(\cB_\nu, K_{\nu,i})$ by convolution (here $\cB_\nu = G_\nu/ B_\nu$ is the full flag variety of $G_\nu$, where $B_\nu = B \cap G_\nu$ is a Borel). Under the perfect pairing (\ref{eqn:LLC/R}), we have
\begin{equation}\label{eqn:W-action-adjoint}
	\langle w \cdot [M], [\cF] \rangle = \langle [M], (-1)^{\ell_{int}(w)} w \cdot [\cF] \rangle
\end{equation}
\cite[14.9.(b)]{Vogan:IC4}, \cite[17.16]{ABV}. Let us describe these actions more precisely. On the representation side, any representation $M$ determines a coherent family $\Theta_M: \nu + X^* \to K \Rep_\nu(\check G_{\BR,x})$ with $\Theta_M(\nu) = [M]$, where $X^* \subset \fh^*$ denotes the lattice of characters of a torus in $\check G$. Then $w \in W_{int}$ acts on $[M]$ by
\begin{equation*}
	w \cdot [M] = w \cdot \Theta_M(\nu) := \Theta_M( w\inv \nu)
\end{equation*}
\cite[7.2.28]{Vogan:book}. On the parameter side, 
we have a convolution product
\begin{equation}\label{eq: convolution}
	D^b(G_\nu/ B_\nu, K_{\nu,i}) \times D^b(G_\nu/ B_\nu, B_\nu) \aro D^b(G_\nu/ B_\nu, K_{\nu,i}),
\end{equation}
\begin{equation*}
	(\cF, \cG) \mapsto \cF \star \cG
\end{equation*}
see \cite[\S 7.2]{Achar:book}. For $w \in W_{int}$, let $\cM_w$ for the $!$-direct image to $G_\nu/B_\nu$ of the perverse sheaf $\underline{\BC}_{B_\nu w B_\nu/B_\nu}[\ell_{int}(w)]$ on the Schubert cell $B_\nu w B_\nu/B_\nu$. Then $w$ acts on any $[\cF] \in K D^b(G_\nu/ B_\nu, K_{\nu,i})$ by convolution with $\cM_{w\inv}$:
\begin{equation}\label{eq: convolve with std}
	w \cdot [\cF] := [\cF \star \cM_{w\inv}].
\end{equation}
Let $C_w$ denote the Kazhdan-Lusztig basis elements in the Hecke algebra $\cH_{int}$ of $W_{int}$ \cite{KL:Hecke}. Then the element $C_w|_{q=1}$ in the group algebra $\cH_{int}|_{q=1} = \BZ[W_{int}]$ acts by convolution with $\cL_{w\inv}$, where $\cL_{w\inv}$ is the unique simple $B_\nu$-equivariant perverse sheaf supported on $\overline{B_\nu w\inv B_\nu/B_\nu}$.

\begin{lemma}\label{lem:conv-with-Cw}
	Let $w_\lambda \in W_\lambda$ be the longest element. We have an equality
	\begin{equation*}
		(-1)^{\ell_{int}(w_\lambda)} p_\lambda^* p_{\lambda*} = - \star \cL_{w_\lambda} = C_{w_\lambda}|_{q=1} \cdot -
	\end{equation*}
	as operators on $K \Perv( \cB_\nu)$.
\end{lemma}

\begin{proof}
	The proof is almost identical to the case where $w_\lambda$ is a simple reflection \cite[Lemma 7.2.8]{Achar:book}, and we omit it.
\end{proof}

\subsubsection{Reduction to $\mu$ regular}

\begin{lemma}\label{lem:red-to-reg}
	Suppose Proposition \ref{prop:sing_to_sing} holds in the cases where $\mu$ is regular. Then it holds for general $\mu$.
\end{lemma}

\begin{proof}	
	We first need the following equations
	\begin{align*}
		p_{\lambda*} &= \bar p_{\lambda*} p_{\mu*},\\
		T_\lambda^\nu &= T_\mu^\nu T_\lambda^\mu.
	\end{align*}
	The first one is because $p_\lambda = \bar p_\lambda \circ p_\mu$. For the second one, recall first that $T_\nu^\lambda = T_\mu^\lambda T_\nu^\mu$, since 
	these are translations within the same closed Weyl chamber all going deeper into walls. 
	Then by using the adjunction between translation functors, we obtain
	\begin{align*}
		\Hom_{\check \fg}(-,T_\lambda^\nu-)
		&= \Hom_{\check \fg}(T_\nu^\lambda -,-)
		= \Hom_{\check \fg}( T_\mu^\lambda T_\nu^\mu -,-)\\
		&= \Hom_{\check \fg}(T_\nu^\mu -, T_\lambda^\mu -)
		= \Hom_{\check \fg}(-, T_\mu^\nu T_\lambda^\nu -)
	\end{align*}
	which proves $T_\lambda^\nu = T_\mu^\nu T_\lambda^\mu$.
	
	By assumption of the lemma, we have the following equalities
	\begin{align}
		\langle T_\lambda^\nu [M], [\cF] \rangle  &= \langle [M], p_{\lambda*} [\cF] \rangle \label{eqn:lambda-nu}\\
		\langle T_\mu^\nu [M], [\cF] \rangle &= \langle [M], p_{\mu*} [\cF] \rangle.
	\end{align}
	Therefore
	\begin{align*}
		\langle [M], \bar p_{\lambda*} p_{\mu*} \cF \rangle_\lambda 
		&= \langle [M], p_{\lambda*} [\cF] \rangle_\lambda \\
		&= \langle T_\lambda^\nu [M], [\cF] \rangle_\nu \\
		&= \langle T_\mu^\nu T_\lambda^\mu [M], [\cF] \rangle_\nu\\
		&= \langle T_\lambda^\mu [M], p_{\mu*} [\cF] \rangle_\mu.
	\end{align*}
	Hence we have
	\begin{equation}\label{eqn:lambda-mu}
		\langle T_\lambda^\mu [M], [\cF'] \rangle_\mu = \langle [M], \bar p_{\lambda*} [\cF'] \rangle_\lambda
	\end{equation} 
	provided that $[\cF'] \in K \Perv(\cP_\mu, K_{\mu,i})$ is in the span of objects of the form $[p_{\mu*} \cF]$, $\cF \in \Perv(\cP_\nu, K_{\nu,i})$. 
	
	\begin{claim}\label{clm:push-surj}
		Up to a constant, any costandard object in $K \Perv(\cP_\mu, K_{\mu,i})$ is of the form $[p_{\mu*} \cF]$
	\end{claim}
	
	Since our $K$-groups are now over $\BC$, these scaled costandard objects in $K \Perv(\cP_\mu, K_{\mu,i})$ form a basis, and hence (\ref{eqn:lambda-mu}) holds for any $[M]$ and $[\cF]$. This completes the proof of the lemma.
\end{proof}

\begin{proof}[Proof of Claim \ref{clm:push-surj}]	
	We give two proofs of the claim. 
	
	The first one goes as follows. The claim is equivalent to saying that $p_{\mu*}$ is surjective on the Grothendieck group. By assumption of lemma, $p_{\mu*}$ is adjoint to $T_\mu^\nu$, and hence $p_{\mu*}$ is surjective if and only if $T_\mu^\nu$ is injective. This is true if we restrict to the Grothendieck group of category $\cO$, see \cite[Lemma 2.5(1)]{Backelin:Koszul}. There exists an embedding of the Grothendieck group Harish-Chandra modules into a direct sum of several copies of Grothendieck groups of category $\cO$ \cite[Lemma 4.11]{BMSZI}. This embedding intertwines tensor products with finite dimensional representations and taking isotypic part of generalized infinitesimal characters. Hence $T_\mu^\nu$ is injective on Harish-Chandra modules as well.
	
	We now provide a geometric proof. To ease notations, we temporarily drop the subscript $(-)_\mu$ in various objects, and write $K_{\mu,i}$ simply as $K$. Any costandard object is of the form $i_{Q*} \tau[\dim Q]$ where $i_Q: Q \inj \cP$ is a $K$-orbit and $\tau$ is a $K$-equivariant local system on $Q$. Let $\tilde Q = p_\mu\inv(Q)$, and define maps
	\begin{equation*}
		\begin{tikzcd}
			\tilde Q \ar[d, "q_\mu"'] \ar[r, "i_{\tilde Q}"]
			& \cB \ar[d, "p_\mu"]\\
			Q \ar[r, "i_Q"] 
			& \cP
		\end{tikzcd}.
	\end{equation*}
	We would like to show that $q_{\mu*} q_\mu^* \tau = \tau \otimes_\BC H^{*}(P/ B)$ where $K$ acts trivially on $H^*(P/ B)$. Assuming this for the moment, then we have
	\begin{align*}
		(i_{Q*} \tau) \otimes_\BC H^{*}(P/ B)
		&= i_{Q*}(\tau \otimes_\BC H^{*}(P/ B))\\
		&= i_{Q*} q_{\mu*} q_\mu^* \tau = p_{\mu*} i_{\tilde Q*} (q_\mu^*\tau),
	\end{align*}
	and therefore in the Grothendieck group we have
	\begin{align*}
		n \cdot [ i_{Q*} \tau] = [p_{\mu*} i_{\tilde Q*}(q_\mu^* \tau)]
	\end{align*}
	where $n = \sum_{i \in \BZ} (-1)^i \dim H^i(P/ B) \neq 0$ (note that $n>0$ since $P/ B$ has no odd cohomology). This proves the claim.
	
	It remains to show $q_{\mu*} q_\mu^* \tau = \tau \otimes_\BC H^*(P/ B)$. The map $\tilde Q \to Q$ can be described as the map on induced spaces 
	\begin{equation*}
		q_\mu: K \times_{ K_z} p_\mu\inv(z) \aro K \times_{K_z} \{z\}
	\end{equation*}
	induced by $a_\mu: p_\mu\inv(z) \to \{z\}$, where $z \in Q$ is any point and $K_z$ is the stabilizer of $z$ in $K$. Hence, writing $\tau_z$ for the stalk of $\tau$ at $z$, we have $q_\mu^* \tau = K \times_{K_z} a_\mu^* \tau_z = K \times_{K_z} (\tau_z \otimes \underline{\BC}_{p_\mu\inv(z)})$. By smooth base change, $q_{\mu*} q_\mu^* \tau = K \times_{K_z} a_{\mu*} (\tau_z \otimes \underline{\BC}_{p_\mu\inv(z)}) = K \times_{K_z} (\tau_z \otimes a_{\mu*}\underline{\BC}_{p_\mu\inv(z)}) = K \times_{K_z} (\tau_z \otimes H^*(p_\mu\inv(z)))$. Finally, notice that the $K_z$-action on $H^*(p_\mu\inv(z))$ is trivial. Indeed, the $K_z$-action factors through the inclusion map $K_z = K \cap P_z \subset P_z$ where $P_z \subset G$ is the parabolic corresponding to the point $z$. Since $P_z$ is connected, there is no nontrivial $P_z$-equivariant perverse sheaf on a point. In particular, the action $P_z \acts H^*(p_\mu\inv(z))$ is trivial, whence so is the action of $K_z$. Consequently
	\begin{equation*}
		p_{\mu*} p_\mu^* \tau = \Big( K \times_{K_z} \tau_z \Big) \otimes H^*(p_\mu\inv(z)) = \tau \otimes H^*(p_\mu\inv(z)).
	\end{equation*}
	Since $p_\mu\inv(z) \cong P/ B$ as varieties, this is equal to $\tau \otimes_\BC H^*(P/ B)$, as desired.
\end{proof}

\subsubsection{Proof in the regular case}

We start the proof of Proposition \ref{prop:sing_to_sing}.
By Lemma \ref{lem:red-to-reg}, it suffices to consider the case where $\mu = \nu$ is regular. 

\begin{lemma}
	Proposition \ref{prop:sing_to_sing} holds for regular $\mu=\nu$ if the following equation holds
	\begin{equation}\label{eqn:trans-vs-pushpull}
		\langle T_\lambda^\nu T_\nu^\lambda [M], [\cF] \rangle_\nu
		= \langle [M], p_\lambda^* p_{\lambda*} [\cF] \rangle_\nu
	\end{equation}
	for any $[M]$ and $[\cF]$.
\end{lemma}

\begin{proof}
	From the assumption we have 
	\begin{align*}
		\langle T_\lambda^\nu T_\nu^\lambda [M], [\cF] \rangle
		&= \langle [M], p_\lambda^* p_{\lambda*} [\cF] \rangle\\
		&= \langle T_\nu^\lambda [M], p_{\lambda*} [\cF] \rangle
	\end{align*}
	where the second equation is because $T_\nu^\lambda$ and $p_\lambda^*$ are adjoint. Hence
	\begin{equation}\label{eqn:lambda-nu-b}
		\langle T_\lambda^\nu [M'], [\cF] \rangle = \langle [M'], p_{\lambda*} [\cF] \rangle
	\end{equation}
	whenever $[M']$ is in the span of classes of the form $[T_\nu^\lambda M]$. However, every simple $M'$ is of this form: every simple module at a singular infinitesimal character $\lambda$ is the translation of a unique simple module at the regular infinitesimal character $\nu$. Since simple objects form a basis for the Grothendieck group. Hence (\ref{eqn:lambda-nu-b}) holds for all $[M']$, which proves \ref{prop:sing_to_sing} for regular $\mu = \nu$.
\end{proof}

Now we want to write both $T_\lambda^\nu T_\nu^\lambda$ and $p_\lambda^* p_{\lambda*}$ in terms of the Weyl group actions described in \textsection \ref{subsec: W-actions}. Recall that $W_\lambda$ denotes the stabilizer of $\lambda$ in $W$, and $w_\lambda \in W_\lambda$ is the longest element. By Lemma \ref{lem:conv-with-Cw}, $p_\lambda^* p_{\lambda*}$ agrees with the action of $(-1)^{\ell_{int}(w_\lambda)} C_{w_\lambda}|_{q=1}$. 

\begin{lemma}
	We have an equality
	\begin{equation*}
		T_\lambda^\nu T_\nu^\lambda = \sum_{w \in W_\lambda} w
	\end{equation*}
	as operators on $K \Rep_\nu(\check G_{\BR,x})$.
\end{lemma}

\begin{proof}
	This is \cite[Proposition 7.3.2]{Vogan:book}.
\end{proof}

As a result, (\ref{eqn:trans-vs-pushpull}) holds if we have
\begin{equation*}
	\langle \sum_{w \in W_\lambda} w \cdot [M] , [\cF] \rangle = \langle [M], (-1)^{\ell_{int}(w_\lambda)} C_{w_\lambda}|_{q=1} \cdot [\cF] \rangle.
\end{equation*}
To establish this equation, recall that 
\begin{equation*}
	C_{w_\lambda}|_{q=1}  
	= \sum_{w \in W_\lambda} (-1)^{\ell_{int}(w_\lambda) -\ell(w)} w
\end{equation*}
or equivalently
\begin{equation*}
	(-1)^{\ell_{int}(w_\lambda)} C_{w_\lambda}|_{q=1} 
	= \sum_{w \in W_\lambda} (-1)^{\ell_{int}(w)} w
\end{equation*}
in the group algebra $\BZ[W_{int}]$ (this follows from, for example, the BGG resolution). Therefore, using (\ref{eqn:W-action-adjoint}),
\begin{align*}
	\langle \sum_{w \in W_\lambda} w \cdot [M] , [\cF] \rangle
	&= \sum_{w \in W_\lambda} \langle w \cdot [M], [\cF] \rangle\\
	&= \sum_{w \in W_\lambda} \langle [M], (-1)^{\ell_{int}(w)} w \cdot [\cF] \rangle\\
	&= \langle [M], \sum_{w \in W_\lambda} (-1)^{\ell_{int}(w)} w \cdot [\cF] \rangle\\
	&= \langle [M], (-1)^{\ell_{int}(w_\lambda)} C_{w_\lambda}|_{q=1} \cdot [\cF] \rangle,
\end{align*}
as required. This completes the proof of Proposition \ref{prop:sing_to_sing}.

\subsection{Factorization of translation functors}\label{sec:factor}

We now restrict our attention to $G_\BR = \GL_n(\BC)$, $G_\BC = \GL_n(\BC) \times \GL_n(\BC)$. We are interested in the factorization of certain translation functors applied to virtual representations of $G_\BR$. We employ the language of coherent families, following \cite[Chapter 7]{Vogan:book}.

We identify the Cartan $\fh_\BC$ of $\fg_\BC$ with $\BC^{2n}$, and write $e_1,\ldots, e_n, f_1,\ldots,f_n$ for the coordinate functions. We identify the Weyl group $W$ with the symmetric group $S_n \times S_n$ by letting $(w,1) \cdot e_i := e_{w(i)}$ and $(1,w) \cdot f_i := f_{w(i)}$. We write $W_\lambda$ for the stabilizer of $\lambda \in \fh_\BC^*$ in $W$, and write $X^* \subset \fh_\BC^*$ for the character lattice.

\begin{lemma}\label{lem:translation_by_coherent_family}
	Suppose $\lambda \in \fh_\BC^*$ is integral and
	\begin{equation*}
		\mu = \lambda - (e_{j_1} + \cdots + e_{j_r}), \quad 1 \le j_1 < \cdots < j_r \le n.
	\end{equation*}
	For any coherent family $\Theta$ over $X^*$,
	\begin{equation*}
		T^{\mu}_{\lambda} (\Theta(\lambda)) = \sum_{s \in W_{\lambda}/W_{\lambda} \cap W_{\mu}} \Theta(s(\mu)).
	\end{equation*}
\end{lemma}

\begin{proof}
	Let $V$ be 
	the irreducible representation of $G_\BC$ of extremal weight $\nu_0 := e_{j_1} + \cdots + e_{j_r}$, viewed as a representation of $G_\BR$ by restriction. 
	Note the set of weights $wt(V)$ in $V$ coincides with the Weyl group orbit $W \cdot \nu_0$. Hence all weights of $V$ have multiplicity one. By the definition of coherent family, we have
	\begin{equation*}
	V^{\vee} \otimes \Theta(\lambda) = \sum_{\nu' \in wt(V^\vee)}  \Theta(\lambda+ \nu') = \sum_{\nu \in wt(V)}  \Theta(\lambda - \nu).
	\end{equation*}
	Then 
	\begin{equation*}
	T^{\mu}_{\lambda} (\Theta(\lambda)) = \sum_{\substack{\nu \in wt(V): \\ \lambda - \nu \in W \cdot \mu }}  \Theta(\lambda - \nu).
	\end{equation*}
	It is not hard to see that for any $\nu \in wt(V)$, $\lambda - \nu \in W \cdot \mu$ if and only if $\nu \in W_{\lambda} \cdot \nu_{0}$. Moreover, if $s_1(\nu_0) = s_2(\nu_0)$ for $s_1, s_2 \in W_{\lambda}$, then $s_2^{-1}s_1 \in W_{\mu} \cap W_{\lambda}$. Therefore,
	\begin{equation*}
	T^{\mu}_{\lambda} (\Theta(\lambda)) = \sum_{s \in W_{\lambda}/W_{\lambda} \cap W_{\mu}} \Theta(\lambda - s(\nu_0)) = \sum_{s \in W_{\lambda}/W_{\lambda} \cap W_{\mu}} \Theta(s(\lambda - \nu_0)) = \sum_{s \in W_{\lambda}/W_{\lambda} \cap W_{\mu}} \Theta(s(\mu))
	\end{equation*}
	which is the equality we want.
\end{proof}

\begin{lemma} 
	Let $F_0$ be a facet of $\fh_\BC^{*}$ and $F_1, F_2$ be two facets in the closure of $F_0$. Suppose $\lambda_i \in F_i$ are integral. For any coherent family $\Theta$ over $X^*$, we have
	\begin{equation}\label{eqn:trans-srs}
		T^{\lambda_2}_{\lambda_0} \circ T^{\lambda_0}_{\lambda_1}(\Theta(\lambda_1))= |W_{\lambda_1} \cap W_{\lambda_2}/W_{\lambda_0}| \sum_{s \in W_{\lambda_1}/W_{\lambda_1} \cap W_{\lambda_2}} \Theta(s(\lambda_2))
	\end{equation}
\end{lemma}

\begin{proof}
	By a slight modification of \cite[Proposition 7.2.22(b)]{Vogan:book}, 
	\begin{equation*}
		T^{\lambda_0}_{\lambda_1}(\Theta(\lambda_1))
		= \sum_{s \in W_{\lambda_1}/W_{\lambda_0}} \Theta(s \lambda_0).
	\end{equation*}
	Moreover, by Lemma 7.3.1 and Proposition 7.2.22(a) of \textit{loc. cit.},
	\begin{equation*}
		T_{\lambda_0}^{\lambda_2}(\Theta(s \lambda_0))
		= T_{s\lambda_0}^{s\lambda_2}(\Theta( s \lambda_0)) 
		= \Theta(s \lambda_2).
	\end{equation*}
	Therefore
	\begin{equation*}
		T^{\lambda_2}_{\lambda_0} \circ T^{\lambda_0}_{\lambda_1}(\Theta(\lambda_1))
		= \sum_{s \in W_{\lambda_1}/W_{\lambda_0}} \Theta(s \lambda_2).
	\end{equation*}
	Since $\Theta(s \lambda_2) = \Theta(st\lambda_2)$ for any $t \in W_{\lambda_1} \cap W_{\lambda_2}$, the right hand side is equal to the right side of (\ref{eqn:trans-srs}). 
\end{proof}

\begin{proposition}\label{prop:2nd_decomposition}
	Suppose $\lambda = (\lambda_L, \lambda_R) \in \fh_\BC^*$ is integral dominant. Suppose $\lambda_{L,j} = \cdots = \lambda_{L,j+r-1} > \lambda_{L,j+r}$. Write 
	\begin{equation*}
		\lambda' = \lambda - (e_j + \cdots + e_{j+r-1}),\quad
		\lambda'' = \lambda + (e_1 + \cdots + e_{j-1}).
	\end{equation*}
	Then
	\begin{equation*}
		T^{\lambda'}_{\lambda} = T^{\lambda'}_{\lambda''} \circ T^{\lambda''}_{\lambda}.
	\end{equation*}
\end{proposition}

Note that $\lambda$, $\lambda'$ and $\lambda''$ are all dominant, and we have $W_{\lambda''} = W_\lambda \cap W_{\lambda'}$. 

\begin{proof}
	Let $\Theta$ be any coherent family over $X^*$. By Lemma~\ref{lem:translation_by_coherent_family},
	\begin{equation*}
		T^{\lambda'}_{\lambda} (\Theta(\lambda)) = \sum_{s \in W_{\lambda}/W_{\lambda} \cap W_{\lambda'}} \Theta(s(\lambda')).
	\end{equation*}
	Apply the preceeding lemma to $\lambda_0 = \lambda''$, $\lambda_1 = \lambda$ and $\lambda_2 = \lambda'$, we obtain
	\begin{equation*}
		\sum_{s \in W_{\lambda}/W_{\lambda} \cap W_{\lambda'}} \Theta(s(\lambda'))
		= \frac1{|W_\lambda \cap W_{\lambda'}/ W_{\lambda''}|} (T_{\lambda''}^{\lambda'} \circ T_{\lambda}^{\lambda''})(\Theta(\lambda)).
	\end{equation*}
	But by construction $W_{\lambda''} = W_\lambda \cap W_{\lambda'}$, so $|W_\lambda \cap W_{\lambda'}/ W_{\lambda''}| = 1$. This completes the proof.
\end{proof}


\printbibliography

\end{document}